\documentclass[12pt,reqno]{amsart}
\usepackage[headings]{fullpage}
\usepackage{amssymb,amsmath,amscd,tikz,mathtools}
\usepackage{graphicx}
\usepackage{texdraw}
\usepackage{pb-diagram}
\usepackage[all,cmtip]{xy}
\usepackage{url}
\usepackage[bookmarks=true,%
    colorlinks=true,%
    linkcolor=blue,%
    citecolor=blue,%
    filecolor=blue,%
    menucolor=blue,%
    urlcolor=blue,%
    breaklinks=true]{hyperref}
\usepackage{slashed}    
\usepackage{tikz-cd}
\usepackage{stmaryrd}
\usepackage{adjustbox}

\usepackage{verbatim}

\newcounter{theoremcount}
\newtheorem{theorem}[theoremcount]{Theorem}
\newtheorem{proposition}{Proposition}[section]
\newtheorem{corollary}[proposition]{Corollary}
\theoremstyle{definition}
\newtheorem{lemma}[proposition]{Lemma}
\newtheorem{definition}[proposition]{Definition}
\newtheorem{remark}[proposition]{Remark}

\newtheorem{example}[proposition]{Example}

\def\BN{\mathbb N}

\def\BZ{\mathbb Z}
\def\BQ{\mathbb Q}
\def\BR{\mathbb R}
\def\BC{\mathbb C}
\def\BP{\mathbb P}
\def\BF{\mathbb F}
\def\BK{\mathbb K}

\def\bM{\mathbf{M}}

\def\calT{\mathcal T}
\def\calP{\mathcal P}

\def\calB{\mathcal B}

\def\calH{\mathcal H}

\def\HH{\mathbf H}
\def\s{\sigma}

\def\SL{\mathrm{SL}}

\def\pt{\partial}

\def\Res#1{\underset{#1}{\mathrm{Res}}\,}

\def\a{\alpha}

\def\ve{\varepsilon}
\def\th{\theta}

\def\Li{\mathrm{Li}}
\def\PSL{\mathrm{PSL}}

\def\Im{\mathrm{Im}}

\def\be{\begin{equation}}
\def\ee{\end{equation}}

\def\z{\zeta}

\def\e{\bold e}

\def\diag{\mathrm{diag}}

\def\sma#1#2#3#4{\bigl(\smallmatrix#1&#2\\#3&#4\endsmallmatrix\bigr)} % small matrix
\usepackage{accents}

\def\calO{\mathcal{O}}

\def\={\;=\;}
\def\inn{\;\in\;}
\def\cm{\;\equiv\;}
\def\+{\,+\,}
\def\-{\,-\,}
\def\Tran{\mathrm{T}}

\def\sym{\mathrm{sym}}

\def\FGI{\mathrm{FGI}}

\def\fr{\mathrm{\varphi}_p}
\def\CS{\Omega^{\FGI}}

\def\Phitof{f}

\makeatletter

\renewcommand\thepart{\@Roman\c@part}%
\renewcommand\part{%
   \if@noskipsec \leavevmode \fi
   \par
   \addvspace{6.7ex}%
   \@afterindentfalse
   \secdef\@part\@spart}
\def\@part[#1]#2{%
    \ifnum \c@secnumdepth >\m@ne
      \refstepcounter{part}%
      \addcontentsline{toc}{part}{Part~\thepart.\ #1}%
    \else
      \addcontentsline{toc}{part}{#1}%
    \fi
    {\parindent \z@ \raggedright
     \interlinepenalty \@M
     \normalfont
     \ifnum \c@secnumdepth >\m@ne
       \centering\large\scshape \partname~\thepart.%
       \hspace{1ex}%
     \fi%
     \large\scshape #2%
     \markboth{}{}\par}%
    \nobreak
    \vskip 4.7ex
    \@afterheading}
  \def\@spart#1{
  \refstepcounter{part}%
  \addcontentsline{toc}{part}{#1}%
    {\parindent \z@ \raggedright
     \interlinepenalty \@M
     \normalfont
     \centering\large\scshape #1\par}%
     \nobreak
     \vskip 4.7ex
     \@afterheading}
\renewcommand*\l@part[2]{%
  \ifnum \c@tocdepth >-2\relax
    \addpenalty\@secpenalty
    \addvspace{0.75em \@plus\p@}%
    \begingroup
      \parindent \z@ \rightskip \@pnumwidth
      \parfillskip -\@pnumwidth
      {\leavevmode
       \normalsize \bfseries #1\hfil \hb@xt@\@pnumwidth{\hss #2}}\par
       \nobreak
       \if@compatibility
         \global\@nobreaktrue
         \everypar{\global\@nobreakfalse\everypar{}}%
      \fi
    \endgroup
  \fi}

\def\l@subsection{\@tocline{2}{0pt}{2pc}{6pc}{}}
\makeatother

\begin{document}
\title[The Habiro ring of a number field]{
       The Habiro ring of a number field}
\author{Stavros Garoufalidis}
\address{% Max Planck Institute for Mathematics \\
         % Vivatsgasse 7, 53111 Bonn, GERMANY \newline
  International Center for Mathematics, Department of Mathematics \\
  Southern University of Science and Technology \\
  Shenzhen, China \newline
  {\tt \url{http://people.mpim-bonn.mpg.de/stavros}}}
\email{stavros@mpim-bonn.mpg.de}
\author{Peter Scholze}
\address{Max Planck Institute for Mathematics \\
         Bonn, Germany \newline
         {\tt \url{http://people.mpim-bonn.mpg.de/scholze}}}
\email{scholze@mpim-bonn.mpg.de}
\author{Campbell Wheeler}
\address{Institut des Hautes \'Etudes Scientifiques \\
         Le Bois-Marie, Bures-sur-Yvette, France 
         \newline
         {\tt \url{https://www.ihes.fr/~wheeler}}}
\email{wheeler@ihes.fr}

\author{Don Zagier}
\address{Max Planck Institute for Mathematics \\
        Bonn, Germany
\newline
\indent $\mathrm{and}$ International Centre for Theoretical Physics, Trieste, Italy
\newline 
         {\tt \url{http://people.mpim-bonn.mpg.de/zagier}}}
\email{dbz@mpim-bonn.mpg.de}


\date{13 August 2025}%{3 December 2024}%\today }
%\dedicatory{}

\begin{abstract}
  We introduce the Habiro ring of a number field $\BK$ and modules over it graded by
  $K_3(\BK)$. Elements of these modules are collections of power series at each
  complex root of unity that arithmetically glue with each other after applying
  a Frobenius endomorphism, and after dividing at each prime by a collection of
  series that depends solely on an element of the Bloch group.

  The main theorems of this paper concern number fields, their algebraic
  $K$-theory and its regulator maps (Borel, $p$-adic and \'etale), whereas the explicit
  collections of series are defined by a careful algebraic analysis of the infinite
  Pochhammer symbol at roots of unity. 
  The origin of the above
  mentioned power series comes from perturbative Chern--Simons theory and by
  expansions of the admissible series of Kontsevich--Soibelman, both ultimately related
  to the infinite Pochhammer symbol.
  This link suggests that some
  Donaldson-Thomas invariants have arithmetic meaning and that some elements of
  the Habiro ring of a number field have enumerative meaning.
\end{abstract}

\maketitle

{\footnotesize
\tableofcontents
}

%%%%%%%%%%%%%%%%%%%%%%%%%%%%%%%%%%%%%%%%%%%%%%%%%%%%%%%%%%%%%%%%%%%%%%%%%%%% 
%%%%%%%%%%%%%%%%%%%%%%%%%%%%%%%%%%%%%%%%%%%%%%%%%%%%%%%%%%%%%%%%%%%%%%%%%%%%

\section{Introduction}
\label{sec.intro}

\subsection{Prologue}
\label{sub.prologue}

This paper introduces a Habiro ring of a number field $\BK$ and constructs modules
over this ring graded by the third algebraic $K$-theory group $K_3(\BK)$. Roughly
speaking, this ring consists of collections of power series around each root
of unity that are integral and $p$-adically glue, after a
Frobenius twist, whereas the modules consist of collection of power series around
each root of unity that are non-integral and only after a $p$-adic modification, they
$p$-adically glue. Due to the Frobenius twist, the nontriviality of the Habiro
ring of a number field and of explicit elements of the the corresponding modules is
not an easy consequence and in fact it forms our main results.

As such, the main theorems of the paper concern number fields, their algebraic
$K$-theory and its regulator maps (Borel, $p$-adic and \'etale), and the explicit
series alluded to above are defined by a careful algebraic analysis of the infinite
Pochhammer symbol at roots of unity. 

Although the theorems of our paper concern number fields, the origin of the above
mentioned power series comes from experiments in perturbative
Chern--Simons theory by two of the authors~\cite{GZ:kashaev,GZ:qseries}, and 
also by expansions of the admissible series of Kontsevich--Soibelman concerning
Donaldson--Thomas invariants of quivers~\cite{KS:cohomological}.

This indicates that our newly introduced rings, their modules and their elements
carry non-trivial information about number fields, and at the same
time are useful in other areas of mathematics and mathematical physics. 

\subsection{Integrality and congruence experiments}
\label{sub.intexp}

Early experiments by two of the authors on perturbative
invariants of Chern--Simons theory revealed unexpected integrality and gluing
properties of functions near roots of unity. The explanation of this phenomenon was
the original motivation for this work.
One of the experiments in question involved the power series $\Phi(h)$ that appear
in the asymptotic expansion of the Kashaev invariant of the $4_1$ knot, whose first
few terms are given (after multiplication by an overall eighth root of unity)
by~\cite[Eqn.(3)]{GZ:kashaev} 
\be
\label{as41}
\Phi(h)\=  \frac1{\sqrt[4]{-3}}\, 
\Bigl(1 + \frac{11}{24 \cdot 3 \sqrt{-3}}\,h + \frac{697}{1152 \cdot
  (3\sqrt{-3})^2}\,h^2
 + \frac{724351}{414720 \cdot (3 \sqrt{-3})^3}\,h^3 +\cdots\Bigr)\,. 
\ee

%% see : Exact-Check-Integrality-Twisted-Symmetrized-PhiSeries-at-1-41
%%       for 600 terms of the series

\noindent
At the time, 100 terms of the series $\Phi(h)$ were known, and the denominator of
the coefficient of $h^{100}$ was given by
\begin{small}
\begin{equation*}
  2^{397}\!\! \cdot 3^{298}\!\! \cdot 5^{40}\!\! \cdot 7^{22}\!\! \cdot 11^{12}\!\!
  \cdot 13^9 \!\cdot 17^6\!\cdot 19^5\!\cdot 23^4\!\cdot 29^3\!\cdot 31^3\!\cdot 37^2
  \!\cdot 41^2\!\cdot 43^2\!\cdot 47^2\!\cdot 53 \cdot 59 \cdot 61 \cdot 67 \cdot 71
  \cdot 73 \cdot 79 \cdot 83 \cdot 89 \cdot 97 \!\cdot\! 101
\end{equation*}
\end{small}

\noindent
in agreement with Theorem 9.1 of~\cite{GZ:kashaev}. 
However, the symmetrised series $\Phi(h)\Phi(-h)$, after a change of variable $q=e^h$
starts with 
\be
\label{Phi41sym}
\Phi(h)\Phi(-h) \= \frac1{\sqrt{-3}}\,
\big(1 - \frac{1}{3^3} (q-1)^2 + \frac{1}{3^3} (q-1)^3 - \frac{4}{3^5} (q-1)^4
- \frac{1}{3^5} (q-1)^5 + O((q-1)^6) \big)\,,
%% - \frac{25}{3^9} (q-1)^6 
%% + 214/6561 (q-1)^7 - 12842/531441 (q-1)^8 - 47749/531441 (q-1)^9
%% + 142030/1594323 (q-1)^10 + O((q-1)^11) \big)\,.
\ee
with all coefficients being integral away from $3$,
e.g. the denominator of the coefficient of $(q-1)^{100}$ being $3^{146}$.

But that's not all. The asymptotics of the Kashaev invariant at other complex roots
of unity $\z$ (other than $1$ discussed above) involve a power series $\Phi_\z(h)$
discussed in~\cite{GZ:kashaev}, whose constant term is in the ring $\BZ[\z_6,\z,1/3]$.
(These series were denoted by $\Phi_\a(h)$ in~\cite{GZ:kashaev}, where
$\z=e^{2 \pi i \a}$, with $\a \in \BQ$ and $\z_6$ a primitive 6th root of unity).
On the other hand the series~\eqref{Phi41sym}, which lies in
$\BZ[1/\sqrt{-3}][\![q-1]\!]$, can be evaluated when $q=\z_p$ is a primitive $p$-th
root of unity (with $p$ a prime $p \neq 3$), the result being a well-defined element
of $\BZ_p[1/\sqrt{-3},\z_p]$, because $\z_p-1$ is $p$-adically small. These two
numbers matched, up to sign, with the constant term
of $\Phi_{\z_p}(h) \Phi_{\z_p}(-h)$. The sign depended on the prime $p$ and
was equal to $1$ when $p=1 \bmod 3$ and $-1$ when $p=2 \bmod 3$. In more invariant
terms, the sign was given by the Legendre symbol $(\tfrac{p}{3})$. 

Both parts of the experiment were reminiscent of properties that elements of the
Habiro ring have, and this was a motivation for the current paper. It turned out that
the symmetrised series considered above are concrete elements of an abstractly defined
Habiro ring of a number field (in this case $\mathbb{Q}(\sqrt{-3})$) whose definition
was originally motivated by the question
whether $q$-de Rham cohomology admits a definition over the Habiro ring. In this
general definition, the unexpected Legendre symbol originates from the $p$-Frobenius
endomorphism of a $p$-completed ring.

This answer pushed the question further to find a home for the unsymmetrised
collections of power series at roots of unity, which after all is the output of
perturbative Chern--Simons theory. It turned out that such collections are elements
of rank one modules over the Habiro ring of the number field indexed by the third
algebraic $K$-group of the number field.

Quite unexpectedly, another source of elements of these modules comes from
the admissible series introduced and studied by
Kontsevich--Soibelman~\cite{KS:cohomological}. These are generating series of
cohomological Hall algebras that appear in enumerative algebraic geometry such
as Donaldson--Thomas theory.

These two sources of elements gave two complementary clues about the definition of
the rank one modules. Combining both views with a correction coming from $p$-adic
analysis gave the definition of the promised modules.

A corollary of our work is that the Donaldson-Thomas invariants have arithmetic
meaning and that the elements of the Habiro ring of a number field have enumerative
meaning, both being related to the comparison of the $K$-theory of local and global
fields and to the ideas of $q$-de Rham cohomology.

\subsection{The original Habiro ring}
\label{sub.habZ}
The Habiro ring, given by the remarkably short definition~\cite{Habiro:completion}
\be
\label{habdefZ}
\calH \;:=\; \varprojlim_n \;\BZ[q]/(q;q)_n\BZ[q]
\ee
(where $(q;q)_n=(1-q) \dots (1-q^n)$ is the $q$-Pochhammer symbol),
originated in Quantum Topology as a natural home for quantum invariants of knots
and 3-manifolds such as the Kashaev invariant of a knot~\cite{HuL} or the
Witten-Reshetikhin-Turaev
invariant of an integer homology 3-sphere~\cite{Habiro:WRT}. A detailed discussion
of this ring is given in the recent work~\cite{GL:skein}. 

Let us recall some basic properties of $\calH$ whose proofs can be found
in Habiro~\cite{Habiro:completion}. From its very definition, it follows that
it consists of elements of the form
\be
\label{eq:elts.hab}
f(q)\=\sum_{n=0}^\infty P_{n+1}(q) (q;q)_{n}, \qquad P_{n}(q) \inn \BZ[q]\,.
\ee
(The shift of index by $1$ will be very convenient in Section~\ref{sec.hab} and
reflect multiplicative properties of $n$.)
This makes it easy to construct elements of $\calH$. In fact, those quantum
invariants of knots and three-manifolds often arise as expressions of this form.
An example coming from the Kashaev invariant of the trefoil ($3_1$) knot is given by
~\cite[Eqn.(4.7)]{K94}
\be
\label{eq:KZ}
f_{3_1}(q) \= \sum_{n=0}^{\infty}(q;q)_{n}\inn\calH\,.
\ee
This particular element of $\calH$ was studied years later by Kontsevich and Zagier
(cf.~\cite{Za:strange}), without knowing its topological provenance, from the point
of view of its modularity properties (with respect to $\tau$, where $q=e^{2\pi i\tau}$).

A different property of the Habiro ring will play the central role in this paper.
Observe that if $q=\z_m$
is a primitive $m$-th root of unity, then $(q;q)_n=0$ for all $m >n$; thus an
element $f(q) \in \calH$ can be evaluated at $q=\z_m$ and the result is an element
of $\BZ[\z_m]$. More generally, when $q=\z_m+x$, we have
$(q;q)_n \in x^{\lfloor{n/m\rfloor}}\BZ[\z_m][x]$.
This implies that if $f(q)$ is an element of the Habiro ring, then
$f(\z_m+x) \in \BZ[\z_m][\![x]\!]$, giving rise to a ring homomorphism 
$\iota_m: \calH \to \BZ[\z_m][\![x]\!]$. Putting these homomorphisms together,
we obtain a map 
\be
\label{iota}
\iota: \calH \to \prod_{m\geq1}\BZ[\z_m][\![x]\!]\,,
\quad\text{such that}\quad
\iota(f) \= (f_m(x))_{m \geq 1}\,,
\qquad f_m(x)\;:=\;f(\z_m+x) \,,
\ee
where $\z_m$ for each $m$ is a fixed primitive $m$-th root of unity. We will
always suppose that the collection $(\z_m)$ is a compatible collection of roots
of unity of order $m$, that is a collection satisfying
\be
\label{zmdef}
\z_{mm'} \=\z_m\z_{m'}, \qquad (m,m')\=1, \qquad
(\z_{p^r})^p \=\z_{p^{r-1}}
\ee
for all positive integers $m$ and $m'$, primes $p$ and positive integers $r$.
For example, one can think of $\z_m=\exp(2\pi i\sum_{p\text{ prime}}p^{-v_{p}(m)})$,
with $v_p(m)$ being the $p$-adic valuation of $m$. 
We note that our choice of compatible collection differs from the traditional
choice $\omega_m=e^{2\pi i/m}$, which does not satisfy~\eqref{zmdef}, but instead
satisfies $\omega_{mm'}^m=\omega_{m'}$. However, our choice satisfies the important
property that $v_p(\z_{pm}-\z_m)>0$ for all positive integers $m$ and all
primes~$p$, and this will be more convenient for us.

For a ring $R$, the map $\iota$ motivates the notion of Galois invariant
$R$-valued ``functions near roots of unity''. By this, we mean elements of the ring 
\be
\label{PR}
\calP_{R}
\=
\Big(\prod_{\z\in\mu_{\infty}} R[\z]
\llbracket x\rrbracket\Big)^{\mathrm{Gal}(\BQ^{\mathrm{ab}}/\BQ)}
\;\cong\; \prod_{m\geq 1} R[\z_m]\llbracket x\rrbracket\,.
\ee
(Here the $f_m$ suffice to expand $F(q)$
around any root of unity, and not just around our distinguished collection of
roots of unity, using Galois invariance.)
We will denote a typical element of $\calP_R$ by $f(q)=(f_m(x))$, where $q=\z_m+x$,
rather than the more cumbersome $(f_m(x))_{m \geq 1}$.
We see that equation~\eqref{iota} now
gives a map $\iota:\calH\to\calP_\BZ$. Note that $f_m(x)$ are formal series with
no convergence properties assumed. In fact, in examples where $R\subseteq\BC$
coming from perturbative knot invariants, the associated series
will be factorially divergent series over the complex numbers, whereas at the same
time they are always $p$-adically convergent on a disc of radius $p^{-p/((p-1)(p-2))}$,
as follows from~\cite[Thm.9.1, Thm.9.2]{GZ:kashaev}.
The series $f_m(x)$ are related to ``functions near $\BQ$'' defined in the
previous work~\cite{GZ:kashaev} and more specifically the series $\Phi_{\alpha}(h)$,
where $\z_m+x=e^{2\pi i\alpha-h/m}$ and $\alpha \in \BQ$.

In~\cite{Habiro:completion} Habiro proved that the map~\eqref{iota} (and in fact,
each of its factors $\iota_m:\calH_\BZ\rightarrow\BZ[\z_m][\![x]\!]$)
is injective and that its image consists of the collection of power series
$f_m(x)=(\iota_m f)(x)$ that arithmetically ``glue'', generalising earlier
work of Ohtsuki~\cite{Ohtsuki:poly}. Let us explain how gluing works. 
The map~\eqref{iota} satisfies the formal substitution property
\be
\label{glue0}
f_m(x + \z_{pm} - \z_m) \= f_{pm}(x) 
\ee
for all positive integers $m$ and all prime numbers $p$. The only issue is that this
equation requires one to re-expand the left hand side, which is a power series in $x$,
after shifting $x$ to $x + \z_{pm} -\z_m$. This is possible using the binomial
theorem on $x$ and $\z_{pm}-\z_m$ if the shift
$\z_{pm} -\z_m$ is \emph{small}, and indeed it is in the completion
$\BZ[\z_{pm}]^\wedge_p$ of $\BZ[\z_{pm}]$, since $\z_{pm} -\z_m$ has positive
$p$-valuation. 
Here $R^{\wedge}_p=\varprojlim_n R/(p^n)$, which is isomorphic to $R\otimes \BZ_p$ and
will often be denoted simply by $R_p$, denotes the $p$-completion
of a ring $R$. Then, equation~\eqref{glue0} holds as an identity in the ring
$\BZ[\z_{pm}]^\wedge_p[\![x]\!]$, and this compatibility property is what we meant
by ``gluing''.

Summarising, an element $f(q)$ of the Habiro ring can be identified with a
collection of power series $(f_m(x))_{m \geq 1} \in \prod_{m \geq 1} \BZ[\z_m][\![x]\!]$
that satisfy the gluing property~\eqref{glue0} for all primes~$p$.

The gluing property implies
that the entire collection $(f_m(x))_{m \geq 1}$ is uniquely determined by any one
of its coordinate functions $f_m(x)$, indeed in many different ways, which of
course are consistent with each other. For instance $f_6(x)$ can be determined from
$f_1(x)$ via the chain $f_1(x) \to f_3(x)=f_1(x+\z_3-1) \to f_6(x)=f_3(x+\z_6-\z_3)$ 
or via the chain $f_1(x) \to f_2(x)=f_1(x+\z_2-1) \to f_6(x)=f_2(x+\z_6-\z_2)$:
\be
\begin{tiny}
\begin{aligned}
\begin{tikzpicture}
\draw (0,1) node (p1) {$f_1(x)$};
\draw (-3,0) node (p2) {$f_3(x) = f_1(x+\z_3-1)$};
\draw (3,0) node (p3) {$f_1(x+\z_2-1) = f_2(x)$};
\draw (0,-1) node (p4) {$f_3(x+\z_6-\z_3) = f_6(x) = f_2(x+\z_6-\z_2)$};
\draw[->] (p1) -- (p2);
\draw[->] (p1) -- (p3);
\draw[->] (p2) -- (p4);
\draw[->] (p3) -- (p4);
\end{tikzpicture}
\end{aligned}\,.
\end{tiny}
\ee
% \adjustbox{scale=0.7,center}{
% \begin{tikzcd}%[scale=0.7em]
%   &f_{1}(x)\arrow[ddr]&\\
%   & &\\
%   f_{3}(x) = f_1(x+1-\z_3)\arrow[uur,<-] & & f_{2}(x)=f_1(x+1-\z_2)\arrow[ddl] \\
%   & &\\
%   &f_6(x)=f_3(x+\z_3-\z_6)=f_2(x+\z_2-\z_6)\arrow[uul,<-]
% \end{tikzcd}
% }
% From now on, we will identify $\calH$ with $\iota(\calH)$.
Notice, however, that the equalities here hold in different completions.
We can phrase this gluing using the
% There is an alternative way to phrase gluing using elementary
language of
$p$-adic power series and their corresponding functions.  
Namely, an element of the Habiro ring consists of a collection of series
$f_m(x) \in \BZ[\z_m][\![x]\!]$ that satisfy the following property. For every
prime $p$, $f_m(x)$ is convergent on the open disc $|x|_p < 1$ that contains
$\z_{pm} -\z_m$, and its re-expansion $f_m(x+\z_{pm} -\z_{m})$, which is therefore
defined but a~priori has potentially transcendental coefficients in $\BZ_p[\z_{pm}]$,
has algebraic coefficients (in fact, in $\BZ[\z_{pm}]$) and agrees with $f_{pm}(x)$.
In other words, for each fixed prime $p$, the
collection of functions $(f_m(x))$ can be used to define a single $p$-adic
analytic function on the \emph{disconnected} domain 
$\{q\in\BC_p\;|\;|q|_p=1\}$, which is the union of the (not all distinct) open unit
discs centred at roots of unity.
(Of course this set is totally disconnected anyway in the $p$-adic topology,
but here we mean that no two of the discs in question are $p$-adically close, so that
the analytic functions on them cannot
be continued
from one to the other.)
The remarkable fact is that, given these series come from
something global, this ``disconnectedness'' disappears once
we start varying the prime $p$, in the sense that the small
neighbourhoods of any two roots of unity can be connected by neighbourhoods
overlapping with respect to varying $p$-adic metrics.
It is this property that makes the Habiro ring work.

\subsection{The Habiro ring of a number field}
\label{sub.hab}

An abstract definition of the Habiro ring of a number field was motivated by
the ideas of $q$-de Rham cohomology in~\cite{Scholze:canonical}. Our first task
is to give a power series realisation of this ring.
As in the case of the Habiro ring, the new definition involves $\calO_{\BK}$-valued
functions near the roots of unity, where $\calO_{\BK}$ is the ring of integers of a
number field $\BK$. However, there are two new features that emerge,
both of which already appeared in the numerical experiments described in
Section~\ref{sub.intexp}, namely:
\begin{itemize}
\item
  one has to invert the primes dividing some integer $\Delta$,
\item
  the gluing now involves a Frobenius twist.
\end{itemize}

Both features cause subtleties, as we will see shortly. If $\Delta \neq 1$,
the Habiro ring of $\calO_\BK[1/\Delta]$ is no longer an integral domain, while the
twist makes it
unclear how to construct elements that glue in the required way.
\footnote{We remark that both problems disappear when $\BK$ is abelian over $\BQ$.
Indeed, in this case we can define an $R$-module
structure on $\calH_R$ from the embedding $R\to\calH_R$ given by
$a\mapsto(\varphi_{m}a)_m$,
where $\varphi_m=\prod_p\fr^{v_p(m)}$ is the product of lifts from $\BK_p$ to $\BK$
of the Frobenius automorphisms $\fr$\thinspace,
which exists for fields $\BK$ abelian over $\BQ$.
This embedding also gives a canonical isomorphism
$\calH_R\cong\calH_{\BZ[1/\Delta]}\otimes_{\BZ} R$ for such~$\BK$.}

We now give the promised definition. Let $\BK$ be a number field with ring of integer
$\calO_{\BK}$ and $\Delta$ a positive integer divisible by the discriminant of
$\BK$.\footnote{One could also let $\Delta$ be a non-zero element of $\calO_{\BK}$
  divisible by all prime ideals whose square divides the discriminant, but for
  simplicity we will always choose $\Delta$ in $\BZ$.}
Below, we usually take $\Delta$ to also be divisible by $6$.
The Habiro ring\footnote{ 
Here, we use $\calH_{R}$ rather than the more accurate
$\calH_{\BZ[1/\Delta] \to \calO_{\BK}[1/\Delta]}$. We will also refer to this ring
as the Habiro ring of the field $\BK$, rather than of the ring $R$.} $\calH_{R}$
associated to the finite \'etale map
\be
\label{Fetale}
\BZ[1/\Delta] \to R, \qquad R \= \calO_{\BK}[1/\Delta]
\ee
is defined as follows:

\begin{definition}
\label{def.hab}
The Habiro ring $\calH_{R}$ is the subset of $\calP_R$ consisting of elements
\be
\label{habf}
f(q)\=(f_m(x))_{m\in\BZ_{>0}} \inn \calP_R
\ee
that satisfy the gluing property
\be
\label{gluef}
f_m(x+\z_{pm}-\z_m) \= (\fr f_{pm})(x)
\inn R^\wedge_p[\z_{pm}]\llbracket x \rrbracket
\ee
for primes $p$ and all positive integers $m$, where $\fr$ is the Frobenius
endomorphism of $R_p^{\wedge}$, lifted to an endomorphism of
$R^\wedge_p[\z_{pm}]\llbracket x \rrbracket$ fixing both $\z_{pm}$ and $x$.
\end{definition}

\noindent Note that if $p$ divides $\Delta$,
the above equation is trivial since the completed ring is trivial.
We also define $\calH_{R}|_{\Delta}$ to be the same ring where we restrict to
$m$ prime to $\Delta$,
and more generally $\calH_{R}|_{\gamma}$ to be the ring where we restrict to
$m$ prime to any positive integer~$\gamma$.

Note that the two equations~\eqref{glue0} and~\eqref{gluef} are nearly identical,
the crucial difference being the presence of $\fr$ in~\eqref{gluef}.
The reason this was not seen in the previous case was that $\fr=\mathrm{id}$
for $R=\BZ[1/\Delta]$.
The Frobenius twist makes it unclear if there is any natural $R$-module
structure on $\calH_{R}$, and moreover how to construct nontrivial elements in
$\calH_{R}$.
However, one can equivalently define $\calH_R$ as a finite \'etale
$\calH_{\BZ[1/\Delta]}$-algebra by gluing the naive base changes
$R[\zeta_m][\![q-\zeta_m]\!]$ over $\BZ[\zeta_m][\![q-\zeta_m]\!]$ along Frobenius
twists after $p$-adic completion. This abstract definition of the Habiro ring
implies in particular that $\calH_{R}$ is a finitely generated projective
$\calH_{\BZ[1/\Delta]}$-module of rank $r=[\BK:\BQ]$.
(This follows from the corresponding fact after $p$-completion, which itself follows
from the fact that $R_p[\z_{m}]\llbracket x \rrbracket$ has rank $r$ over
$\BZ_p[\z_{m}]\llbracket x \rrbracket$.)

There is another way to view these conditions involving $p$-completions.
Specifically, we have natural ring isomorphisms
\be
\label{eq.habp}
(\calH_R)_p \cong \Big\{
\begin{matrix}\text{collections of $R^\wedge_p[\z_m][\![x]\!]$-series for
$m \geq 1$ that}\\
\text{glue w.r.t. re-expansion in $x\mapsto x+\z_{mp}-\z_m$} \end{matrix} \Big\} \cong
\prod_{\underset{(m,p)=1}{m \geq 1}} R^\wedge_p [\z_{m}][\![x]\!]
\cong \calH_{R_p} \,, 
\ee
because there are no gluing constraints between the expansions
around different roots of unity of order prime to $p$. 
Moreover, for every prime $p$, we have a commutative diagram
\be
\label{Hviapcomp}
\begin{array}{ccc}
  \calH_R & \longrightarrow & \calH_{R^\wedge_p}\\
  \cap &  & \cap\\
  \calP_R & \overset{\varphi}{\longrightarrow} & \calP_{R^\wedge_p}
\end{array},
\ee
where the bottom map is the composite of the obvious inclusion with the Frobenius
endomorphism $\fr^{v_p(m)}$ acting on $R^\wedge_p [\z_{m}][\![x]\!]$,
and the Habiro ring is the universal ring associated to the combinations over $p$
of these diagrams, i.e.,
if we combine all of the $p$-completions $R_p^{\wedge}$ into
$\widehat{R}=\varprojlim_{N}R/MR\cong\prod_{p\;\text{prime}}R_p^{\wedge}$, then 
$\calH_{\widehat{R}}\cong\prod_{p}\calH_{R_p^\wedge}$ and $\calH_R$ is the intersection
in $\calP_{\widehat{R}}$ of $\calP_R$ and
$\calH_{\widehat{R}}\cong\calH_{\BZ[1/\Delta]}\otimes_{\BZ}\widehat{R}$.

\begin{remark}
\label{rem.domain}
Note that the ring $\calH_R$ is not an integral domain, since the gluing
property~\eqref{gluef} is nontrivial only for primes $p$ not dividing $\Delta$,
otherwise $R^\wedge_p=0$. Instead, $\calH_R$ is a product of (in general infinitely
many) integral domains indexed
by the equivalence classes of $\BN$ under the equivalence relation generated by
$m \sim_\Delta pm$
for primes $p$ not dividing $\Delta$. Among those
integral domains is $\calH_{R}|_\Delta$ (corresponding to the equivalence class of
all positive integers prime to $\Delta$), whose elements are collections of
power series at roots of unity of order prime to $\Delta$ that satisfy the
gluing condition~\eqref{gluef} for all primes $p$ not dividing $\Delta$.
It is still possible that there is a definition of $\calH_R$ when $R$ is equal to
$\calO_\BK$ rather than to $\calO_\BK[1/\Delta]$ that is an integral domain,
although the definition could not be exactly along the lines of the definition
for $\calO_\BK[1/\Delta]$.
\end{remark}

\subsection{\texorpdfstring{Modules over the Habiro ring of $\BK$ indexed
    by $K_3(\BK)$}{Modules over the Habiro ring of K indexed by K\_3(K)}}
\label{sub.habmod}

In this section we define a collection of modules $\calH_{R,\xi}$ over the Habiro ring
$\calH_{R}$ (with $R$ given in~\eqref{Fetale}) labelled by elements of
$\xi\in K_3(\BK)$. (Philosophically, the modules ought to be labelled by $K_3(R)$.
However, given the isomorphism $K_{2r-1}(R) \simeq K_{2r-1}(\BK)$ for $r>1$
(see e.g., Sec.5.2 of Weibel~\cite{Weibel} as well as~\cite[Eqn.(33)]{CGZ}), we
choose to label our modules by $K_3(\BK)$.)
These modules turn out to be the home of perturbative
quantum knot invariants, as we will see in Section~\ref{sub.results}.

Before explicitly defining these modules,
we recall the original motivation for their existence.
In Section~\ref{sub.intexp}, we recalled from~\cite{GZ:kashaev} a series~$\Phi(h)$
coming from the figure eight knot~$4_1$.
Its properties become much better when one passes to a ``completion'' given by
$\widehat{\Phi}(h)=\exp(i\mathrm{Vol}(4_1)/h)\,\Phi(h)$, where
$\mathrm{Vol}(4_1)$ is given by $2\thinspace\Im(\Li_{2}(e^{\pi i/3}))=2.0299\cdots$.
This dilogarithm is actually equal to a particular value of a regulator, which is a
function from $K_3(\BQ(\sqrt{-3}))$ to $\BC/(2\pi i)^2\BZ$.
This already indicates a link between this series and algebraic $K$-theory of the
field $\BQ(\sqrt{-3})$,
but work of Calegari and two of the authors~\cite{CGZ} showed that this link is
even stronger.
The asymptotics of the same Kashaev invariant at different roots of unity gives
similar series, but now the constant term at $q=\z_m$ (at least if $3|m$) contains
the $m$-th root of a unit in $\BQ(\sqrt{-3},\z_m)$ which up to $m$-th powers depends
only on the class $\xi\in K_3(\BQ(\sqrt{-3})$ giving the volume.
We want our module $\calH_{R,\xi}$ to contain an element whose expansion at
$q=e^{-h}$ is $\Phi(h)$.
In this section, we will give such a definition. Its elements have power series
expansions at roots of unity with subtle integrality properties that the series
$\Phi(h)$ (and all the other series $\Phi_\z(h)$) indeed satisfy.
To achieve this, since our point of view is $p$-adic for varying $p$, we will have
to replace the completion $\widehat{\Phi}(h)$ by a $p$-adic version in which
the usual dilogarithm is replaced with a $p$-adic dilogarithm. This motivates the
Definition~\ref{def.HRxi} given below.

The elements of $\calH_{R,\xi}$ will give collections of power
series at roots of unity whose constant terms contain the
$m$-th root of the above mentioned unit $\ve_m(\xi)$ defined
\hbox{in \cite[Thm.1.5]{CGZ}}. This unit comes from the Chern class map $c_{\z_m}$
of algebraic $K$-theory
\be
\label{emdef}
\ve_m\=c_{\z_m}^2 : K_3(\BK) \to \BK(\z_m)^\times/(\BK(\z_m)^\times)^m
\ee
and satisfies the property that
$(\s_\gamma \ve_m(\xi))^\gamma/\ve_m(\gamma)$ is \emph{canonically} an
$m$-th power for every integer $\gamma$ prime to $m$, where $\s_\gamma$ is the
automorphism sending $\z_m$ to $\z_m^\gamma$.
In~\cite{CGZ} this property is called $\chi^{-1}$-equivariance.
In more invariant terms, this is a
map from the set $K_3(\BK)$ towards the groupoid of $\mu_m$-torsors over $\BK(\z_m)$
(with this $\chi^{-1}$-equivariance datum); in fact, these~\hbox{$\mu_m$-torsors}
extend (necessarily uniquely) to $\calO_{\BK}[\zeta_m]$. Indeed, the mod $m$ \'etale
Chern character gives a canonical map
\be
\begin{small}
\begin{aligned}
\label{cm}
c_{\z_m}: K_3(\BK)\cong K_3(\calO_{\BK})\to H^1(\calO_{\BK},\BZ/m(2))
\to H^1(\calO_{\BK}[\zeta_m],\BZ/m(2))^{(\BZ/m\BZ)^\times}
\cong H^1(\calO_{\BK}[\zeta_m],\mu_m)^{\chi^{-1}}\!\!\! ,
\end{aligned}
\end{small}
\ee
where the last isomorphism is induced by cup product with
$\zeta_m\in H^0(\calO_{\BK}[\zeta_m],\mu_m)^{\chi}$. In particular, the units
$\ve_m(\xi)$ can be chosen to be integral at any given finite set of places, and we
will always implicitly fix such a choice at the relevant places.
We remind the reader that $K_3(\BK)$, at least after tensoring with $\BQ$, is
isomorphic to the Bloch group $B(\BK)$ and that in this interpretation $\xi$ is
represented as a formal linear combination $\sum_i[x_i]$, with
$x_i\in\BK^{\times}\smallsetminus\{1\}$, satisfying
$\sum_{i}x_i\wedge(1-x_i)=0\in\wedge^{2}\,\BK^{\times}$.
We will use both points of views in this paper.

% \red{(For the benefit of readers who are not specialists in algebraic geometry
% $H^1(\calO_{\BK}[\zeta_m],\mu_m)$
% is an abbreviation of the group cohomology $H^1(\mathrm{Gal}(\BK[\z_m]/\BQ),\mu_m)$
% IS THIS TRUE? I SUSPECT THAT IT SHOULD BE
% $H^1_{et}(\mathrm{Spec}(\calO_{\BK}[\zeta_m]),\mu_m)$.
% ARE THEY THE SAME? FOR A FIELD I WOULD GUESS SO BUT HERE DO THEY AGREE,
% $\mathbb{G}_m$-torsors
% are equivalent to line bundles and free rank-one modules are
% trivial line bundles. At the risk of stating the obvious to experts,
% we think of $f(q)$ as section of line bundles---even for
% elements of $\calH_R$. The modules $\calH_{R,\xi}$ are will be related to line
% bundles.)}

To give a first impression of the modules $\calH_{R,\xi}$ over $\calH_R$, we note
that their base change to $\BK[\z_m][\![x]\!]$ is induced by the $\mu_m$-torsor
described above (along the map from $\mu_m$-torsors to $\mathbb G_m$-torsors,
i.e.~line bundles),
noting that the $\mu_m$-torsor deforms uniquely to the power series ring. In concrete
terms, this is the collection of power series
\be
f_{m}(x)\inn\varepsilon_{m}(\xi)^{1/m}\,\BK[\z_m][\![x]\!]
\ee
whose leading term contains the $m$-th root $\ve_m(\xi)^{1/m}$ as a factor.

This determines the line bundle (described as usual by its collection of power
series) over~$R\otimes_{\BZ}\BQ=\BK$. In order to descend to $R$, we need to have
a $p$-adic description. For this we use the $p$-adic dilogarithm map
\be
\label{Dpdef2}
D_p:K_3(\BK_p) \to \BK_p\,,\qquad\BK_{p} \= R^{\wedge}_p[\tfrac{1}{p}]
\= \BK\otimes_{\BQ} \BQ_p
\ee
of Coleman~\cite{Coleman}, which coincides with the $p$-adic regulator map
as shown in Besser--de Jeu~\cite[Thm.1.6(2)]{BJ:syntomic} and is discussed in detail
in Section~\ref{sub.Dp}. With these maps we can define $\calH_{R^\wedge_p,\xi}$
via globally invertible group-like sections,
which can be thought of as an upgrading to power series of the collection of units
defined in~\cite{CGZ}.

\begin{definition}
\label{def.HRxi}
Fix $\xi \in K_3(\BK)$ and a prime $p$. An invertible $L_p(\xi)$-section is a
collection
\be
\label{fcyclic.loc}
f(q) \=(f_{m}(x))_{m\geq1,(m,p)=1}, \qquad
f_{m}(x)\inn \varepsilon_{m}(\xi)^{1/m}\,
(R^\wedge_p[\z_m]^{\times} + x\BK_p[\z_m][\![x]\!])
\ee
of power series (here $x=q-\z_m$ as usual) that satisfies
% \be
% \label{fxidef}
% \fr \log\big(\widehat{f}\thinspace\big)(q^p)-p\log\big(\widehat{f}\thinspace\big)(q)
% \in \prod_{m \geq1\,,\,(m,p)=1} \frac{p}{x}R^\wedge_p[\z_m][\![x]\!]\, .
% \ee
\be
\label{fxidef}
\log\bigg(\frac{\fr \widehat{f}(q^p)}{\widehat{f}(q)^p}\bigg)\;
\in \prod_{m \geq1\,,\,(m,p)=1} \frac{p}{x}R^\wedge_p[\z_m][\![x]\!]\, .
\ee
% \be
% \label{fxidefv2}
% \log\Big(\frac{\fr f(q^p)}{f(q)^p}\Big)
% \in \frac{p^2D_p(\xi)-\fr D_p(\xi)}{m^2p\log(q)}+\prod_{m \geq1\,,\,(m,p)=1}
% \frac{p}{x}R^\wedge_p[\z_m][\![x]\!]\, .
% \ee
Here $\widehat{f}(q)=(\widehat{f}_m(x))_m$ is the formal completion of $f$ defined by
\be
\label{hatfm}
\log(\widehat{f}_m(x))\=\frac{D_{p}(\xi)}{m^2\log(q)}+\log(f_m(x))
\;\;\,\in\;\; x^{-1}\thinspace\BK_p[\z_m][\![x]\!]\, .
\ee
We denote the $\calH_{R_p^\wedge}$-span of invertible $L_p(\xi)$-sections by
$\calH_{R_p^\wedge,\xi}$.
\end{definition}

The construction of such invertible sections can be done via telescoping sums.
Note also that the condition~\eqref{fxidef} is a logarithmic version of that used in
Dwork's lemma (see Lemma~\ref{lem:dwork} of Section~\ref{sub.habp}).
The existence of invertible sections and Dwork's lemma makes it easy to see that
$\calH_{R_p^\wedge,\xi}$ is a rank one module over $\calH_{R_p^\wedge}$.
This is summarised in the following theorem, whose proof will be given in
Section~\ref{sub.habp}.

\begin{theorem}
\label{thm.locals}
Let $\BK$ be a number field, and $\Delta$ a number divisible by its discriminant
and by~$6$. For $\xi \in K_3(\BK)$ and $p$ a prime with $(p,\Delta)=1$, there exists
invertible $L_p(\xi)$-sections. Moreover, $\calH_{R^\wedge_p,\xi}$ is a free rank~$1$
module over $\calH_{R^\wedge_p}$.
\end{theorem}

In Section~\ref{sub.habp} we will construct some natural and explicit
invertible sections.
These sections will depend on an expression of $\xi$ as a $\BZ_p$-linear combination of
symbols associated to roots of unity $\z$ in $\BK_p$, and on the expansion of the
infinite Pochhammer symbol $(q^{1/2}\z;q)_\infty$.
While these collections are only defined at $\z_m$ with $m$ prime to $p$,
they have a unique extension to all roots of unity satisfying integrality
and gluing conditions, as shown in Corollary~\ref{cor:unique.loc} of
Section~\ref{sub.habp}.

We can patch together these modules over the $p$-completed ring to define modules
over the global one.
This is completely analogous to the description of the global Habiro ring from
the $p$-completed Habiro ring via the diagram~\eqref{Hviapcomp}.
These collections are motivated and modelled on series coming from
perturbative Chern--Simons theory at roots of unity. 

\begin{definition}
\label{def.HRmod}
Fix $\xi \in K_3(\BK)$. The $\calH_{R}$-module $\calH_{R,\xi}$
consists of collections
\be
\label{fcyclic}
f(q) \=(f_{m}(x))_{m\geq1}, \qquad
f_{m}(x)\inn \varepsilon_{m}(\xi)^{1/m}\,
\BK[\z_m][\![x]\!]
\ee
such that under the canonical map $\BK\to\BK_p$ we have
$(f_{m}(x))_{m\geq 1,(m,p)=1}\in\calH_{R^\wedge_p,\xi}$, and the following
gluing condition for $\gamma\in\BZ_{>0}$ is satisfied:
\be
\label{eq:gluing}
f(q^\gamma)^\gamma f(q^{-1})\inn \calH_{R[1/\gamma]}|_{\gamma}\,,
\ee
where $\calH_{R[1/\gamma]}|_{\gamma}$ denotes the restriction of $\calH_{R[1/\gamma]}$
to collections of power series at roots of unity $\z_m$ with $m$ prime to $\gamma$,
as introduced after Definition~\ref{def.hab}.
\end{definition}
Note that the function $(\gamma^* f)(q):=f(q^\gamma)$ behaves like an element of
the module $\calH_{R[1/\gamma],\xi/\gamma}|_{\gamma}$ (where again $|_\gamma$ denotes
the restriction to roots of unity of order prime to $\gamma$).
The issue is that one must make sense of $\xi/\gamma$ as an element of $K_3$.
We also note that the $\chi^{-1}$-equivariance of $\ve_m$ makes the a priori
ambiguity of their $m$-th roots disappear in equation~\eqref{eq:gluing}. For the
moment we assume that $p$ is prime to $6$ (cf.~the remark in Section~\ref{sub.admFGI}),
because only then does the above condition yield the correct gluing condition:
one needs that there is some $\gamma$ prime to $p$ such that
$\gamma^2-1$ is also prime to $p$ (see the end of the proof of Lemma~\ref{lem:ftild}
of Section~\ref{sub.habp}).
Moreover, notice that $f(q^{\gamma})^{\gamma}f(q^{-1})$ determines
$f(q)$ for $q$ near all roots of unity of order prime to $6$ as we vary
$\gamma\in\BZ_{>0}$. More generally, one can consider the combinations
$f(q^{\gamma_1})\cdots f(q^{\gamma_n})f(q^{-\gamma_1'})\cdots f(q^{-\gamma_{n'}'})$,
where $\sum_{k}\gamma_k^{-1}-\sum_k\gamma_{k}'^{-1}=0$,
the previous case being $n=\gamma$, $\gamma_i=\gamma$, $n'=1$ and $\gamma_1'=1$.
The important property of these combinations is that including the polar
part of equation~\eqref{hatfm} has no effect on the
expansions at roots of unity with order prime to $\prod_k\gamma_k\prod_k\gamma_k'$.
For example, taking $q=\z_m+x$ with $(m,\gamma)=1$ we find that
$f(q^\gamma)^\gamma f(q^{-1})=\widehat{f}(q^\gamma)^\gamma \widehat{f}(q^{-1})$.

The module $\calH_{R,\xi}$ associated to the global ring is not necessarily free,
but it is of rank one and locally free, i.e., invertible.
This was observed numerically in our early experiments. In these experiments, we
compared power series coming from two knots giving the same class in $K_3(\BK)$ (in
particular, the $5_2$ and $(-2,3,7)$-pretzel knots with the field $\BK$ being the
cubic field of discriminant $-23$ with $R$ being its ring of integers with $23$
inverted). We observed that these series seemed to satisfy a linear dependence
over $\calH_R$. This numerical observation is indeed true and is implied by
Theorem~\ref{thm.2} below and the following:

\begin{theorem}
\label{thm.HR}
We have $\calH_{R,0}=\calH_R$. In general, multiplication gives a canonical isomorphism
\be
\calH_{R,\xi}\otimes_{\calH_R} \calH_{R,\xi'}
\;\cong\;
\calH_{R,\xi+\xi'}\,.
\ee
In particular $\calH_{R,\xi}$ is an invertible $\calH_R$-module, and
$\xi\mapsto \calH_{R,\xi}$ defines a homomorphism of abelian groups
\be
\label{eq:K3Pic}
K_3(\BK)\;\to\;
\mathrm{Pic}(\calH_{R})\,.
\ee
\end{theorem}
\noindent Here, $\mathrm{Pic}(\calH_{R})$ denotes as usual the abelian group of
line bundles on $\mathrm{Spec}(\calH_R)$, i.e., the group under tensor product of
invertible $\calH_R$-modules.
This map is actually a homomorphism from the abelian group $K_3(\BK)$ (in the category
of sets) to the abelian group of line bundles on~$\calH_{R}$ (in the category of
groupoids), i.e., to every element in $K_3(\BK)$ we obtain an invertible
\hbox{$\calH_R$-module} as opposed to just an isomorphism class.

We observe that, just as for $\calH_R$, the definition of the module $\calH_{R,\xi}$
becomes simpler when $\BK$ is abelian over $\BQ$. Indeed, in that case the field
$\BK$ can be embedded in $\BQ(\z)$ for some root of unity $\z$. Assume that
$\xi\in \calB(\BK)$ can be represented as a $\mathrm{Gal}(\BQ(\z)/\BK)$-invariant
combination $\xi=\sum_j n_j[\z^j]$. (We do not know whether this is always the case,
at least after tensoring with $\BQ$, but it certainly holds in many examples, e.g.,
the example coming from the figure eight knot, where $\xi=2[\z_6]$.)
Then $\calH_{R,\xi}$ contains the element $\prod_j(q^{1/2}\z^j;q)_\infty^{n_j}$ and is
freely generated by this element as an $\calH_R$-module.
Compare Theorem~\ref{thm.Psi} of Section~\ref{sub.habp}.

The next proposition summarises some of the properties of these $\calH_R$-modules
and the operation, defined after Definition~\ref{def.HRmod}, $f\mapsto \gamma^{*}f$.
Its proof, along with the proof of Theorem~\ref{thm.HR}, are given in
Section~\ref{sub.proofmain}.

\begin{proposition}
\label{prop.HRxi}
We have: 
\newline
\emph{(}a\emph{)}
  If $f \in \calH_{R,\xi}$, then $\tau f \in \calH_{R,-\xi}$, where
  $\tau f := (-1)^* f$ is the involution $f(q) \mapsto f(q^{-1})$.
\newline
\emph{(}b\emph{)}
  If $f,g \in \calH_{R,\xi}$, then $f \cdot\tau g \in \calH_{R}$. Moreover,
  there exist $a, b \in \calH_{R}$ such that $af=bg$.
  % This follows from the
  % associativity $(f \cdot \tau f) \cdot g = f \cdot (\tau f \cdot g)$
  % together with the fact that $f \cdot \tau f, \tau f \cdot g \in \calH_{R}$.
\newline
\emph{(}c\emph{)}
  If $f\in\calP_R$ and $\gamma,\gamma'\in\BZ$ are non-zero then $(\gamma \gamma')^* f
  \= \gamma^* (\gamma'^* f)$.
  % , and
  % hence $\gamma^* (\tau f)\= (-\gamma)^* f$.
\newline
\emph{(}d\emph{)}
  If $f \in \calH_{R,\xi}$ and $\gamma$ a positive integer, then
  $(\gamma^* f)^\gamma \in \calH_{R[\gamma^{-1}],\xi}|_{\gamma}$
  (see Definitions~\ref{def.hab}~and~\ref{def.HRmod}).
\newline
\emph{(}e\emph{)}
If $f\in\calH_{R,\xi}$ and $m\in\BZ_{>0}$ and $f_m(x)=0$ then $f_{pm}(x)=0$
for all primes $p$ not diving $\Delta$.
  % (even though it is not $p$-integral)
  % determines the series $f_{pm}(x)$ for all primes $p$ with prime to $\Delta$.
\newline
\emph{(}f\emph{)}
  If $f(q) \in \calH_{R,\xi}$, then 
  $f_m(0) \in R[\z_m,\ve_m^{1/m}]$ for $(m,\Delta)=1$. 
\end{proposition}

\begin{remark}
The operation $\tau$ corresponds geometrically to the orientation reversal of an
oriented manifold, i.e., the quantum invariants of an oriented three-manifold with
the opposite orientation are equal to the operation $\tau$ applied to the original
invariants.
\end{remark}

\subsection{Admissible series and formal Gaussian integration}
\label{sub.admFGI}
It was conjectured by Nahm that the modularity of certain $q$-hypergeometric
functions is intimately related to the vanishing of certain associated classes in
algebraic $K$-theory, as was indeed later seen in a variety of
works~\cite{Nahm,Zagier:dilog,GZ:asymptotics,CGZ}.
The asymptotics of these $q$-hypergeometric functions as $q$ appraoches roots of
unity can be defined using formal Gaussian integration~\cite{Zagier:dilog,DG,DG2}.
Modularity would then imply some triviality of these asymptotic series.
We will see that this link goes even deeper, and that the
integrality of the formal series (as series in $q-\z$) is also linked to the
vanishing of the same classes in algebraic $K$-theory.

As we will see, there are two sources of these collections of power series at roots
of unity, with complementary properties.
One of these is the ``admissible series''
of Kontsevich--Soibelman~\cite[Sec.6]{KS:cohomological}, which arise as generating
series of cohomological Hall algebras that appear in enumerative algebraic geometry
such as Donaldson--Thomas theory.

\begin{definition}
For $N\in\BZ_{>0}$ and $t=(t_1,\dots,t_N)$, an admissible series is a power series
$F(t,q)\in\BQ(q)[\![t]\!]$ satisfying $F(t,q)=1+\mathrm{O}(t)$ and
\be
\label{logF}
\log F(t,q) \= -\!\!\!\!\sum_{0\neq n\in(\BZ_{\geq0})^N}\sum_{\ell \geq 1}
\frac{L_n(q^\ell)}{\ell\thinspace(1-q^\ell)}
t_1^{\ell n_1}\cdots t_N^{\ell n_N}\,,\qquad
\text{where}\qquad
L_n(q)\inn\BZ[q^{\pm1}].
\ee
\end{definition}

\noindent An equivalent way to describe admissible series is via their factorisation
into Pochhammer symbols. In particular, a series $F(t,q)$ is admissible if and only
if the unique integers $c_{n,i}$ (generalised Donaldson--Thomas invariants), defined
by the equality
\be
\label{defgamma}
F(t,q) \= \prod_{0\neq n\in(\BZ_{\geq0})^N}
\prod_{i \in \BZ} (q^i t_1^{n_1}\cdots t_N^{n_N};q)_\infty^{c_{n,i}},
\ee
vanish for all but finitely
many $i$ for each fixed $n$.
Here, as usual $(x;q)_n = \prod_{j=0}^{n-1}(1-q^j x)$
and
$(x;q)_{\infty}^{-1}=\sum_{k=0}^{\infty}x^k/(q;q)_k$.
The equivalence follows from the elementary but key identity~\eqref{logpoc}
for the infinite Pochhammer symbol, which implies the relation
\be
  L_n(q) \= \sum_{i \in \BZ} c_{n,i} \, q^i \inn \BZ[q^{\pm 1}]
\ee
between the exponents $c_{n,i}$ and the Laurent polynomials $L_{n}(q)$.

An admissible series $F(t,q)$ can be expanded at each root of unity and defines
a collection of series in $t^{1/m}=(t_1^{1/m},\dots,t_N^{1/m})$ given by
\be
\label{Phidef}
\Phitof_m(t,x) \= F(t^{1/m},\z_m+x) \,. %\in e^{\frac{\z_m V(t)}{m^2 x}} \,,
%\BQ[\z_m][\![t^{1/m}]\!][\![x]\!] 
\ee
The properties of these collections of series at roots of unity are explained in
detail in Section~\ref{sub.admissible} below.

Certain admissible series are easy to construct, compute and analyse.
A main theorem of~\cite[Sec.\thinspace6.1, Thm.\thinspace9]{KS:cohomological} (see
also Efimov~\cite{Efimov}) is that $q$-hypergeometric Nahm
sums\footnote{The original Nahm sums were given for positive definite $A$ by the
  special case when $t=(-1)^{\diag(A)}q^{b}$ for some $b\in\BZ^{N}$.}
\be
\label{FAdef}
F_A(t,q) \= \sum_{n \in \BZ^N_{\ge0}} \frac{
% (-q^{\frac{1}{2}})^{n^t A n}q^{\frac{1}{2}\mathrm{diag}(A)\cdot n}
(-1)^{\mathrm{diag}(A)\cdot n}\,
q^{\frac{1}{2}(n^t A n + \mathrm{diag}(A)\cdot n)} }{(q;q)_{n_1}
\dots (q;q)_{n_N}} t_1^{n_1} \dots t_N^{n_N}
\ee
defined from a symmetric, integral $N \times N$ matrix $A$,
are always admissible series. These series encode the Poincar\'e polynomials of
representations of a quiver with symmetric potential and arbitrary dimension vector.

An elementary proof of the admissibility of $F_A(t,q)$ that uses the system of linear
\hbox{$q$-difference} equations~\eqref{PhiAshift} is given in Section~\ref{sub.DT}.
The proof also leads to a notion of a level $m$ admissible series defined in
Section~\ref{sub.madmissible} (abbreviated by $m$-admissible series throughout
this paper), the prototypical example being given by exactly the same sums as in
equation~\eqref{FAdef} with $n$ restricted to an $m$-congruence class
$k\in\{0,\dots,m-1\}^N$
\be
\label{FAmdef}
F_{A,m,k}(t,q) \;=\!\!\!\!
\sum_{n \,\in\, k+m\BZ^N_{\ge0}}\!\!\!\!\!\! \frac{
(-1)^{\mathrm{diag}(A)\cdot (n-k)}
q^{\frac{1}{2}(n^t A n-k^t A k)
  +\frac{1}{2}\mathrm{diag}(A)\cdot (n-k)} }{(q^{k_1+1};q)_{n_1-k_1}
\dots (q^{k_N+1};q)_{n_N-k_{N}}}\,t_1^{n_1-k_1}\!\!\cdots
t_N^{n_N-k_N}\!\!.
\ee

On the other hand, we also have a collection of power series at roots of unity
$\Phitof^\FGI_{A,m}(t,x)$ associated, by formal Gaussian integration, to an integer
symmetric matrix $A$. These series were previously defined
in~\cite{Zagier:dilog,DG,DG2} for the special case $t=(1,\dots,1)$ giving series
$\Phi_{A,m}(h)=\Phitof^\FGI_{A,m}(1,x)$, expressed there with respect to the
variable $h=\log(1+\z_m^{-1}x)$. This definition mimics the perturbation
expansions of complex Chern--Simons theory given in terms of state-integrals, and is
discussed in detail in Section~\ref{sub.FGI}.
Both collections $\Phitof_{A,m}(t,x)$ and $\Phitof^\FGI_{A,m}(t,x)$ are
in $\BQ[\z_m](\!(x)\!)[\![t^{1/m}]\!]$, for every positive integer $m$. 

Our two main results are that these two seemingly independent collections of
power series agree and that their values at $t=(1,\dots,1)$ belong to explicit
Habiro modules. In fact, these series
can be defined by the admissible series of Kontsevich--Soibelman and this gives
a new perspective to the arithmetic properties of DT-invariants and the
enumerative properties of the elements of the Habiro ring.

\begin{remark}
\label{rem:coprimeto6}
The primes $p=2$ and $3$ require special consideration. This is true for the
Bloch group itself, which has a uniform definition up to $2$-torsion and agrees
with the algebraic third $K$ group of a field away from some primes that always
include $2$ and $3$. There are further difficulties with the primes $2$ and $3$
due to the presence of a $\tfrac{1}{24}$ whose origin can be seen to
come from the Dedekind $\eta$-function and the famous equation of Euler
$\Li_{2}(1)=-\tfrac{1}{24}(2\pi i)^2$. Therefore, throughout the paper we will
only consider primes $p>3$ and always assume that $2,3|\Delta$ unless otherwise
stated. However, a more refined approach should be able to cover the remaining cases
as well.
\end{remark}

\subsection{The main theorems}
\label{sub.results}

We now state our main results. The first result identifies the two collections
$\Phitof_{A,m}(t,x)$ and $\Phitof^\FGI_{A,m}(t,x)$ of power series at roots of unity and
draws complementary conclusions about them. A key ingredient is the fact that
both collections are solutions to the $q$-holonomic system of equations
\be
\label{PhiAshift}
F_A(t,q) - F_A(\s_jt,q) \= (-1)^{A_{j,j}} t_{j} q^{A_{j,j}}
F_A\Big(\prod_{i=1}^N \s_i^{A_{i,j}}t,q\Big), \qquad j=1,\dots,N\,,
\ee
where $\s_jt$ shifts $t_j$ to $q t_j$ and keeps $t_{j'}$ for $j' \neq j$ fixed.

After setting $q=1$ and replacing the operators $\s_j$ by $z_j(t)$, the
equations~\eqref{PhiAshift} become the $t$-deformed Nahm equations~\eqref{zjt}
\be
\label{zjt}
1-z_j(t) \= (-1)^{A_{j,j}} t_j \prod_{i=1}^N z_j(t)^{A_{i,j}},
\qquad j\=1,\dots,N\,,\qquad
z_j(0)\=1\,.
\ee
These equations in turn define the ring 
\be
\label{Sdef}
S\= \BZ[t^{\pm 1}, z^{\pm 1}(t), \delta(t)^{-1/2}]\big/
(1-z(t)-(-1)^A t \,z(t)^A)\,,
\ee
whose relations are a shorthand of the equations~\eqref{zjt},
where $z(t)=(z_1(t),\dots,z_N(t))$ and
\be
\label{taulambda}
%\delta(t)\= \prod_{j=1}^Nz_j(t)^{-A_{jj}}(1-z_j(t)) \det(-\Lambda(t)), \qquad
%\Lambda(t)\=-A-\mathrm{diag}\Big(\frac{z(t)}{1-z(t)}\Big) 
\delta(t)\;:=\;\prod_{j=1}^Nz_j(t)^{-A_{jj}} \det(\mathrm{diag}(1-z(t)) A +
\mathrm{diag}(z(t)))
\ee
is the discriminant of the $t$-deformed Nahm equations, so that (after inverting $2$)
$S$ is an \'etale $\BZ[t]$-algebra.
The equations~\eqref{zjt} are a $t$-deformation of Nahm's equations and appear
both in the admissible series (see Section~\ref{sub.admissible}) as well as in
the formal Gaussian integration series (see Section~\ref{sub.FGI}), and
play a key role in identifying the two collections of series. 

More generally, for a positive integer $m$, we define
\be
\label{RA}
S^{(m)}\=
\BZ[\z_m,t^{\pm 1/m}, z^{\pm 1}(t),\delta(t)^{-1/2}]\big/
(1-z(t)-(-1)^A t \,z(t)^A)
\ee
and note that $S^{(1)}=S$.

\begin{theorem}
\label{thm.1}
For every symmetric matrix $A$ with integer entries we have:
\be
\label{eq.thm1}
\Phitof_{A}(t,q) \= \Phitof_{A}^\FGI(t,q)\,.
\ee
\end{theorem}

\begin{theorem}
\label{thm.FGI2}
Fix a symmetric matrix $A$ with integer entries, a positive integer $m$ and
an $m$-congruence class $k \in\{0,\dots,m-1\}^N$. Then, for every prime $p$ with
$(m,p)=1$, we have:
\be
\label{dphiA}
\log(F_{A,m,k}(t^{p/m},q^p))
-p\log(F_{A,m,k}(t^{1/m},q))
\inn \frac{p}{x} S^{(m)}[z(t)^{1/m}]^{\wedge}_p [\![x]\!]\,,\qquad
q\=\z_{m}+x \,.
\ee
\end{theorem}

These theorems are proved in Section~\ref{sub.synthesis}.

Our main theorem below, which follows from the above theorems by specialisation
to $t=1$, states that the perturbative complex Chern--Simons invariants of knots
and 3-manifolds are elements of the Habiro modules. The $t=1$ specialisation refers
to the ring homomorphism
\be
\begin{array}{c}
S \= \BZ[t^{\pm}, z^{\pm 1}(t), \delta(t)^{-1/2}]
\big/(1-z(t)-(-1)^A t \,z(t)^A)_{\phantom{2}_{\phantom{2}_{\phantom{2}_{
\phantom{2}_{\phantom{2}_{\phantom{2}}}}}}}\\
\big\downarrow_{\phantom{2}_{\phantom{2}_{\phantom{2}_{\phantom{2}}}}}\\
R[\delta^{-1/2}] \= \BZ[z^{\pm 1}, \delta^{-1/2}]\big/(1-z-(-1)^A \,z^A)\,,
\end{array}
\ee
where we again use shorthand in the definition of $R[\delta^{-1/2}]$ with
$z=(z_1,\dots,z_N)$ and the relations given by the Nahm equations
\be
\label{zjt=0}
1-z_j \= (-1)^{A_{j,j}} \prod_{i=1}^N z_j^{A_{i,j}},
\qquad j\=1,\dots,N \,.
\ee
The specialisation $t=1$ is compatible with the Frobenius endomorphism $\fr(t)=t^p$
of the $p$-completions of the above rings. Below, we will fix an irreducible
component of the equations~\eqref{zjt=0} that is non-degenerate, i.e., an isolated
solution $z$ with $\delta \neq 0$.

Such data gives rise to an element
$\xi=\sum_{j=1}^N [z_j] \in B(\BK)$ in the Bloch group of the number field $\BK$
generated by $z$. In geometry, the data $(A,z)$ comes from an ideal triangulation
$\calT$ of a 3-manifold with torus boundary components, together with a solution
$z$ of the gluing (i.e., the Neumann--Zagier~\cite{NZ}) equations, as was explained
in~\cite{DG}. Thus, we denote the specialisation of $\Phitof_{A}(t,q)$ to $t=1$, after
removing the principle part of the logarithm, by $\Phitof_{A,z}(q)$.

Fixing a prime $p$ prime to $\Delta$, $V(t)$ specialises under $t=1$ to $D_p(\xi)$.
When $t=1$, $V(t^p)=\fr V(t)$ specialises
to $D_p(\fr \xi) = \fr D_p(\xi)$, where $\fr$ now reduces to the Frobenius
endomorphism of $\BK_p$.

To state the next theorem we need to introduce some technical assumptions regarding
primes of bad reduction. Recall that in Definition~\ref{def.HRmod} of
Section~\ref{sub.habmod} we used collections of series at all roots of unity.
However, below we will restrict to the subset of roots of unity with order prime
to $\Delta$. To this end, extending the notation of Definition~\ref{def.hab} of
Section~\ref{sub.hab} and the remark
% ~\ref{rem.domain}
of Section~\ref{sub.hab} we define $\calH_{R,\xi}|_\Delta$ to denote the set
of collections of series indexed by roots of unity of order
prime to $\Delta$ that satisfy~\eqref{gluef} for all positive integers $m$ and
primes $p$ such that $(pm,\Delta)=1$. When $\xi=0$, we have
$\calH_{R,0}|_\Delta = \calH_{R}|_\Delta$.

\begin{theorem}
\label{thm.2}
Fix a symmetric matrix $A$ with integer entries and a non-degenerate solution
$z$ of the Nahm equations with associated $\xi$. Then we have:
\be
\label{eq.thm2}
\Phitof_{A,z}(q) \inn \calH_{R[\delta^{-1/2}],\xi}|_{\Delta} \,.
\ee
\end{theorem}

\begin{remark}
A stronger statement of the above theorem is probably true in
$\calH_{R[\delta^{-1/2}],\xi}$ without including $\Delta$,
and in fact follows if Theorem~\cite[Thm.~1.6]{CGZ} holds for any $m\in\BZ$.
Moreover, under the action of the Galois automorphism of $R[\delta^{-1/2}]$
sending $\sqrt{\delta}$ to $-\sqrt{\delta}$, the element $\Phitof_{A,z}$ lives in
the $-1$-eigenspace.

All the results of this section can be strengthen to rational $A$ by extending the
proofs using $m$-admissible series, where $m$ will include factors from the
denominator of $A$. This will consequentially introduce more factors in $\Delta$.
\end{remark}

The above theorem combined with part (f) of Proposition~\ref{prop.HRxi} of
Section~\ref{sub.habmod} implies:

\begin{corollary}
\label{cor.cint}  
The constant term of the expansion of $\Phitof_{A,z}(q)$ at $m$-th roots of unity
$\z_m$ of order prime to $\Delta$ is in $R[\z_m]$.
\end{corollary}

\noindent This is by no means an obvious
fact. In the case of the series associated to the $4_1$ knot, it asserts
the integrality of the following sum (see, e.g., ~\cite[Eqn.(95)]{GZ:kashaev})
\be
\label{cherry41}
\frac{1}{\sqrt{m}}%D_{\z_m}(\th) D_{\z_m}(\th^{-1})
\sum_{k \in \BZ/m\BZ}
(\z_m \th;\z_m)_k (\z_m^{-1} \th^{-1};\z_m^{-1})_k \inn \BZ[\z_{6m}],
\qquad \th^m\=\z_6\,,
\qquad\text{for }(m,6)\=1\,.
\ee

The above theorems allow us to construct elements of the ring $\calH_R$ by
considering torsion elements of the Bloch group, or by symmetrising elements of
$\calH_{R,\xi}$ under $q\mapsto q^{-1}$.

\begin{corollary}
\label{cor.3}
Fix $A$ and $z$ as in Theorem~\ref{thm.2}.
\newline
\emph{(}a\emph{)} We have
\be
\label{symhab}
\Phitof_{A,z}(q)\thinspace \Phitof_{A,z}(q^{-1}) \inn \calH_{R} \,.
\ee
\emph{(}b\emph{)} If $r \, \xi =0 \in K_3(\BK)$ for some positive integer $r$, then
$\Phitof_{A,z}(q)^{r} \in \calH_{R[\delta^{-1/2}]}$.
\end{corollary}

In fact, one can obtain even more elements of $\calH_R$ by applying ``descendants''
to the above theorem, see Theorem~\ref{thm.desc} of Section~\ref{sub.sym}, as well
as elements of the modules $\calH_{R,\xi}$ by specialising $t=q^\nu$; see the Remark
% ~\ref{rem.desc}
in Section~\ref{sub.proofmain}.
Keeping in mind that the Habiro ring~$\calH_R$ of a number field $\BK$ is a rank
$r$ module over $\calH_{\BZ[1/\Delta]}$, where $r=[\BK:\BQ]$, the above constructions
presumably give a spanning set for $\calH_R$ at the generic point of $\calH$.
We illustrate this with examples in Section~\ref{sec.examples}.

The above corollary gives an if and only if statement: Let us say that an orbit of
the Nahm equation is Bloch-torsion if the associated element of the Bloch group
is torsion. Then an orbit is Bloch-torsion if and only if the series
$\Phitof_{A,z}(q)^{2r} \in \calH_{R}$, where $r$ is the order of the torsion element
of the Bloch group. One direction is given in the above corollary. The converse
direction follows from the fact that the vanishing of the Bloch-Wigner dilogarithms
of all complex embeddings of an element of the Bloch group implies torsion, or from
the fact that the triviality of the unit $\ve_m(\xi)$ for all but finitely $m$
implies torsion.

\subsection{The relation with perturbative Chern--Simons theory}
\label{sub.CS}

Theorem~\ref{thm.2} involves a symmetric matrix with integer entries and a
non-degenerate solution $z$ of the Nahm equations. As such, it looks removed from
perturbative Chern--Simons theory with complex gauge group.
In this subsection, we briefly explain the relation betwee Theorem~\ref{thm.2} and
perturbative complex Chern--Simons theory.

To begin with, Chern--Simons theory
with a complex gauge group (such as $\SL_2(\BC)$) was formulated as a path integral
by Witten~\cite{Witten:complexCS}, following his earlier work on a path integral
with compact gauge group (such as $\mathrm{SU}(2)$)~\cite{Witten:jones}.

However,
there is a crucial difference between complex versus compact Lie groups. The path
integral with compact gauge group satisfies cut-and-paste properties that lead to
its effective computation, and to a combinatorial definition developed by
Reshetikhin--Turaev~\cite{RT:ribbon,Tu:book}.
On the other hand, the path integral with complex
gauge group is lacking cut and paste properties and, to the best of our knowledge,
it is unknown for instance how to compute it for closed 3-manifolds. 

For reasons that are not entirely understood, the partition function of 
Chern--Simons theory with complex gauge group $\SL_2(\BC)$ for manifolds with torus
boundary components conjecturally reduces to a finite-dimensional integral (the so called-state
integral) whose integrand is a product of Faddeev's quantum dilogarithm
functions~\cite{Faddeev}, assembled out of an ideal triangulation of
the manifold. This was the approach taken by
Andersen-Kashaev~\cite{AK, AK-review} and Dimofte~\cite{dimofte-rev}
following prior ideas of~\cite{Hikami,DGLZ}. Focusing for simplicity
on the case of a 3-manifold with a single torus boundary component
(such as the complement of a hyperbolic knot in $S^3$), the
state-integral is a holomorphic function of
$\tau \in \BC'=\BC\setminus (-\infty,0]$.

We now turn to perturbative Chern--Simons theory with complex gauge group
$\SL_2(\BC)$. Using the above mentioned (state-integral) partition function, in
a series of papers the perturbation theory was defined first at $q$ near
$1$~\cite{DG} and then for $q$ near an arbitrary complex root of unity~\cite{DG2}.
This perturbation theory depends on the combinatorics of an ideal triangulation
of a cusped hyperbolic 3-manifold, a solution $z$ of its Neumann--Zagier equations,
as well as some auxillary choices (of a quad of the tetrahedra and of an edge of the
triangulation). These choices lead to a so-called Neumann--Zagier datum, which
is a the upper part $(\boldsymbol{A}|\boldsymbol{B})$ of a symplectic $N \times N$
matrix with integer entries, a vector $\nu \in \BZ^N$, and a solution $z$ of the
Neumann--Zagier equations
\be
\label{NZeqn}
\prod_{j=1}^N z_j^{\boldsymbol{A}_{ij}} (1-\tfrac{1}{z_j})^{\boldsymbol{B}_{ij}}
\= (-1)^{\nu_i}, \qquad i=1,\dots N\,.
\ee
The solution $z$ in turn defines a representation $\pi_1(M) \to \PSL_2(\BC)$ up to
conjugation, which under mild hypotheses can be lifted to a representation $\rho_z$
to $\SL_2(\BC)$. 
Fixing such a Neumann--Zagier datum, the complex Chern--Simons perturbative series
around $\rho_z$ are defined in~\cite{DG,DG2} by formal Gaussian integration, 
and these are none other than the series $\Phi(h)$ of Equation~\eqref{as41} for
the case of the $4_1$ knot, where $\rho_z$ is an $\SL_2(\BC)$ lift of the discrete
faithful $\PSL_2(\BC)$-representation of the $4_1$ knot. 

The symplectic property of $(\boldsymbol{A}|\boldsymbol{B})$ the Neumann--Zagier
matrix implies that when $\boldsymbol{B}$ is invertible over $\BZ$, then
$A=I-\boldsymbol{B}^{-1}\boldsymbol{A}$ is a symmetric $N \times N$ matrix with
integer entries. Moreover, the Neumann--Zagier equations~\eqref{NZeqn} are equivalent
to the Nahm equations~\eqref{zjt=0}, and what's more the perturbative series at
roots of unity defined by formal Gaussian integration \emph{syntactically matches} that
of the corresponding Nahm sum (when $A$ is positive definite) and of the admissible
series for general $A$. This matching is discussed in detail in~\cite{GZ:nahm} and is generalised in Theorem~\ref{thm.1}.

Putting everything together, it follows that Theorem~\ref{thm.2} contains the
asymptotics of complex Chern--Simons theory when a 3-manifold with torus boundary
has an (essential) ideal triangulation and a choice of quad with unimodular matrix
$\boldsymbol{B}$. While such a choice is not known to exist in general, it has
been found for all knots with at most 14 crossings using their
default \texttt{SnapPy} triangulation. 

\subsection{Future extensions}
\label{sub.further}

We conclude this introduction by discussing several natural extensions of
our paper that we hope to come back to.

\noindent $\bullet$ {\bf Bad primes.}
The current definition of $\calH_R$ is a collection of power series associated to
all roots of unity that satisfies the gluing condition~\eqref{gluef}, which allows
to re-expand the series at $\z_{pm}$ from the series at $\z_m$ as long as the prime
$p$ is prime to $\Delta$. In other words, the series at $\z_1$ does not determine
the series at any $\z_m$ when $m$ is not prime to $\Delta$. This implies that
$\calH_R$ is not an integral domain, but instead a product of integral domains,
see the remark
% ~\ref{rem.domain}
of Section~\ref{sub.hab}. On the other hand, the collections $\Phitof_{A,z}(q)$ of
power series associated to a symmetric matrix $A$ together with a non-degenerate
solution $z$ to the Nahm equations assigns a collection of power series
$\Phitof_{A,z,m}(x)$ at all roots of unity, not just those that are prime to an
integer $\Delta$; see Theorem~\ref{thm.2} of Section~\ref{sub.results}. Understanding
how these series at
the bad primes determine the series at the remaining primes is an interesting
question whose answer could potentially improve the current definition of $\calH_R$.

\noindent $\bullet$ {\bf Relative Bloch group.}
The Habiro ring of the \'etale map $\BZ[t]\to S$
(with $S$ given in~\eqref{Sdef}) can be defined and
is being studied in the upcoming thesis of Ferdinand Wagner. With additional work,
one can define modules over this Habiro ring that are graded by a relative Bloch group
(defined like the usual Bloch group with $t \wedge S = 0$)
that ought to be identified with some relative $K$-theory group, and that the
collection of series $\Phitof_{A}(t,q)$ is an element of such a module.

\noindent $\bullet$ {\bf Higher weight modules.}
One can define higher weight modules over the Habiro ring of a number field,
indexed by the odd $K$ groups of a number field, which we hope to discuss in a
subsequent publication.

\noindent $\bullet$ {\bf Line bundles on $A_{\mathrm{inf}}$.}
The line bundles defined in the present paper can also be defined over
the period ring $A_{\mathrm{inf}}$ of $p$-adic Hodge theory, for all elements
of $K_3^{\mathrm{cont}}(\mathbb C_p)$. In this picture, the functions $\Phitof_A$
define sections of this line bundle on $\mathrm{Spa}(A_{\mathrm{inf}})$ away
from a disc around $\{q^{1/p}=1\}$.

\noindent $\bullet$ {\bf Relation with holomorphic quantum modular forms.}
A fifth extension concerns the quantum modularity properties
of the collection of the series associated to knots over $\BC$, and especially
the extension of matrices of holomorphic functions on $\BC\smallsetminus\BR$ to the
cut plane $\BC'=\BC\smallsetminus (-\infty,0]$, or more generally
$\BC_\gamma=\BC\smallsetminus \{\tau \in \BR \,| c \,\tau+d \leq 0\}$ for an element
$\gamma=\sma abcd \in \SL_2(\BZ)$; see~\cite[Sec.5]{GZ:kashaev}
and~\cite[Sec.5]{GZ:qseries}. There are some tantalising parallels between the
$p$-adic picture of the preceding bullet and this complex picture, similar to the
parallels between the Fargues-Fontaine curve and the so called twistor-$\mathbb P^1$.

%%%%%%%%%%%%%%%%%%%%%%%%%%%%%%%%%%%%%%%%%%%%%%%%%%%%%%%%%%%%%%%%%%%%%%%%%%% 
%%%%%%%%%%%%%%%%%%%%%%%%%%%%%%%%%%%%%%%%%%%%%%%%%%%%%%%%%%%%%%%%%%%%%%%%%%%

\section{Two complementary collections of power series at roots of unity}
\label{sec.adm}

In this section we will describe two beautiful consequences of a formal power
series satisfying a simple system of linear $q$-difference equations. On one hand
these series will have a subtle integrality property called
\emph{admissibility}\thinspace---\thinspace introduced by
Kontsevich-Soibelman~\cite{KS:cohomological}\thinspace---\thinspace and on the other
hand they will satisfy \emph{algebraicity}. The combination of these two
properties will provide the proofs of our main theorems. We will begin by explaining
these results for the case of the Pochhammer symbol. Then we will describe some
generalities on admissible series and an elementary proof that our series are
admissible. This will lead us to discuss the generalisation to level $m$ admissible
series, which come from the restriction to congruence sums of the original
admissible series.
After that we will give a different formula for the same objects using
formal Gaussian integration and an example of how the WKB (Wentzel–Kramers–Brillouin)
method can be used to prove a similar statement. We will end this section with
proofs of Theorem~\ref{thm.1} and Theorem~\ref{thm.FGI2} of
Section~\ref{sub.results}, which require combining all of the previous results.

\subsection{The  infinite Pochhammer symbol}
\label{sub.poch}

A key role in our paper is played by the infinite Pochhammer symbol
$(t;q)_\infty \= \prod_{n \geq 0} (1-q^n t)$.
In this section we collect some elementary, although remarkable, properties of it. 
There are two well-known complementary expansions of the logarithm of $(t;q)_\infty$
\begin{align}
\label{logpoc}
\log (t;q)_\infty & \= -\sum_{\ell \geq 1} \frac{t^\ell}{\ell(1-q^\ell)}
%\in t \BQ(q)[\![t]\!]
\\
\label{logpoc2}
% -\log (t;q)_\infty
& \= \sum_{k \geq 0} \frac{B_k}{k!} \Li_{2-k}(t) \log(1+x)^{k-1}\,,
%\in \frac{\Li_2(t)}{\log(1+x)} -\frac{1}{2} \Li_1(t) + x \BQ[t,\frac{1}{1-t}][\![x]\!]
\end{align}
where $\Li_n(t)=\sum_{k \geq 1} \tfrac{t^k}{k^n}$ is the $n$th polylogarithm and
$B_k=B_{k}(0)$ are the Bernoulli numbers given by the generating series~\eqref{bern}
with $B_0=1$, $B_1=-1/2$, etc.

These two expansions in turn imply two complementary views
of the Dwork-type difference, namely for every prime $p$, and for $q=1+x$, we have

\begin{align}
\label{logpocc}
\log(t^p;q^p)_\infty-p \log(t;q)_\infty & \=
p \sum_{\ell \geq 1, \,\, p\nmid\ell} \frac{t^\ell}{\ell(1-q^\ell)}
%\in \frac{p}{x} \BZ_{(p)}[\![t]\!][\![x]\!]
\\
\label{logpocd}
% \log(t^p;q^p)_\infty-p \log(t;q)_\infty
& \=
-p \sum_{k \geq 0} \frac{B_k}{k!} \Li_{2-k}^{(p)}(t) \log(1+x)^{k-1}\,,
%\in \frac{\Li_2^{(p)}(t)}{\log(1+x)} -\frac{1}{2} \Li^{(p)}_1(t) + 
\end{align}
where 
\be
\label{Lip}
\Li_n^{(p)}(t) \= \Li_n(t) - \frac{1}{p^n} \Li_n(t^p) \=
\sum_{k \geq 1, \,\, p\nmid k} \frac{t^k}{k^n} 
\ee
is the $p$-version of the polylogarithm~\cite[Prop.6.2]{Coleman}.

On the one hand, ~\eqref{logpocc} implies that the coefficient of $x^{k-1}$ in the
Dwork-type difference is in $p\BZ_{(p)}[\![t]\!]$, and on the other hand,
~\eqref{logpocd} implies that it is a $\BQ$-linear combination of the series
$\Li_{2-s}^{(p)}(t)$ for $s=0,\dots,k$. 
But in fact, more is true. In~\cite[Prop.6.2]{Coleman} Coleman proves (working with
the variable $1/t$) that $\Li_n^{(p)}(t) \in \BZ_{(p)}[\![t]\!]$ is convergent for
$|t/(1-t)|< p^{(p-1)^{-1}}$ and consequently lies in the ring
$\BZ[\tfrac{t}{1-t}]^\wedge_p$ of all series in
$\BZ_p[\![t]\!]$ that converge for $|t/(1-t)| \leq 1$. In the present paper, we will
only use a weaker statement with an elementary proof.

\begin{lemma}
\label{lem.lip}
For every integer $n$ and prime number $p$ we have
\be
\label{eq.LIP}
\Li_n^{(p)}(t) \inn \BZ[t,\tfrac{1}{1-t}]^\wedge_p \,.
\ee
\end{lemma}

%% see pari file: lip.completion.pari

\begin{proof}
For $r=1,\dots,p-1$ we have
\be
\frac{1}{r+pk} \=
r^{-1}\sum_{\ell=0}^{\infty}(-1)^\ell(r^{-1}pk)^{\ell}\inn\BZ_p\,.
\ee
This implies that for all positive integers $n$ and $N$ we have
\be
\frac{1}{(r+pk)^n} \cm \frac{1}{(r+p(k+p^N))^n} \pmod{p^N} \,.
\ee
Using equation~\eqref{Lip} and separating the summation over the
positive integers prime to $p$ into congruence classes modulo $p^N$, we obtain that
\be
\Li_n^{(p)}(t) \cm \sum_{r=1}^{p-1} \sum_{k=0}^{p^N-1}
\sum_{\ell=0}^{\infty} \frac{t^{r+pk+p^N\ell}}{(r+pk)^n}
\=
\frac{1}{1-t^{p^N}} \sum_{r=1}^{p-1} \sum_{k=0}^{p^N-1}
\frac{t^{r+pk}}{(r+pk)^n} \pmod{p^{N}}\,.
\ee
Moreover, $\frac{1}{1-t^{p^N}}\in\BZ[t,\tfrac{1}{1-t}]^\wedge_p$, which completes
the proof.
\end{proof}

\begin{proposition}
\label{prop.poclip}
\emph{(}a\emph{)} For every prime $p$ we have 
\be
\label{logpocf}
\log(t^p;q^p)_\infty-p \log(t;q)_\infty \inn
\frac{p}{x}\BZ[t,\tfrac{1}{1-t}]^\wedge_p [\![x]\!]\,,\quad q\=1+x \,.
\ee
\emph{(}b\emph{)} Fix a root of unity $\z \in \BC_p\smallsetminus \{1\}$ of order
not a power of $p$. The function
\be
\label{eq.cor.logpoc}
\log(\z^p;q^p)_{\infty}
-
p\log(\z;q)_{\infty}
\inn
p\frac{\Li_{2}^{(p)}(\z)}{x}
+p\BZ_{(p)}[\z][\![x]\!]\,,\quad q\=1+x
\ee
is meromorphic on the disc $|x|<1$ with a simple pole at $x=0$ and the residue as
given. 
\end{proposition}

\begin{proof}
Equation~\eqref{logpocf} follows from~\eqref{logpocc}, Lemma~\ref{lem.lip},
and the elementary identity
\be
\label{elem}
\BZ[t,\tfrac{1}{1-t}]^\wedge_p [\tfrac{1}{p}] \cap
\BZ_{(p)}[\![t]\!] \,\, \subseteq \,\,   \BZ[t,\tfrac{1}{1-t}]^\wedge_p \,.
\ee
Part (b) follows trivially from part (a) by applying
the natural map
\be
\BZ[t,\tfrac{1}{1-t}]^\wedge_p \to \BZ_p[\z], \qquad t \mapsto \z \,,
\ee
which is compatible with the Frobenius endomorphism $t\mapsto t^p$.
\end{proof}

%Equations~\eqref{logpocf} and~\eqref{eq.cor.logpoc} have an analogue 
%\begin{align}
%\label{logpocf2}
%\log(t^p;q^p)_\infty-p \log(t;q)_\infty &\in
%\frac{\z_{m'} p}{m'^2 x}\BZ[t,\tfrac{1}{1-t^{m'}}]^\wedge_p [\![x]\!] \,,
%\end{align}
%for $q=\z_{m'}+x$ where $(m',p)=1$
%and
%\begin{align}
%\label{eq.cor.logpoc2}
%\log(\z^p;(\z_{m'}+x)^p)_{\infty}
%-
%  p\log(\z;\z_{m'}+x)_{\infty}
%  &\in
%  p\frac{\Li_{2}^{(p)}(\z^{m'})}{x}
%  +p\BZ_{(p)}[\z_{m'},\z][\![x]\!]
%\end{align}
%for $\z \not\in \mu_{m' p^\infty}$.

The power series expansion in equation~\eqref{logpoc2} was centred about $q=1$,
but it has an extension at all roots of unity given by
\be
\label{logpocm}
\log  (t;q)_{\infty}
\=
\sum_{j=0}^{m-1} \log (q^jt;q^m)_{\infty}
\=
\sum_{j=0}^{m-1}\sum_{k=0}^{\infty}
\frac{B_{k}(j/m)m^{k-1}}{k!}
\Li_{2-k}(\z_m^j t)\log(1+x/\z_m)^{k-1} \,,
\ee
where $q=\z_{m}+x$ and the Bernoulli polynomials $B_k(y) \in \BQ[y]$
are defined by the generating series 
\be
\label{bern}
\frac{xe^{xy}}{e^x-1} \= \sum_{k=0}^\infty \frac{B_k(y)}{k!} x^k
\ee
with $B_0(y)=1$ and $B_1(y)=y-1/2$.

\subsection{Admissible series}
\label{sub.admissible}

The admissible series introduced by Kontsevich--Soibelman~\cite{KS:cohomological}
arise as generating series of Poincaré polynomials of Cohomological Hall Algebras.
The latter are algebra structures that one can define on the homology of a moduli
stack parametrising objects in abelian categories satisfying appropriate conditions.

By their very definition, the product of two admissible series is admissible, and
since the choice of the integers $c_{n,i}$ in~\eqref{defgamma} is arbitrary (as long
as they have finite support for each $n$), it follows that admissible series belong
to an uncountable group of formal power series. In contrast with $q$-series with
integer coefficients, which are typically defined only for $|q|<1$,
\be
\label{Fqqi}
F(t,q) \quad \text{is admissible if and only if} \quad F(t,q^{-1})
\quad \text{is admissible} \,.
\ee

Another property of admissible series $F(t,q)$ is that the ratio (which is often
considered in the literature) and the symmetrisation (which is less common, but
useful for us) 
\be
\label{GSdef}
G(t,q) \= F(qt,q)/F(t,q), \qquad F^\sym(t,q) \= F(t,q) F(t,q^{-1}) 
\ee
are integral i.e., both lie in $\BZ[q^{\pm 1}][\![t]\!]$. This follows easily
from equation~\eqref{defgamma} combined with the two identities
\be
(1-t)(q t;q)_\infty \= (t;q)_\infty, \qquad
(t;q^{-1})_\infty \= \frac{1}{(q t;q)_\infty}
\ee
in $\BQ(q)[\![t]\!]$, which follow from the known $t$-series expansions of both sides
or from the fact that they satisfy the same first order linear $q$-difference
equation. More precisely, equation~\eqref{defgamma} implies that
\be
\label{GS2}
G(t,q) \;=\!\! \prod_{0\neq n\in(\BZ_{\geq0})^N}
\prod_{i \in \BZ} (q^i t^n;q)_{n_1+\cdots+n_{N}}^{-c_{n,i}},
\qquad
F^\sym(t,q) \;=\!\! \prod_{0\neq n\in(\BZ_{\geq0})^N}
\prod_{i \in \BZ} (q^{1-i} t^n;q)_{2i-1}^{-c_{n,i}}
\,.
\ee
Admissible series initially appear to treat~$q$ near~$0$ or infinity. As already
mentioned, we can also
consider~$q$ near~$1$, or better yet, near a primitive $m$-th root of unity $\z_m$. 
Doing so, we are led to associate to an admissible series $F(t,q)$ a collection of
completed series $\Phitof(t,q)=(\Phitof_m(t,x))_{m \geq 1}$ defined for $q=\z_m+x$
via the very short definition~\eqref{Phidef}. 
The next lemma unravels this definition. Its proof follows directly from the
definition of the admissibility~\eqref{logF}, and among other things
explains the necessity for the rescaling $t^{1/m}$ of the $t$ (multi-)variable.
To state it, recall $L_{n}(q)$ from equation~\eqref{logF} and set

\begin{align}
\label{Vdef}
V(t) & \= \sum_{0\neq n\in(\BZ_{\geq0})^N} L_n(1) \Li_2(t^n) \inn \BQ[\![t]\!]\,,
\\
\label{tautdef}
\delta(t) & \= \exp\Big( \sum_{0\neq n\in(\BZ_{\geq0})^N} L_n(1) \Li_1(t^n)
-2 \sum_{0\neq n\in(\BZ_{\geq0})^N} L'_n(1) \Li_1(t^n) \Big) \inn\BQ[\![t]\!]\,,
\end{align}
with constant terms $V(0)=0$ and $\delta(0)=1$.
Furthermore, for every integer $m \geq 1$, we let $U_m(t)=e^{u_m(t)}$, where
\be
\label{Umdef}
\begin{aligned}
u_m(t)  \= &
- \frac{m-1}{2 m} \!\!\!\!\!\sum_{0\neq n\in(\BZ_{\geq0})^N}\!\!\! L_n(1) \Li_1(t^n)
+ \frac{1-m}{m} \!\!\!\!\!\sum_{0\neq n\in(\BZ_{\geq0})^N}\!\!\! L'_n(1) \Li_1(t^n)
\\ & \quad-\sum_{j=1}^{m-1} \sum_{0\neq n\in(\BZ_{\geq0})^N}\!\!
\frac{L_n(\z_m^j)}{1-\z_m^j}
\sum_{k\in\BZ/m\BZ}
\frac{\z_m^{-kj}}{m}
\Li_{1}(\z_m^{k}t^{n/m})
% \Li_1^{[j,m]}(t^{n/m})
\quad\in\; t\thinspace\BQ[\![t]\!]\,.
\end{aligned}
\ee

\begin{lemma}
\label{lem.Phimexp}
For every $m$, we have 
\be
\label{Phiexp}
\Phitof_m(t,x) \in \, e^{\frac{V(t)}{m^2 \log(1+x/\z_m)}} \frac{1}{\sqrt{\delta(t)}}
U_m(t)
(1 + x\, t^{1/m} \,\BQ[\z_m][\![t^{1/m}]\!][\![x]\!]), \qquad \Phitof_m(0,x) \=1 \,.
\ee
\end{lemma}

When $F(t,q)=F_A(t,q)$ is given by~\eqref{FAdef}, we will
denote the corresponding series $\Phitof_A(t,q)$, $V_A(t)$, $\delta_A(t)$ and
$U_{A,m}(t)$,
respectively.

\begin{proof}
To simplify the notation, we assume $N=1$.
Fix $m$, and observe that the right hand side of~\eqref{logF}
has a simple pole at $q=\z_m$ coming from the terms with $\ell$ divisible by $m$,
for which
\be
\Res{q=\z_m} \frac{L_n(q^{\ell})}{\ell\thinspace(1-q^{\ell})}
\= -\frac{L_n(1)\thinspace\z_m}{\ell^2} \,.
\ee
Setting $\ell=mr$ with $r \geq 1$, we see that the residue of the right hand
side of~\eqref{logF} is given by
\be
\sum_{n, r \geq 1} \frac{L_n(1)\thinspace \z_m}{m^2 r^2}\thinspace t^{n m r} \=
\frac{1}{m^2} \sum_{n \geq 1} L_n(1)\thinspace \Li_2(t^{n m})
\= \frac{1}{m^2} V(t^m) \,.
\ee
This explains the need for the rescaling $t \mapsto t^{1/m}$ in order to 
match with equation~\eqref{Phiexp}. The remaining terms of~\eqref{logF} 
can be expanded into power series in $t^{1/m}$ with coefficients in $\BQ[\z_m]$. 

When we compute the constant term of the $x$-series expansion of $\Phitof_m(t,x)$
there are two contributions to $u_m(t)$ %=u^{(0)}_m(t) + u^{(1)}_m(t)$
from $\log F(t,q)$ in~\eqref{logF}, depending whether or not $m$ divides $\ell$.
When $m$ divides $\ell$, with $q=\z_m+x$, we have

%% see Mathematica file: DTInvariants.Admissible.Series.nb
%% section: Constant term of admissible series F(t,q).

\be
\label{uptox}
- \frac{L_n(q^\ell)}{\ell(1-q^\ell)} \= \z_m \frac{L_n(1)}{\ell^2} \frac{1}{x} +
\Big( \frac{L_n(1)}{2\ell^2} - \frac{L_n(1)}{2 \ell} + \frac{L'_n(1)}{\ell} \Big)
+ \mathrm{O}(x)
\ee
Setting $m=r \ell$ and summing over $n \geq 1, r \geq 1$ and rescaling
$t$ to $t^{1/m}$ gives the volume term $V(t)/(2 m^2)$ and the first two terms
in equation~\eqref{Umdef}.

When $m$ does not divide $\ell$, then $\ell=rm+j$ for some $r \geq 0$ and
$1 \leq j \leq m-1$. Then, again with $q=\z_m+x$, we have
\be
- \frac{L_n(q^\ell)}{\ell(1-q^\ell)} \= -\frac{L_n(\z_m^j)}{(mr+j)(1-\z_m^j)}
+ \mathrm{O}(x) 
\ee
and after rescaling $t$, we find that this contribution is
\be
-\sum_{j=1}^{m-1} \sum_{n \geq1, \,\, r \geq 0}
\frac{L_n(\z_m^j)}{(mr+j)(1-\z_m^j)} t^{n(mr+j)/m} \=
-\sum_{j=1}^{m-1} \sum_{n \geq 1} \frac{L_n(\z_m^j)}{1-\z_m^j}
\sum_{k\in\BZ/m\BZ}
\frac{\z_m^{-kj}}{m}
\Li_{1}(\z_m^{k}t^{n/m})\,,
% \Li_1^{[j,m]}(t^{n/m})
\ee
which is the last term of equation~\eqref{Umdef}.
The lemma follows.
\end{proof}

Equations~\eqref{Vdef} and~\eqref{Umdef} imply the following corollary, which in
a sense is a gluing property of the collection of power series $\Phitof(t,q)$. 

\begin{corollary}
\label{cor.F12}
Let $F^{(1)}$ and $F^{(2)}$ be admissible series with the same number of variables
$t$, and $L_{n}^{(i)}$, $V^{(i)}(t)$, $\delta^{(i)}(t)$, and $U_{m}^{(i)}(t)$ $(i=1,2)$
be as defined before Lemma~\ref{lem.Phimexp}.
Then
\begin{itemize}
\item[(a)]
  $V^{(1)}=V^{(2)}$ if and only if
  $L_n^{(1)}(1)=L_n^{(2)}(1)$ for all $n \geq 1$.
\item[(b)]
  $F^{(1)}=F^{(2)}$ if and only if
  $V^{(1)}=V^{(2)}$, $\delta^{(1)}=\delta^{(2)}$ and $U_m^{(1)}=U_m^{(2)}$ for all
  $m \geq 1$.
\item[(c)]
  $F^{(1)}=F^{(2)}$ if and only $\Phitof^{(1)}_m = \Phitof^{(2)}_m$ for some (and hence
  for every) $m \geq 1$.
\end{itemize}
\end{corollary}

\begin{proof}
Only the if direction of (b) is not obvious. To show it, we note that
equation~\eqref{Umdef} and the fact that $\Li_s(t)=t + O(t^2)$ imply that $V(t)$
and $\delta(t)$ determine $L_n(1)$ and $L_n'(1)$ for all~$n$,
and hence with $U_m(t)$ determine
$s_m(t) := \sum_{j=1}^{m-1} \sum_{n \geq 1}
\frac{L_n(\z_m^j)}{1-\z_m^j}
\sum_{k\in\BZ/m\BZ}
\frac{\z_m^{-kj}}{m}
\Li_{1}(\z_m^{k}t^{n/m})
% \Li_1^{[j,m]}(t^{n/m})
$. Now use 
$\sum_{k\in\BZ/m\BZ}
\frac{\z_m^{-kj}}{m}
\Li_{1}(\z_m^{k}t^{n/m})
% \Li_1^{[j,m]}(t)
=t^{jn/m}/j+\cdots$ to deduce that the coefficient of $t^{1/m}$ in
$s_m(t)$ is $L_1(\z_m)/(1-\z_m)$ for all $m$. Thus $s_m(t)$ determines
$L_1(q) \in \BZ[q^{\pm 1}]$. Subtracting this first contribution from $s_m(t)$,
we find that the coefficient of $t^{2/m}$ is now $L_2(\z_m)/(2(1-\z_m))$ and hence
is again determined by $s_m(t)$. Continuing by induction on $n$, we deduce that
$s_m(t)$ determines $L_n(q)$ for all $n \geq 1$.
\end{proof}

We next discuss a Dwork-type quotient of admissible series for a prime $p$. 
The product formula~\eqref{defgamma} and equation~\eqref{logF} for admissible
series lead immediately to the equality
\be
\label{admis.dwork}
\log\Big(\frac{F(t^p,q^p)}{F(t,q)^p}\Big)
\= 
p\!\!\!\!\sum_{0\neq n\in(\BZ_{\geq0})^N}\sum_{\underset{(p,\ell)=1}{\ell>0}}
\frac{L_n(q^\ell)}{\ell(1-q^\ell)} t^{\ell n}\,.
\ee
Expanding near $q=\z_{m}+x$ as in~\eqref{Phidef} with $(m,p)=1$, we see that the
left-hand side, even after division by $p$, is $p$-integral.
(This strong integrality property will be important.) This proves:

\begin{lemma}
\label{lem.dwork.admis}
If $F(t,q)$ is admissible, then for all primes $p$ and positive integers $m$ not
divisible by~$p$ we have
\be
\label{dphiA1}
\log(F(t^{p/m},q^p)) -p\log(F(t^{1/m},q))
\inn \frac{p}{x} \BZ_{(p)}[t^{1/m},\z_{m}][\![t,x]\!]\,,
\qquad q\=\z_{m}+x\,.
\ee
\end{lemma}

\subsection{\texorpdfstring{Admissibility of
    $q$-hypergeometric series}{Admissibility of q-hypergeometric series}}
\label{sub.DT}

Our main interest is in the admissible series $F_A(t,q)$ defined by a
$q$-hypergeometric sum that is determined by an integral symmetric matrix $A$
as in~\eqref{FAdef}. These admissible series are in several respects special.
Most importantly for us, they are $q$-holonomic, i.e., satisfy the system
~\eqref{PhiAshift} of linear $q$-difference equations with respect to $t$ that,
together with the initial condition $F_A(0,q)=1$, uniquely determine $F_A(t,q)$.
The proof that $F_A(t,q)$ satisfies the system of equations~\eqref{PhiAshift}
follows easily from the fact that $F_A(t,q)$ is a sum of proper $q$-hypergeometric
series, and in more elementary terms, it is an easy consequence of the fact that
the $q$-Pochhammer symbol $(q;q)_n$ satisfies the relation
$(q;q)_{n+1}=(1-q^{n+1})(q;q)_n$ for all non-negative integers $n$. 

The condition for an admissible series to be $q$-holonomic is a delicate
intersection between additive and multiplicative properties, and for the case of
the series $F_A(t,q)$, the shape of the $q$-difference equations~\eqref{PhiAshift}
depends on the matrix $A$ and in fact determines it.

We next discuss the algebraic system of $t$-deformed Nahm equations~\eqref{zjt}
that we encountered in the introduction. It is clear that the latter has a
unique formal power series solution $z(t)\=(z_1(t),\dots,z_N(t))$ with
$z(0)=1 \in \BZ^N$, which in fact satisfies $z(t) \in \BZ[\![t]\!]$. But more is
true. Namely, $z(t)$ is given by a hypergeometric series, as was discovered by
Rodriguez Villegas~\cite[Sec. 4.1]{RV:Apoly} (see also~\cite[Sec. 7]{Z:diffeq}).
For example, in the rank one case
when $A \in \BZ$, the unique solution of the equation $1-z=t\,(-z)^A$ in
$\BZ[\![t]\!]$ with $z(0)=1$ is given by the hypergeometric
series~\cite[Eqn.(3.0.8)]{RV:Apoly}
\be
\label{zA}
z(t) \= \sum_{k=0}^{\infty} \frac{(-1)^{(A+1)k}\binom{Ak}{k}}{(A-1)k+1} t^k 
\ee
(as follows from Lagrange inversion) and more generally, $z(t)^s$ as well as
$\log(z(t))$ have hypergeometric series expansions~\cite[p.5]{RV:Apoly}
\be
\label{zAs}
z(t)^s \= s \sum_{k=0}^{\infty} 
\frac{(-1)^{(A+1)k}\binom{Ak+s-1}{k}}{(A-1)k+s}t^k, \qquad
\log(z(t)) \= \sum_{k=1}^{\infty} \frac{(-1)^{(A+1)k}\binom{Ak}{k}}{Ak} t^k
\,.
\ee
Note that $z(t)$ given in~\eqref{zA} is an algebraic function of $t$ in $\BC\BP^1$
with three singularities at $t=0,\infty, (-1)^A(A-1)^{A-1}/A^A$.

%% see Mathematica file: Discriminant.Nahm.Equation.nb

The involution~\eqref{Fqqi} for the admissible series $F_A(t,q)$ becomes the
identity
\be
\label{FAqqi}
F_A(t,q) \= F_{I-A}(t,q^{-1})\,,
\ee
which follows easily from the sum definition of $F_A(t,q)$, equation~\eqref{FAdef}
and the elementary fact that $(q^{-1};q^{-1})_n = (-1)^n q^{-n(n+1)/2}(q;q)_n$. 

\medskip
Next we give a simple proof that the $q$-hypergeometric series
$F_A(t,q)$ of equation~\eqref{FAdef} are admissible. For this we will need the following
lemma whose proof we defer to Section~\ref{sub.synthesis}, after
Theorem~\ref{thm.fgi.is.admis} of that section.

\begin{lemma}
\label{Vlem}
For every integral symmetric $N \times N$ matrix $A$,   
there exists a series $V(t) = V_A(t)\in\BQ[\![t]\!]$ such that for all $m\in\BZ_{>0}$
we have
\be
\log(F_{A}(t,\z_m+x)) \= \frac{\z_mV(t^m)}{m^2x}+O(x^0)\,.
\ee
\end{lemma}

This lemma essentially follows from the $q$-difference equations and the WKB
algorithm, or alternatively from the identification of the series $F_A(t,q)$ with
the one $F_A^\FGI(t,q)$ that comes from formal Gaussian integration.

The power series $V_A(t)$ is effectively computable (and is given explicitly in
equation~\eqref{Vt}), the example of $A\in\BZ_{>0}$ being given by the following:
\be
V_{A}(t) \= -\Li_{2}(1-z(t)) -\frac{A}{2}\log(z(t))^2
\= \sum_{k=1}^{\infty} \frac{(-1)^{(A+1)k}\binom{Ak}{k}}{Ak^2} t^k\,.
% \begin{aligned}
%   V_2(t) &\= -t + \frac{3}{4}t^2 - \frac{10}{9}t^3 + \frac{35}{16}t^4
%  - \frac{126}{25}t^5 + \frac{77}{6t}^6 - \frac{1716}{49}t^7 + \frac{6435}{64}t^8
% % +\frac{24310}{81}t^9
% +\mathrm{O}(t^{9})\,,\\
% V_3(t) &\=  t + \frac{5}{4} t^2 + \frac{28}{9} t^3 + \frac{165}{16} t^4
%  + \frac{1001}{25} t^5
% + \frac{1547}{9} t^6 + \frac{38760}{49} t^7 + \frac{245157}{64} t^8
% + \mathrm{O}(t^9)\,.\\
% \end{aligned}
\ee

Assuming Lemma~\ref{Vlem}, and using the $q$-holonomic system of equations
that $F_A$ satisfies, we can give an alternative proof of the admissibility of
$F_A(t,q)$. 

\begin{theorem}
\label{thm.DT}\cite{Efimov,KS:cohomological}
Suppose that $A$ is an integral symmetric $N \times N$ matrix. Then the unique
$c_{n,i}\in\BZ$ such that 
\be
F_{A}(t,q)
\= \sum_{n \in \BZ^N_{\ge0}} \frac{
(-1)^{\mathrm{diag}(A)\cdot n}
q^{\frac{1}{2}(n^t A n + \mathrm{diag}(A)\cdot n)} }{(q;q)_{n_1}
\dots (q;q)_{n_N}} t_1^{n_1} \dots t_N^{n_N}
\=
\prod_{0\neq n\in(\BZ_{\geq0})^N}
\prod_{i \in \BZ} (q^i t^{n};q)_\infty^{c_{n,i}}
\ee
have finite support, i.e. for fixed $n$ all but finitely many $i$ satisfy $c_{n,i}=0$.
\end{theorem}

\begin{proof}
To begin with, every $F(t,q) \in \BZ(\!(q)\!)[\![t]\!]$ satisfying $F(0,q)=1$
has a unique expansion of the form~\eqref{defgamma} for integers $c_{n,i}$, where for
every $n \in \BZ^N_{\geq 0}-\{0\}$, $c_{n,i}=0$ is zero for sufficiently negative
integers $i$. This follows easily by induction on the total degree of $q^it^n$.

So, it suffices to show that when $F=F_A$, for each fixed~$n$, there are only
finitely many non-zero~$c_{n,i}$. To make the idea clear, assume~$N=1$. Then~$F=F_A$
satisfies the linear $q$-difference equation
\be
F(t,q)-F(qt,q) \= (-q)^A t F(q^At,q)
\ee
Consider the ratio
\be
\label{FG}
G(t,q) \= \frac{F(qt,q)}{F(t,q)}
\= \prod_{n >0} \prod_{i \in \BZ} (q^i t^n;q)_n^{-c_{n,i}}\,.
\ee
It recovers $F$ by telescoping product $F(t,q)=G(t,q)^{-1} G(qt,q)^{-1} \dots $.
Moreover, $G$ satisfies the Ricatti equation
\be
1-G(t,q) \= (-q)^A t \prod_{j=0}^{A-1} G(q^j t,q) 
\ee
and it follows by induction on the powers of $t$ that
$G(t,q) \in \BZ[q^{\pm 1}][\![t]\!]$.
Therefore, we find as in equation~\eqref{logF} that
\be
  -\log(F(t,q))
  \=
  \sum_{\ell,n=1}^{\infty}
  \frac{L_{n}(q^{\ell})}{\ell(1-q^{\ell})}
  t^{n\ell}\,,
\ee
with $L_{n}(q)\in\frac{1-q}{1-q^{n}}\BZ[q^{\pm1}]$.
Suppose for induction that $L_{n}(q)\in\BZ[q^{\pm1}]$ for $n<N$.
This implies that the $V(t)$ from Lemma~\ref{Vlem} is given by
\be
\label{VVV}
V(t) \= \sum_{n=1}^{N-1} L_{n}(1)\Li_{2}(t^{n})+O(t^N)\,.
\ee
If $L_{N}(q)$ does not belong to $\BZ[q^{\pm1}]$, it must have a pole at $\z_{a}$
for some $a|N$
and $a>1$ with residue $\alpha\neq0$. Notice that for $ab=N$ this would imply that
the coefficient of $x^{-1}$ in $\log(F(t,\z_a+x))$ would be equal to
\be
  \frac{\z_a}{a^2}\sum_{n=1}^{N-1}
  L_{n}(1)\Li_{2}(t^{an})+\frac{\alpha}{1-\z_a}t^{ab}+O(t^{N+1})\,.
\ee
This contradicts Lemma~\ref{Vlem} and completes the proof for $N=1$.
For the case of general $N \geq 1$, $F(t,q)$ has an expansion of the
form~\eqref{defgamma}, where for each fixed $n \in \BZ^N_{\geq 0}-\{0\}$, $c_{n,i}=0$
is zero for sufficiently negative integers $i$. Define
\be
\label{Gj}
G_{j}(t,q) \= \frac{F(t_1,\cdots,qt_j,\cdots,t_N,q)}{F(t,q)}
\=
\prod_{0\neq n\in(\BZ_{\geq0})^N}
\prod_{i \in \BZ} (q^i t^{n};q)_{n_j}^{-c_{n,i}}, \qquad j\=1,\dots, N \,.
\ee
Since $F$ satisfies~\eqref{PhiAshift}, it follows that the power series $G_j(t,q)$
for $j=1,\dots,N$ satisfy the system of equations
\be
\label{eq:riccati}
1-G_{j}(t,q) \= (-q)^{A_{j,j}} t_{j} \prod_{k=0}^{A_{j,j}-1}G_{j}
\Big(\s_j^k\prod_{\underset{i\neq j}{i=1}}^N \s_i^{A_{i,j}}t,q\Big)\,,
\ee
where $\prod_{k=0}^{-m}a_k=a_{-1}^{-1}\cdots a_{-m}^{-1}$ for $m>0$. This system
of equations has a unique solution expanded in power series in $t$ 
and we claim that $G_j(t,q) \in \BZ[q^{\pm 1}][\![t]\!]$. This follows easily by
induction. Then we find that $\log(F(t,q))$ has the expression in equation~\eqref{logF}
with $L_{n}(q)\in\bigcap_{i=1}^{N}\tfrac{1-q}{1-q^{n_i}}\BZ[q^{\pm1}]$.
Then applying a similar contradiction argument as for $N=1$, and again using
Lemma~\ref{Vlem},
we can conclude the proof.
% Then from equation~\eqref{eq:riccati}, we see that
% $G_j(t,q)\in1+t\BZ[q^{\pm1}][t]+t_1^{m+1}\cdots t_N^{m+1}\BZ[\![q,t]\!]$.
% This, together with~\eqref{Gj} implies that for every ${0\neq n\in(\BZ_{\geq0})^N$
% with $n_j \neq 0$, $c_{n,\ell}$ is non-zero for finitely many $\ell$. Since for every
% ${0\neq n\in(\BZ_{\geq0})^N$ there exists $j$ such that $n_j\neq0$, the result follows. 
\end{proof}

\subsection{\texorpdfstring{Level $m$ admissible series}{Level m admissible series}}
\label{sub.madmissible}

The proof of Theorem~\ref{thm.2} of Section~\ref{sub.results} requires an extension
of admissibility, which is modelled on the congruence sums of equation~\eqref{FAmdef}
and equation~\eqref{FGIcong}. We introduce this notion here and then prove that
the above-mentioned sums give examples of level $m$ admissible series. This will
be crucial in proving the integrality of Theorem~\ref{thm.FGI2} of
Section~\ref{sub.results}.

\medskip

To state the definition we need the following:

\begin{lemma}
\label{lem.madm}
Every $F(t,q)\in1+t\,\BZ[\tfrac{1}{m}](\!(q)\!)[\![t]\!]$ can be written uniquely
in the form
\be
\label{madmis}
F(t,q) \= \exp\Big(-\sum_{n=1}^{\infty}
\sum_{\underset{(\ell,m)=1}{\ell=1}}^{\infty}
\frac{L_{n}(q^{\ell})}{\ell(1-q^{m\ell})} t^{n\ell}\Big)
\ee
with $L_{n}(q)\in\BZ[\tfrac{1}{m}](\!(q)\!)$, and conversely. 
\end{lemma}

\begin{proof}
The existence and uniqueness of $L_{n}(q)\in\BQ(\!(q)\!)$ that
satisfy~\eqref{madmis} is clear by removing one $L_{n}(q)$ at a time and induction
on powers of $t$. To show that $L_{n}(q)\in\BZ[\tfrac{1}{m}](\!(q)\!)$, we use
the fact that
\be
\label{eq:building.zmadmis}
\exp\bigg(\!-\sum_{\underset{(\ell,m)=1}{\ell=1}}^{\infty}
\frac{q^{k\ell}}{\ell(1-q^{m\ell})} t^{n\ell}\bigg)
\inn 1+t\,\BZ[\tfrac{1}{m}](\!(q)\!)[\![t]\!]\,,
\ee
which itself follows from the fact that by inclusion/exclusion on the divisors of $m$,
equation~\eqref{eq:building.zmadmis} can be written as a product of Pochhammer
symbols raised to the power of $\tfrac{1}{m}$. Indeed, we find that
\be
\exp\bigg(\!-
\sum_{\underset{(\ell,m)=1}{\ell=1}}^{\infty}
\frac{q^{k\ell}}{\ell(1-q^{m\ell})}
t^{n\ell}\bigg)
\=
\prod_{d|m}(q^{dk}t^{dn};q^{dm})_{\infty}^{\mu(d)/d}
\inn
1+t\,\BZ[\tfrac{1}{m}](\!(q)\!)[\![t]\!]\,,
\ee
where $\mu$ is the Möbius function.
\end{proof}

\begin{definition}
\label{def.zmadmis}
A series $F(t,q)\in1+t\,\BZ[\tfrac{1}{m}](\!(q)\!)[\![t]\!]$ is called level
$m$ admissible (and abbreviated by $m$-admissible) if, for the unique
$L_{n}(q)\in\BZ[\tfrac{1}{m}](\!(q)\!)$ from equation~\eqref{madmis}, we have
\be
L_{n}(q)
\inn\BZ[\tfrac{1}{m},q^{\pm1},\Phi_{d}(q)^{-1}
\,|\,d\not\equiv0\pmod{m}]
\qquad\text{and}\qquad
L_n(\z_m)\inn\BZ[\tfrac{1}{m}]\,,
\ee
where $\Phi_{d}$ denotes the $d$-th cyclotomic polynomial.
\end{definition}

Of course level $1$ admissible series are simply the admissible series of
Kontsevich--Soibelman. As in Section~\ref{sub.admissible}, we can expand the
logarithm of a $m$-admissible series for $q$ near a $\z_{mn}$ root of unity for any
$n\in\BZ_{>0}$. This will be a Laurent series with simple polar part given by a
series
\be
\label{eq:Vt.m}
\frac{V(t)}{m^2} \= \sum_{n=1}^{\infty} \sum_{\underset{(\ell,m)=1}{\ell=1}}^{\infty}
\frac{L_{n}(\z_m^{\ell})}{m\ell^2} t^{n\ell} \inn\BQ[\![t]\!]\,.
\ee
With this definition we can state the generalisation to $m>1$ of Lemma~\ref{Vlem}
of Section~\ref{sub.DT} and Theorem~\ref{thm.DT} of Section~\ref{sub.DT} for the
congruence sums $F_{A,m,k}$ in equation~\eqref{FAmdef}. 

\begin{lemma}
\label{Vlem.m}
Let $A$ be a symmetric integral $N \times N$ matrix and $V(t)=V_A(t)\in\BQ[\![t]\!]$
as in Lemma~\ref{Vlem} of Section~\ref{sub.DT}. Then for all $m>0$
and all $k\in\{0,\dots,m-1\}^N$, we have that
\be
\log(F_{A,m,k}(t,\z_m+x)) \= \frac{\z_mV(t^m)}{m^2x}+O(x^0)\,.
\ee
\end{lemma}

This lemma will be proved after Theorem~\ref{thm.fgi.is.admis} of
Section~\ref{sub.synthesis}. Assuming this lemma, we have the following:

\begin{theorem}
\label{thm.zm.admis}
Fix an integral symmetric $N \times N$ matrix $A$, a positive integer $m$ and
a residue class $k\in\{0,\dots,m-1\}^N$. Then the series
\be
F_{A,m,k}(t^{1/m},q)\inn1+t\,\BZ[\tfrac{1}{m}](\!(q)\!)[\![t]\!]
\ee
is $m$-admissible.
\end{theorem}

\begin{proof}
We will give the details of the proof when $N=1$, while for $N>1$ we use
similar methods to the proof of Theorem~\ref{thm.DT} of Section~\ref{sub.DT}.
The series
\be
H_{A,m,k}(t,q) \= \frac{(-1)^{Ak} q^{\frac{1}{2}Ak(k+1)} }
{(q;q)_{k}}t^{k}F_{A,m,k}(t,q)
\ee
satisfies the linear $q$-difference equation
\be
\sum_{\ell=0}^{m}(-1)^{\ell}q^{-\ell(\ell-1)/2}
\binom{m}{\ell}_{\!\!q^{-1}}H_{A,m,k}(q^{\ell}t,q)
\=
q^{Am(m+1)/2}t^{m}H_{A,m,k}(q^{Am}t,q)\,,
\ee
where $\binom{m}{\ell}_{\!q^{-1}}
=\tfrac{(q^{-1};q^{-1})_m}{(q^{-1};q^{-1})_\ell(q^{-1};q^{-1})_{m-\ell}}$ is the
$q^{-1}$-binomial coefficient. Consider the ratio
\be
G_{A,m,k}(t,q)
\=
\frac{H_{A,m,k}(qt,q)}{H_{A,m,k}(t,q)}
\=
q^k+\sum_{n=1}^{\infty}a_{n}(q)t^{mn}\,.
\ee
This satisfies the non-linear equation $q$-difference equation
\be
\sum_{\ell=0}^{m}(-1)^{\ell}q^{-\ell(\ell-1)/2}\binom{m}{\ell}_{\!\!q^{-1}}
\prod_{j=0}^{\ell-1}G_{A,m,k}(q^{j}t,q)
\=
q^{Am(m+1)/2}t^{m}\prod_{j=0}^{Am-1}G_{A,m,k}(q^{j}t,q)\,.
\ee
This gives an integral recursion for $a_{n}(q)$ multiplied by
\be
\begin{aligned}
\sum_{\ell=0}^{m}(-1)^{\ell}q^{-\ell(\ell-1)/2}\binom{m}{\ell}_{\!\!q^{-1}}
\sum_{j=0}^{\ell-1}q^{nmj}
\=
\sum_{\ell=0}^{m}(-1)^{\ell}q^{-\ell(\ell-1)/2}\binom{m}{\ell}_{\!\!q^{-1}}
\frac{1-q^{nm\ell}}{1-q^{nm}}\\
\=
\frac{(q^{1-m};q)_{m}-(q^{nm+1-m};q)_{m}}{1-q^{nm}}
\=
-(q^{nm+1-m};q)_{m-1}\,.
\end{aligned}
\ee
It follows that
\be
-(q^{nm+1-m};q)_{m-1}a_{n}(q)\inn\BZ[q^{\pm1},a_1(q),\dots,a_{n-1}(q)]\,.
\ee
Notice that $-(q^{mn+1-m};q)_{m-1}$, never contains $\Phi_{d}(q)$ with
$d\equiv0\pmod{m}$ as a factor and therefore, by induction, we see that
\be
G_{A,m,k}(t,q)
\inn
q^{k}+t^{m}\BZ[q^{\pm1},\Phi_{d}(q):d\not\equiv0\!\!\!\pmod{m}][\![t^{m}]\!]\,.
\ee
Therefore, solving for the unique Laurent series $L_{n}(q)$ such that
\be
\log(q^{-k}G_{A,m,k}(t,q))
\=
\sum_{n=1}^{\infty}
\sum_{\underset{(\ell,m)=1}{\ell=1}}^{\infty}
\frac{L_{n}(q^{\ell})(1-q^{nm\ell})}{\ell(1-q^{m\ell})}
t^{nm\ell}\,,
\ee
we find that $L_{n}(q)\in\frac{1-q^{m}}{1-q^{nm}}
\BZ[\tfrac{1}{m},q^{\pm1},\Phi_{d}(q)\,|\,d\not\equiv0\pmod{m}]$. It follows
that
\be
\log\big(F_{A,m,k}(t,q)\big) \=
-\sum_{n=1}^{\infty} \sum_{\underset{(\ell,m)=1}{\ell=1}}^{\infty}
\frac{L_{n}(q^{\ell})}{\ell(1-q^{m\ell})} t^{nm\ell}\,.
\ee
Suppose by induction that for $n<n_0$, for some $n_0$, we have $L_{n}(q)\in
\BZ[\tfrac{1}{m},q^{\pm1},\Phi_{d}(q):d\not\equiv0\pmod{m}]$ (the case $n_0=1$
being obvious). If $L_{n_0}(q)\notin
\BZ[\tfrac{1}{m},q^{\pm1},\Phi_{d}(q):d\not\equiv0\pmod{m}]$ it must have a pole
at $q=\z_{am}$ for some $1<a$. From Lemma~\ref{Vlem.m}, the residues at $\z_c$
are given universally by $\z_c\thinspace V(t^{c})\thinspace c^{-2}$ for all
$c\in m\BZ_{>0}$ and we therefore find that
\be
V(t^m)\=m\sum_{n=1}^{\infty}
\sum_{\underset{(\ell,m)=1}{\ell=1}}^{\infty}
\frac{L_{n}(\z_m^{\ell})}{\ell^2}
t^{nm\ell}\,,
\ee
and so
\be
V(t^{am})\=-am\sum_{n=1}^{\infty}
\sum_{\underset{(\ell,m)=1}{\ell=1}}^{\infty}
\frac{L_{n}(\z_m^{\ell})}{\ell^2}
t^{amn\ell}+O(t^{amn_0})\,.
\ee
Assuming that $L_{n_0}(q)$ has a pole at $\z_{am}$ would imply that the logarithm
expanded at $q=\z_{am}+x$ would have an additional residue containing a power $t^{n_0}$.
Given that $V(t^{am})$ is the residue of $\log\big(F_{A,m,k}(t,q)\big)$ at $q=\z_{am}$
from Lemma~\ref{Vlem.m}, and $V(t^{am})$ is a power series in $t^{am}$, we see that
this must imply that $n_0$ is a multiple of $am$.
In particular, we find that $n_0=abm$ for some integer $0<b<n_0$.
However, we see that this would also change the coefficient of $t^{abm}$, which would
contradict the fact that the residue is $V(t^{am})$. This completes the proof.
\end{proof}

We are especially interested in the integrality that is implied by admissibility.
With this in mind, we state a generalisation of Lemma~\ref{lem.dwork.admis} of
Section~\ref{sub.admissible} to the case of $m$-admissible series.

\begin{lemma}
\label{lem.dwork.zm.admis}
Fix an $m$-admissible series $F(t,q)$. For all primes $p$ and positive integers
$m'$ with $(mm',p)=1$, we have
\be
\label{dphiA1m}
\log(F(t^{p/m'},q^p)) -p\log(F(t^{1/m'},q))
\inn \frac{p}{x} \BZ_{(p)}[t^{1/mm'},\z_{mm'}][\![t,x]\!]\,,
\qquad q\=\z_{mm'}+x\,.
\ee
\end{lemma}

\begin{proof}
Notice that
\be
\begin{aligned}
\frac{F(t^p,q^p)}{F(t,q)^p}
&\=
\exp\Big(p\sum_{n=1}^{\infty}
\sum_{\underset{(\ell,m)=1}{\ell=1}}^{\infty}
\frac{L_{n}(q^{\ell})}{\ell(1-q^{m\ell})}
t^{j\ell}
-
p\sum_{n=1}^{\infty}
\sum_{\underset{(\ell,m)=1}{\ell=1}}^{\infty}
\frac{L_{n}(q^{p\ell})}{p\ell(1-q^{mp\ell})}
t^{jp\ell}\Big)\\
&\=
\exp\Big(p\sum_{n=1}^{\infty}
\sum_{\underset{(\ell,mp)=1}{\ell=1}}^{\infty}
\frac{L_{n}(q^{\ell})}{\ell(1-q^{m\ell})}
t^{j\ell}\Big)\,.
\end{aligned}
\ee
Then expanding this sum at $q=\z_{mm'}+x$ concludes the proof.
\end{proof}

\subsection{Formal Gaussian integration}
\label{sub.FGI}

In this section we review a collection of formal power series at roots of unity
that were defined by formal Gaussian integration in~\cite{DG} for $\z=1$ and in
~\cite{DG2} for arbitrary complex roots of unity $\z$. The input to define these
power series was a Neumann--Zagier datum, which, in the case of perturbative complex
Chern--Simons theory, comes from an ideal triangulation of a 3-manifold and a solution
of its gluing equations.
For the simpler case of a single matrix $A$ and the Nahm equations, and for
$\zeta=1$, the same series were found in~\cite{Zagier:dilog}.
We will tailor the definition of
these series for the purpose of our paper and match them with the associated series
$\Phitof_A(t,q)$ of the admissible series $F_A(t,q)$ of~\eqref{FAdef}. The most
important point here is that formal Gaussian integration gives series with
coefficients that are manifestly algebraic functions.

We now recall the definition of $\Phitof_{A,m}$. Fix a symmetric $N \times N$ matrix
$\Lambda$ with integer entries. For a $\BQ$-algebra $R$, formal Gaussian integration
is a map $\langle \cdot \rangle: R[\![w,w^3h^{-1},h]\!] \to R[\![ h]\!]$, where
$w=(w_1,\dots,w_N)$, defined by
\be
\label{eq:bracket}
\left\langle
f(w,h)
\right\rangle_{\Lambda}
\;:=\; \exp\left(
{\frac{h}{2}\sum_{i,j=1}^N (\Lambda^{-1})_{i,j}
\frac{\partial}{\partial w_i}
\frac{\partial}{\partial w_j}}\right)
f(w,h) \Big|_{w=0}
\inn R[\![h]\!]\,.
\ee
This formally recovers the integral
\be
\frac{\int e^{-\frac{1}{2h}w^t \Lambda \,w}f(w,h) \, dw}{
\int e^{-\frac{1}{2h}w^t \Lambda \,w} \, dw}
\inn R[\![h]\!]\,.
\ee
Starting with the admissible series $F_{A}(t,q)$ of equation~\eqref{FAdef} and taking
$q=\z_me^{h}$, summing over congruences modulo $m$, and applying Poisson
summation assuming that one term dominates, leads formally to the integral
\be
\label{eq:inform.int}
\begin{aligned}
&\frac{1}{(q;q)_{\infty}^{N}}\sum_{k\in(\BZ/m\BZ)^N}
(-1)^{\mathrm{diag}(A)\cdot k}q^{\frac{1}{2}(k^t A k
+ \mathrm{diag}(A)\cdot k)}t_1^{k_1} \cdots t_{N}^{k_N}
\int \prod_{j=1}^{N}(q^{k_j+1}e^{w_j};q)_{\infty}\\
&\quad\times \exp\Big(\frac{1}{h}\Big(
\frac{\diag(A)\cdot w\pi i}{m}+\frac{1}{2}w^tAw+w\cdot\log(t)\Big)
+(\frac{\diag(A)}{2}+k^{t}A)\cdot w\Big)
\frac{dw}{(mh)^{N}}\,.
\end{aligned}
\ee
To define the integral in equation~\eqref{eq:inform.int} precisely, we use the
same coordinates as we did previously for the Habiro ring 
\be
q\=\z_me^{h}\=\z_m+x\,,\qquad
h\=\log(1+x/\z_m)\,.
\ee
together with the expansion of the infinite Pochhammer symbol (see~\eqref{logpocm}) 
from which we remove four terms
\be
\label{Psikdef}
\begin{small}
\begin{aligned}
&\psi_{k,z,\z_m}(w,x) \= 
((\z_m+x)^{k+1}ze^w;\z_m+x)_{\infty} \cdot 
\exp\Big(-\frac{\Li_{2}(z^{m})}{m^2\log(1+x/\z_m)}\\
&-\frac{\Li_{1}(z^{m})}{m^2\log(1+x/\z_m)} w 
 -\frac{\Li_{0}(z^{m})}{m^2\log(1+x/\z_m)}w^2
+ \sum_{\ell=0}^{m-1}
\big(\frac{1}{2}-\frac{k+\ell+1}{m}\big)\log(1-\z_m^{k+\ell+1}z) \Big)\,,
\end{aligned}
\end{small}
\ee
where $\psi_{k,z,\z_m}(w,x)\in1+x\BQ(z,\z_m)[\![w,w^3x^{-1},x]\!]$.
The integral~\eqref{eq:inform.int} is dominated by the leading term given by
the exponential of $h^{-1}$ times the function
\be
\label{eq:Vgen}
\frac{1}{2}w^tAw +\sum_{j=1}^{N} -\frac{\Li_{2}(1-e^{mw_j})}{m^2}
+\frac{w_j}{m}\log\Big(\frac{(-1)^{A}t_j^m}{1-e^{mw_j}}\Big) \,,
\ee
whose critical points are exactly $1/m$ times the logarithm of $z$, where
$z=z(t)\in (\BZ[\![t]\!])^N$ is the unique solution to the equations~\eqref{zjt}.
The critical values of equation~\eqref{eq:Vgen} are given by
\be 
\label{Vt}
V^\FGI(t) % \= \int_0^t \frac{\log g(s)}{s} ds
\= -\sum_{j=1}^N \Li_2(1-z_j(t))
- \frac{1}{2} \sum_{i,j=1}^N A_{ij} \log(z_i(t))\log(z_j(t)) \inn \BQ[\![t]\!],
\qquad V^\FGI(0)\=0 \,.
\ee
The above discussion leads to the following definition.

\begin{definition}
\label{def.FGI}
We let:
\be
\label{PhiFGIdef}
\Phitof_{A,m}^\FGI(t,x) \= \sum_{k \in (\BZ/m\BZ)^N}  I_{A,m,k}(t^{1/m},x)\,,
\ee
where $I_{A,m,k}(t,x)$ is given by a formal Gaussian integral
\be
\label{Ikdef}
\begin{small}
\begin{aligned}
&I_{A,m,k}(t,x)
\=\frac{(-1)^{\mathrm{diag}(A)\cdot k}
  q^{\frac{1}{2}(k^t A k + \mathrm{diag}(A)\cdot k)}t_1^{k_1}
\cdots t_{N}^{k_N}
\exp\Big(\frac{V(t^{m})}{m^2\log(1+x/\z_m)}\Big)}{\sqrt{m^N
  \det(-\Lambda(t^m))\prod_{j=1}^{N}(1-z_{j}(t^{m})^{\frac{1}{m}})}}\\
&\times\prod_{j=1}^{N}\Big(
\prod_{\ell=1}^{m-1-k_j}
\Big(\frac{1-\z_m^{k_{j}+\ell}z_{j}(t^{m})^{\frac{1}{m}}}{1-\z_m^{\ell+k_j}}
\Big)^{\frac{1}{2}-\frac{k_{j}+\ell}{m}}\!\!\!\!
\prod_{\ell=m+1-k_j}^{m}
\Big(\frac{1-\z_m^{k_{j}+\ell}z_{j}(t^{m})^{\frac{1}{m}}}{1-\z_m^{\ell+k_j}}
\Big)^{\frac{1}{2}-\frac{k_{j}+\ell}{m}}
(1-\z_m^{\ell+k_j})\Big)\\
&\qquad\times\Big\langle
\exp\Big(\Big(k^{t}A+\frac{1}{2}\mathrm{diag}(A)\Big)
\Big(w+\frac{1}{m}\log z(t^{m})\Big)\Big)
\prod_{j=1}^{N}
\psi_{k_j,z_{j}(t^{m})^{\frac{1}{m}},\z_m}(w_j,x)
\Big\rangle_{\Lambda(t^{m})}
\end{aligned}
\end{small}
\ee
with
\be
\label{lt}
\Lambda(t)\=-A-\mathrm{diag}\Big(\frac{z(t)}{1-z(t)}\Big) \,.
\ee
\end{definition}
An important and non-trivial property is that $I_{A,m,k}(t,x)$ is $m$-periodic
in $k$, so that equation~\eqref{PhiFGIdef} makes sense. This follows from the
definition of $I_{A,m,k}(t,x)$ and a simple change of coordinates in the Gaussian
integration as done in~\cite{AarhusII} or~\cite{GSW}.

Note that the exponential prefactor of all $I_{A,m,k}(t,x)$ is independent of $k$.
Said differently, these formal Gaussian integrals are equi-peaked Gaussians.
Consider
\be
\label{dt}
\delta^\FGI(t)\= \prod_{j=1}^Nz_j(t)^{-A_{jj}}(1-z_j(t)) \det(-\Lambda(t))\,,\qquad
\delta^\FGI(0)\=1 \,.
\ee
and 
\be
\begin{aligned}
\label{Umtdef}
U_m^\FGI(t) &\= \frac{1}{m^N}
\prod_{j=1}^{N}\frac{(1-z_{j})D_{\z_m}(1)}{(1-z_{j}^{\frac{1}{m}})
D_{\z_m}(z_{j}^{\frac{1}{m}})}
\sum_{k \in (\BZ/m\BZ)^N}
\frac{
\z_m^{\frac{1}{2}\mathrm{diag}(A)\cdot k}
t^{\frac{k}{m}}z^{\frac{Ak}{m}+\diag(A)(\frac{1}{2m}-\frac{1}{2})}}
{(-\z_m)^{-\frac{1}{2}k^{t}Ak}\prod_{j=1}^{N}(\z_mz_{j}^{\frac{1}{m}};\z_m)_{k_j}}\,,
\end{aligned}
\ee
where $z_j=z_j(t)$ and 
\be\label{cycdl}
D_{\z_m}(z)
\=
\prod_{\ell=1}^{m-1}
(1-\z_m^{\ell}z)^{\frac{\ell}{m}}\,.
\ee
Note that~\eqref{Umtdef} is a well-defined power series in $t$ and $U_m^\FGI(0) \=1$.
Recall the ring $S^{(m)}$ defined in~\eqref{RA} and let
$S^{(m)}_\BQ=S^{(m)} \otimes \BQ$. 

\begin{lemma}
\label{lem.PhiA}
For every positive integer $m$ we have: 
\be
\label{eq.PhiA}
\log \Phitof^\FGI_{A,m}(t,x) \inn \frac{V^\FGI(t)}{m^2\log(1+x/\z_m)}
-\frac{1}{2} \log \delta^\FGI(t) + \log U_m^\FGI(t) + x S^{(m)}_\BQ[\![x]\!] 
\ee
and $\delta^\FGI(t) \in S$ and $m^{N m} U_m^\FGI(t)^{2m} \in S^{(m)}$. 
\end{lemma}

\begin{proof}
A priori, the coefficients of $x^k$ for $k>1$ in the LHS of equation~\eqref{eq.PhiA}
are in the bigger ring $\BZ[\z_m,t^{\pm 1/m}, z^{\pm 1/m}, \tfrac{1}{\delta}]/(1-z-(-1)^A
t z^A)$. The endomorphism $\gamma_j$ that sends $z_j(t)^{1/m}$ to $\z_m z_j(t)^{1/m}$
and fixes $S^{(m)}$ satisfies
\be
\gamma_j I_{A,m,k}(t,x) \= I_{A,m,k+\delta_j}(t,x) \,.
\ee
This follows from applying $\gamma_j$ and also a change of variables to sending
$w \mapsto w + h$. Therefore, the coefficients in the LHS of equation
~\eqref{eq.PhiA} is invariant under $\gamma_j$ and hence an element of~$S^{(m)}$.
For $U_{m}^\FGI(t)^{2m}$, we see that it is in $m^{-2N}S^{(m)}$. To see
that a factor of $m$ cancels we use
\be
D_{\z_m}(1)^{24m} \= m^{12m}\,,
\ee
which follows from properties of the multiplier system of the Dedekind
$\eta$-function~\cite{Rademacher}.
\end{proof}

We will now improve on the sets where the coefficients of the $x$-series expansions
lie. For
$m,m'\in\BZ_{>0}$ with $m'$ prime to $m$ and $k\in(\BZ/m\BZ)^{N}$ we define
\be
\label{FGIcong}
\CS_{A,m,k}(t,q) \;=\!\!\!\!
\sum_{\underset{\ell\equiv_m k}{\ell\in(\BZ/mm'\BZ)^{N}}}
\frac{(q;q)_{k_1}\cdots(q;q)_{k_N}I_{A,mm',\ell}(t^{1/m},x)}
{(-1)^{\mathrm{diag}(A)\cdot k}q^{\frac{1}{2}(k^t A k + \mathrm{diag}(A)\cdot k)}
t_1^{k_1/m} \cdots t_{N}^{k_N/m}} \,,\quad q\=\z_{mm'}+x\,.
\ee
Notice that when $m=1$ we have
\be
\CS_{A,1,0}(t,\z_{m'}+x) \= \Phitof_{A}^{\FGI}(t^{m'},\z_{m'}+x)\,.
\ee
The sets that the coefficients of the expansions of $\log(\CS_{A,m,k})$ live in
is improved by Dwork-like quotients. In particular the pole and the constant
term in the expansion~\eqref{Ikdef} becomes better.

\begin{lemma}
\label{lem:dwork.diff.VU} 
For all primes $p$ and positive integers $m$ with $(m,p)=1$, we have 
\be
\label{dphiA2}
\log(\CS_{A,m,k}(t^p,q^p)) -p\log(\CS_{A,m,k}(t,q))
\inn x^{-1}S^{(m)}_p[\tfrac{1}{p},z^{1/m}] [\![x]\!]\,,\quad
  q\=\z_{m}+x\,.
\ee
\end{lemma}

\begin{proof}
To prove the result we need to check two conditions. For $V(t)$ of equation
~\eqref{Vt} and a single term in the sum $U_{m}(t)$ of equation~\eqref{Umtdef}
\be
U_{m,k}(t) \=
z^{\frac{Ak}{m}+\diag(A)(\frac{1}{2m}-\frac{1}{2})}
\prod_{j=1}^{N}\frac{(1-z_{j})D_{\z_{m}}(1)}{m(1-z_{j}^{\frac{1}{m}})
D_{\z_{m}}(z_{j}^{\frac{1}{m}})} \frac{(\z_{m};\z_{m})_{k_j}}
{(\z_{m}z_{j}^{\frac{1}{m}};\z_{m})_{k_j}}\,,
\ee
we need to show that
\be
\label{eq:V.dwork}
\frac{V(t^p)}{p}-pV(t)
 % \=
 % \frac{1}{p}\fr\Li_{2}(tz^A)-p\Li_{2}(tz^A)
 % -\frac{1}{2p}\fr\log(z)^tA\log(z)+\frac{p}{2}\log(z)^tA\log(z)
\inn pS^{(1)}_{p}\,
\ee
and
\be
\log(\delta(t^p)^{-1}U_{m,k}(t^p)^2) -p
\log(\delta(t)^{-1}U_{m,k}(t)^2) \in pS^{(m)}_{p}[z^{1/m}] \,.
\ee
Equation~\eqref{Vt} implies that
\be
\label{pe0}
\frac{V(t^p)}{p}-pV(t)
\=
\frac{1}{p}\fr\Li_{2}(tz^A)-p\Li_{2}(tz^A)
-\frac{1}{2p}\fr\log(z)^tA\log(z)+\frac{p}{2}\log(z)^tA\log(z) \,.
\ee
Since $z$ is a unit, there exists $\eta=\eta(z)\in S^{(1)}_p$ such that
$\fr(z)\=z^p e^{p \eta}$. We have
\be
\label{pe1}
\begin{aligned}
\fr\Li_{2}(tz^A)-p^2\Li_{2}(tz^A)
&\=
\Li_{2}(t^pz^{pA}e^{p\eta})-p^2\Li_{2}(tz^A)\\
&\inn
\Li_{2}^{(p)}(t^pz^{pA})
+\Li_{1}(t^pz^pz^{pA})^tpA\eta
+S^{(1)}_p\,,
\end{aligned}
\ee
which follows from $
\Li_{n}(ze^x)\=\sum_{k=0}^{\infty}\frac{x^k}{k!}\Li_{n-k}(z)\,,
$
and $\Li_{n}(z)\in\BZ[z,(1-z)^{-1}]$ for $n\leq0$. Moreover, we have
\begin{align}
\label{pe2}
\fr\log(z)^tA\log(z)-p^2\log(z)^tA\log(z)
& \=
2p\log(z^p)^tA\eta
+p^2\eta^tA\eta \\
\label{pe3}
A \Li_{1}(t^pz^pz^{pA})+ A \log(z^p)
& \=
\Li_1((1-z)^p)-p\Li_{1}(1-z)A
\=
pA \Li_{1}^{(p)}(1-z)\,.
\end{align}
Combining~\eqref{pe1}, ~\eqref{pe2} and ~\eqref{pe3} with
$\Li_{n}^{(p)}(1-z)\in S^{(1)}_{p}$, we deduce that the right hand side of
equation~\eqref{pe0} is also an element of $S^{(1)}_{p}$. Finally, as
$\delta(t)^{-m}U_{m,k}(t)^{2m}$ is a $p$-unit there is a $\beta\in S^{(1)}_p$ such that
$\fr(\delta(t)^{-m}U_{m,k}(t)^{2m})\=\delta(t)^{-mp}U_{m,k}(t)^{2mp} e^{p \beta}$, and a
similar argument can be used to deduce the remaining statements. 
\end{proof}

\begin{remark}
If $U_{m}^{\FGI}(t)$ is a $p$-unit then we can lift the statement of
equation~\eqref{dphiA2} to
\be
\log(\Phitof^{\FGI}_{A}(t^p,q^p)) -p\log(\Phitof^{\FGI}_{A}(t,q))
\inn x^{-1}S^{(m)}_p[\tfrac{1}{p}] [\![x]\!]\,,\quad
q\=\z_{m}+x\,.
\ee
\end{remark}

The final property we will need is that $\CS_{A,m,k}$ satisfies a system
of $q$-difference equations.

\begin{lemma}
\label{lem.cs.qdiff}
For a positive integer $m$ and residue class $k\in\{0,\cdots,m-1\}^N$, the
function $t^{k}\CS_{A,m,k}(t,q)$ defined in equation~\eqref{FGIcong} satisfies
the $q$-difference equations
\be
\label{PhiFGIAshift}
\begin{aligned}
  &\sum_{\ell=0}^{m}(-1)^{\ell}q^{-\ell(\ell-1)/2}
  \binom{m}{\ell}_{\!\!q^{-1}}\s_j^{\ell}\Big(t^{k}\CS_{A,m,k}(t^{m},q)\Big)\\
  &\=
  (-1)^{A_{j,j}m(m+1)/2}q^{A_{j,j}m(m+1)/2}t_i^{m}
  \prod_{i=1}^N \s_i^{mA_{i,j}}\Big(
  t^{k}\CS_{A,m,k}(t^{m},q)\Big)\,,
\end{aligned}
\ee
where $\s_jt:=(t_1,\dots,qt_j,\dots, t_N)$.
\end{lemma}

\begin{proof}
From the definition of $I_{A,m,k}(t,x)$ and a simple change of coordinates in the
Gaussian integration as done in~\cite{AarhusII} or~\cite{GSW}, it is easy to see for
$\ell\in(\BZ/mm'\BZ)^{N}$ that $I_{A,\ell}(t,x)$ satisfies
\be
\label{SAshift}
I_{A,mm',\ell}(t,x) - I_{A,mm',\ell}(\s_j t,x) \= (-1)^{A_{j,j}} t_{j} q^{A_{j,j}}
I_{A,mm',\ell-\delta_j}\bigg(\prod_{i=1}^N \s_i^{A_{i,j}}t,x\bigg),
\qquad j\=1,\dots,N\,,
\ee
where $\delta_j$ is the vector whose only non-zero entry is $1$ in the $j$-th
position. Therefore, from the basic properties of Gaussian polynomials we see that
\be
\begin{aligned}
&\sum_{i=0}^{m}(-1)^{i}q^{-i(i-1)/2}\binom{m}{i}_{\!\!q^{-1}}I_{A,mm',\ell}(\s_j^{i}t,x)
\\
&\= (-1)^{A_{j,j}m(m+1)/2}q^{A_{j,j}m(m+1)/2}t_i^{m}
I_{A,mm',\ell-m\delta_i}\bigg(\prod_{i=1}^N \s_i^{mA_{i,j}}t,q\bigg)\,.
\end{aligned}
\ee
The rest of the proof follows from the definition of $\CS_{A,m,k}$.
\end{proof}

\subsection{Algebraicity via WKB}
\label{sub.WKB}

In this section we outline an independent proof of Lemma~\ref{lem.PhiA} of
Section~\ref{sub.FGI} using instead of formal Gaussian integration, the unique
solution to the system of
linear $q$-difference equations~\eqref{PhiAshift}. This is essentially the WKB
method~\cite{Bender} for linear $q$-difference equations~\cite{GG:qdiff}, which
writes the first derivative of a new function as a
differential polynomial of known functions. In general, integrating will produce
non-algebraic functions, but with some care one can overcome and solve this
vanishing residue problem. Calculations similar to those of this section were done in
unpublished work of Masha Vlasenko.

Rather than explain this method for the general case, we do so for the
case of $1 \times 1$ integer matrices $A$. Consider the power series $z=z(t)$ defined
by
\be
1-z\=tz^A
\ee
and
\be
X\=\frac{z-1}{(1-A)z+A}
\=
\begin{pmatrix}
1 & -1\\
1-A & A
\end{pmatrix}
\cdot
z\,,
\ee
so that 
\be\label{eq:XZrel}
z\=\frac{AX+1}{(A-1)X+1}
\=
\begin{pmatrix}
A & 1\\
A-1 & 1
\end{pmatrix}
\cdot
X \,.
\ee
Then notice that for $\Delta=X(AX+1)((A-1)X+1)$ we have
\be
t\frac{d}{dt} \= \Delta\frac{d}{dX} \,, \qquad
X \= \frac{t}{z}\frac{dz}{dt}\,.
\ee
% Moreover, if $X=0$ then $t=0,z=1$, if $X=-A^{-1}$ then $t=\infty,z=0$,
% if $X=(1-A)^{-1}$ then $t=0,z=\infty$.
% Then as
% \be
%   t\frac{d}{dt}\frac{1}{z}\Big(t\frac{d}{dt}\Big)^kz
%   +X\frac{1}{z}\Big(t\frac{d}{dt}\Big)^kz
%   \=
%   \frac{1}{z}\Big(t\frac{d}{dt}\Big)^{k+1}z\,,
% \ee
% we see that by induction for $k\in\BZ_{>0}$
% \be
%   \frac{1}{z}\Big(t\frac{d}{dt}\Big)^kz
%   \in X\BZ[X]\,.
% \ee
This implies that
\be
\label{eq:logz}
\exp\Big(h t\frac{d}{dt}\Big)\log(z(t))-\log(z(t))-Xh \inn
X(AX+1)((A-1)X+1)\BQ[X][\![h]\!]\,.
\ee
Suppose that $F(t;q)\in\BQ(\!(h)\!)[\![t]\!]$ is the unique solution to the equation
\be
  F(t;h)-F(e^{h}t;h)\=tF(e^{Ah}t;h)\,,
\ee
and their exists $c_k(t)\in\BQ[\![t]\!]$ so that $F(t;h)$ satisfies the following
ansatz
\be
F(t;h) \= \exp\Big(\sum_{k=-1}^{\infty}c_k(t)h^{k}\Big)\,.
\ee
We will prove that such $c_k$ exist and that for $k>0$ they are algebraic functions.
In order to do this we consider the quotient
\be\label{eq:bkdef}
G(t;h) \= \frac{F(e^h t;h)}{F(t;h)}
\;=:\; z(t)\exp\Big(\sum_{k=1}^{\infty}b_k(t)h^k\Big)\,.
\ee
The functional equations of $F$ and $z$ imply that $G$ satisfies the equation
\be
\label{eq:funeqG}
\sum_{k=1}^{\infty}\frac{(-1)^{k+1}}{k}
\Big(\frac{z(t)-G(t;h)}{1-z(t)}\Big)^k
\= \log\Big(1+\frac{z(t)-G(t;h)}{1-z(t)}\Big)
\= \sum_{j=0}^{A-1}\log\Big(\frac{G(e^{jh}t;h)}{z(t)}\Big)\,.
\ee
This implies that
\be
\label{eq:b1}
-\frac{z(t)}{1-z(t)}b_{1}(t) \= Ab_{1}(t)+\frac{1}{2}A(A-1)X
\ee
and
\be\label{eq:b1andb2}
\frac{1}{2}\Big(\frac{z(t)}{1-z(t)}b_{1}(t)\Big)^2-\frac{z(t)}{1-z(t)}
\Big(b_{2}(t)+\frac{b_{1}(t)^2}{2}\Big)
\= Ab_{2}(t)+\sum_{j=0}^{A-1}\Big(jt\frac{d}{dt}\Big)\big(X(t)+b_1(t)\big)\,.
\ee
We have the following lemma.

\begin{lemma}
The coefficients $b_{k}(t)$ defined by equation~\eqref{eq:bkdef} exists and satisfy
the properties:
\be
b_{1}(t) \= \frac{A(A-1)}{2}X^2\,,
\quad\text{and for }k\in\BZ_{>1}\quad b_{k}(t)\inn X\Delta\BQ[X]\,.
\ee
\end{lemma}

\begin{proof}
Firstly, $b_{1}(t)$ can be explicitly computed using~\eqref{eq:XZrel}~and~\eqref{eq:b1}.
Next we notice that
\be
-\frac{zb_{1}(t)}{1-z} \= \frac{(A-1)A}{2}X(AX+1)\inn X\BQ[X]
\ee
and
\be
-\frac{zb_{1}(t)^2}{1-z} - \Big(-\frac{zb_{1}(t)}{1-z}\Big)^2
\= \frac{(A-1)^2A^2}{4}X^2(AX+1)((A-1)X+1) \inn X\Delta\BQ[X]\,.
\ee
By induction on $k$, we see that for $k\in\BZ_{>0}$ we have
\be
\label{eq:indb1}
-\frac{zb_{1}(t)^k}{1-z} - \Big(-\frac{zb_{1}(t)}{1-z}\Big)^k
\inn X\Delta\BQ[X]\,.
\ee
% \be
%   \Big(-\frac{zb_{1}(t)}{1-z}\Big)^k
%   \=
%   \Big(-\frac{zb_{1}(t)}{1-z}\Big)^{k-1}
%   \Big(-\frac{zb_{1}(t)}{1-z}\Big)
%   \=
%   \Big(-\frac{zb_{1}(t)^{k-1}}{1-z}\Big)
%   \Big(-\frac{zb_{1}(t)}{1-z}\Big)
%   +X\Delta\BQ[X]
%   \=
%   \Big(b_{1}(t)^{k-2}\Big)
%   \Big(-\frac{zb_{1}(t)}{1-z}\Big)^2
%   +X\Delta\BQ[X]
%   \=
%   -\frac{zb_{1}(t)^k}{1-z}
%   +X\Delta\BQ[X]\,.
% \ee
Already this and equation~\eqref{eq:b1andb2} implies that
$b_{2}(t)\in X\Delta\BQ[X]$. By induction, suppose that for some $K>2$
we have that $b_{k}(t)\in X\Delta\BQ[X]$ for all $k<K$. Then we
see that
\be
\frac{1-G(t)}{1-z(t)}\quad\in\quad
\frac{1-z(t)\exp(b_{1}(t)h)}{1-z(t)} + \frac{-z(t)b_{K}(t)h^K}{1-z(t)}
+ h^2 X\Delta\BQ[X,h] + \mathrm{O}(h^{K+1})\,,
\ee
and therefore
\be
\begin{small}
\begin{aligned}
\log\Big(\frac{1-G(t)}{1-z(t)}\Big)
\inn \log\Big(\frac{1-z(t)\exp(b_{1}(t)h)}{1-z(t)}\Big)
- \frac{z(t)b_{K}(t)h^K}{1-z(t)} + h X\Delta\BQ[X,h]
+ \mathrm{O}(h^{K+1})\,.
\end{aligned}
\end{small}
\ee
From equation~\eqref{eq:indb1}, we see that
\be
\frac{1-z(t)\exp(b_{1}(t)h)}{1-z(t)} \inn
\exp\Big(\frac{-z(t)b_{1}(t)h}{1-z(t)}\Big)+h X\Delta\BQ[X,h]\,,
\ee
and so
\be
\log\Big(\frac{1-z(t)\exp(b_{1}(t)h)}{1-z(t)}\Big)
\inn \frac{-z(t)b_{1}(t)h}{1-z(t)} +h X\Delta\BQ[X,h]\,.
\ee
Together with the fact that $zX^2(1-z)^{-1}=-X-AX^2\in X\BZ[X]$
and the functional equation~\eqref{eq:funeqG} of $G$, this gives a relation
for $b_{K}(t)$ of the form
\be
\frac{-z(t)b_{K}(t)}{1-z(t)} -Ab_{K}(t) \= X^{-1}b_{K}(t) \inn \Delta\BQ[X]\,.
\ee
This completes the proof.
\end{proof}
This lemma can now be used to determine the properties of the coefficients of $F(t;h)$.

\begin{corollary}
The coefficients $c_{k}(t)$ in the function $F(t;h)$ exists and for $k\in\BZ_{>0}$
\be
c_{k}(t)\inn X\BQ[X]\,.
\ee
\end{corollary}

\begin{proof}
Notice that
\be
\frac{1}{\exp(h t\frac{d}{dt})-1}\log(G(t;h))
\= \sum_{k=0}^{\infty}\frac{B_k}{k!}(h t\frac{d}{dt})^{k-1}\log(G(t;h))
\= \log(F(t;h))\,.
\ee
Therefore, we see that to compute the coefficient $c_{k}(t)$ we need to integrate
$b_{k+1}(t)$ with the initial condition $c_{k}(0)=0$. We see that
\be
\int b_{k+1}(t)\frac{dt}{t} \= \int b_{k+1}(t)\frac{dX}{\Delta} \inn X\BQ[X]\,,
\ee
since $\Delta^{-1}b_{k+1}(t)\in X\BQ[X]$.
\end{proof}

\subsection{Synthesis}
\label{sub.synthesis}

This section combines complementary results about admissible series on the one
hand and formal Gaussian integration on the other. It identifies the two
collections $\Phitof_A(t,q)$ and $\Phitof_A^\FGI(t,q)$ and deduces a stronger
property for their ring of coefficients.
This will allow us to specialise to $t=1$ later and obtain elements of
modules of the Habiro ring. Firstly, we will prove Theorem~\ref{thm.1} of
Section~\ref{sub.results} in the following refined form. 

\begin{theorem}
\label{thm.fgi.is.admis}
For every symmetric matrix $A$ with integer entries, positive integer $m$ and residue
class $k\in\{0,\dots,m-1\}^{N}$, we have:
\be
\label{eq.thm1.again}
F_{A,m,k}(t^{1/m},q) \= \CS_{A,m,k}(t,q), \qquad q\=\z_{mm'}+x\,,
\ee
for any positive integer $m'$, where these functions are defined by
equations~\eqref{FAmdef} and~\eqref{FGIcong}.
\end{theorem} 

\begin{proof}
We will show that for every positive integer $m'$, we have 
\be
\label{thm1.tx}
F_{A,m,k}(t^{1/m},\z_{mm'}+x) \= \CS_{A,m,k}(t,\z_{mm'}+x) \inn
\BQ[\z_{mm'}](\!(x)\!)[\![t]\!] \,.
\ee
The proof of this equality uses crucially the fact that $F_{A,m,k}(t^{1/m},q)$
is the unique solution to a system of linear $q$-difference equations with initial
condition $F(0,q)=1$.
We discuss in detail the case where $A$ is a $1 \times 1$ matrix, in which case
the linear $q$-difference equation is

%% see Mathematica file: kontsevich/DTInvariants.Admissible.Series.Rank1.nb
%% see pari file: A.3.Nahm.sum.admissible.series.q.1.pari.stavros

\be
\label{FArec}
\begin{aligned}
&\sum_{\ell=0}^{m}(-1)^{\ell}q^{-\ell(\ell-1)/2}
\binom{m}{\ell}_{\!\!q^{-1}}\s_j^{\ell}\Big(t^{k}F_{A,m,k}(t,q)\Big)\\
&\=
(-1)^{Am(m+1)/2}q^{Am(m+1)/2}t^{m} \s^{Am}\Big(t^{k}F_{A,m,k}(t,q)\Big)\,.
\end{aligned}
\ee

The first thing to note is that~\eqref{FArec} has a unique solution in the set
$t^k\BQ[\z_{mm'}](\!(x)\!)[\![t^{m}]\!]$. Indeed, if
$t^kF(t,q) = \sum_{j \geq 0} a_j(x) t^{k+mj}$ is
a solution to~\eqref{FArec} with $q=\z_{mm'}+x$, then $a_j(x)$ satisfies
\be
\begin{aligned}
  &\sum_{\ell=0}^{m}(-1)^{\ell}q^{-\ell(\ell-1)/2}
  \binom{m}{\ell}_{\!\!q^{-1}}q^{k\ell+mj\ell}a_j(x)
- (-1)^{Am(m+1)/2}q^{Am(m+1)/2+Am(j-1)}a_{j-1}(x)\\
&\=(q^{k+mj+1-m};q)_{m}a_j(x)-(-1)^{Am(m+1)/2}q^{Am(m+1)/2+Am(j-1)}a_{j-1}(x)
\= 0
\end{aligned}
\ee
for all $j \geq 0$, which together with the initial condition $a_0(x)=1$ uniquely
determines $a_j(x)$.

In fact, the above calculation and the definition of $F_{A,m,k}(t,q)$
prove that $F_{A,m,k}(t,q)$ satisfies~\eqref{FArec}.
Lemma~\ref{lem.cs.qdiff} of Section~\ref{sub.FGI} shows that $\CS_{A,m,k}(t^{m},q)$
is also a solution to~\eqref{FArec}. Moreover, the
fact that $\det(\Lambda^{-1})=O(t)$ dominates the $t^{-1}$ appearing from the
denominators of $\Li_{\leq0}(z)$ shows that
$\CS_{A,m,k}(t,\z_{mm'}+x)\in\BQ[\z_{mm'}](\!(x)\!)[\![t]\!]$.

This concludes the proof of the first part of the theorem when $A$ is a $1 \times 1$
symmetric matrix. The proof in the general case is identical using
equation~\eqref{PhiFGIAshift}, and is omitted. 
\end{proof}

A consequence of Theorem~\ref{thm.fgi.is.admis} is that the power series defined
in equations ~\eqref{Vdef}, ~\eqref{eq:Vt.m}, ~\eqref{tautdef}, and ~\eqref{Umdef}
for the admissible series and those defined in equations~\eqref{Vt}, ~\eqref{dt},
and ~\eqref{Umtdef} for the FGI series coincide:
\be
\label{eq.equal}
\delta \= \delta^\FGI, \qquad V \= V^\FGI, \qquad U_m \= U_m^\FGI \,.
\ee

\begin{proof}[Proof of Lemma~\ref{Vlem} of Section~\ref{sub.DT} and
  Lemma~\ref{Vlem.m} of Section~\ref{sub.madmissible}]
These follow from Theorem~\ref{thm.fgi.is.admis} and the explicit formula
of $\CS_{A,m,k}(t,q)$ given in equation~\eqref{Ikdef}.
\end{proof}

\begin{corollary}
\label{lem.V2F}
If $F_A(t,q)$ is the series~\eqref{FAdef} and $V=V_A(t)$ is the associated potential,
then~$V$ determines~$A$ and hence~$F_A$. 
\end{corollary}

\begin{proof}
We have
\be
\label{eq.V2F}
t_j \pt_{t_j} V(t) \= \log z_j(t), \qquad j\=1,\dots,N
\ee
Hence $V$ determines $z_j$ for $j=1,\dots,N$, and consequently the Hessian
\be
\label{Vhes}
\frac{\partial^2}{\partial z_i \partial z_j} V(t) \=
-A-\mathrm{diag}\Big(\frac{z(t)}{1-z(t)}\Big) \,.
\ee
The result follows.
\end{proof}

Note that the proof of Theorem~\ref{thm.fgi.is.admis} uses the expansion of a solution
of~\eqref{FArec} into power series in $t$ whose coefficients are power series in $x$,
unlike the WKB method, which uses an expansion of a solution into power series in
$x$ whose coefficients are power series in $t$.

%% See Mathematica file: Universal.Denominators.All.OEIS.nb.
%% Sections: 'Exponentiating DDc', 'Exponentiating Dc', Exponentiating Dc1.

%%% see: DTInvariants.Admissible.Series.nb section
%%% 'Denominators for Exp[V/Log[q]]-prefactor versus Exp[V/(q-1)]-prefactor' and
%%% 'Universal denominators of WKB'
%and for for $k < 32$ agree with the OEIS sequence \texttt{A001164}
%(and no longer agree for $k=32$), and they are actually expressed in terms
%of the $D_k$ universal denominator sequence
%$D_k$ % OEIS A144618
%as follows
%%% see Mathematica file: Universal.Denominators.All.OEIS.nb
%For the particular WKB solution to the linear $q$-difference equation~\eqref{Frec},
%the denominators of $S_k$ are given in \texttt{notes/note\_WKB.txt}.
%% See section 'Universal denominators of WKB: OEIS sequence A053657 of the
%% Mathematica file DTInvariants.Admissible.Series.nb.
%It turns out
%that it is the OEIS \texttt{A053657} sequence defined (after a shift of $n$ by 
%$1$) by
%%% see pari file: gluing.hab.mod.t0.pari.high.precision.stavros

%% tests when $t=1$ for $5_2$ are given in the pari file:
%% p.imov.peter.style.pari

%% See Mathematica file: pValuations.Taylor.Series.Rational.Functions.nb

\begin{proof}[Proof of Theorem~\ref{thm.FGI2} of Section~\ref{sub.results}]
Theorem~\ref{thm.fgi.is.admis} together with Lemmas~\ref{lem.dwork.zm.admis} and 
~\ref{lem:dwork.diff.VU} imply the coefficients of $x$ in equation~\eqref{dphiA}
are in $\BZ[t^{1/{m}},\z_{m}][\![t]\!]\cap S^{(m)}_p[\tfrac{1}{p},z^{1/m}]$. To complete
the proof we will show that
\be
\label{eq:int.vs.comp}
\BZ[t^{1/{m}},\z_{m}][\![t]\!]\,\cap\, S^{(m)}_p[\tfrac{1}{p},z^{1/m}]
\;\subseteq\; S^{(m)}_p[z^{1/m}]\,.
\ee
Elements in $S^{(m)}_p[\tfrac{1}{p},z^{1/m}]$ can be represented by polynomials in
$t^{1/{m}}$ of degree less than $m$ whose coefficients are polynomials in
$z^{\pm1/m}$, $t^{\pm1}$ and $\delta^{-1}$. Firstly, we find that
$z(t)^{\pm1/m}\in1+t\BZ[\tfrac{1}{m}][\![t]\!]$ as the LHS of the equation
\be
(-1)^{\diag(A)}z^{-A}(1-z)\=t\,,
\ee
is integral and therefore inverting the series and solving for $1-z$ in terms of
$t$ leads to an integral power series in $t$.
Secondly, as $\delta$ is a polynomial in $z$ times a monomial in $z^{\pm1}$ and has
$\delta(0)=1$ we see that $\delta^{\pm1}\in1+t\BZ[\![t]\!]$.
Thirdly, in any expression with $t$ we can replace it with $z$ using the
functional equations~\eqref{zjt}.
Therefore, for any element of $S^{(m)}[\tfrac{1}{pm},z^{1/m}]$ we can remove all $t$
from the coefficients of $t^{1/m}$ and factor out all denominators given that
multiplying by $z^{\pm1/m},\delta^{-1}$ preserve integrality of the power series.
Therefore, elements in $S^{(m)}[\tfrac{1}{pm},z^{1/m}]$ can be represented by a unit in
$\BZ[\tfrac{1}{m},t^{1/{m}},\z_{m}][\![t]\!]$ times a polynomial in $(1-z)$. Assuming
that this element is in the intersection in equation~\eqref{eq:int.vs.comp}, we see
that this implies that this polynomial is in $\BZ[t^{1/m},\z_{m}][\![1-z]\!]$.
This implies that the polynomials in $1-z$ are $p$-integral and therefore implies
the inclusion of equation~\eqref{eq:int.vs.comp}.
\end{proof}

%%%%%%%%%%%%%%%%%%%%%%%%%%%%%%%%%%%%%%%%%%%%%%%%%%%%%%%%%%%%%%%%%%%%%%%%%%%% 
%%%%%%%%%%%%%%%%%%%%%%%%%%%%%%%%%%%%%%%%%%%%%%%%%%%%%%%%%%%%%%%%%%%%%%%%%%%%

\section{\texorpdfstring{Integrality and gluing of $p$-adic series}{Integrality
    and gluing of p-adic series}}
\label{sec.modules}

In this section we will combine the previous sections to prove our main results.
Firstly, we study the $K$-theory of local fields and its
image under the $p$-adic regulator map. Then we provide the technical results
needed in the proof of the main theorems.
After that we prove our main results about the $p$-completed Habiro rings
(Theorem~\ref{thm.locals} of the introduction and Theorem~\ref{thm.Psi} of this
section) by constructing explicit generators using the infinite Pochhammer symbol
evaluated at roots of unity.
We close by proving the remaining theorems (Theorems~\ref{thm.HR} and~\ref{thm.2})
of Section~\ref{sub.results}.

\subsection{\texorpdfstring{Bloch group and $p$-adic dilogarithm}{Bloch
    group and p-adic dilogarithm}}
\label{sub.Dp}

We begin by recalling some basic facts about the dilogarithm function, the Bloch
group, and their $p$-adic counterparts. The dilogarithm function, defined initially
for $|x|<1$ by
\be
\label{eq:li2}
\Li_{2}(x) \= \sum_{k=1}^{\infty} \frac{x^k}{k^2}\,,
\ee
satisfies the famous five-term functional equation. The five-term equation reflects
the algebraic structure of the third algebraic $K$-group of a field, which is
isomorphic, modulo well-understood torsion, with the Bloch group of the number
field~\cite{Bloch,Zagier:dilog}. Therefore, functions that satisfy the five-term
relation give homomorphisms from the Bloch group, and hence from the third algebraic
$K$-group. The Bloch-Wigner dilogarithm $D:\BC\rightarrow\BR$,
given by $D(z)=\Im(\Li_{2}(z))+\arg(1-z)\log|z|$, is one such
function and satisfies the functional equation
\be
\label{eq:5term}
D(x)+D(y) \= D(xy) +D\Big(\frac{1-y}{1-x^{-1}}\Big) +D\Big(\frac{1-x}{1-y^{-1}}\Big)\,.
\ee
The usual dilogarithm is a holomorphic but multivalued function from $\BC-\{0,1\}$.
The Bloch-Wigner dilogarithm gives a well-defined continuous real valued function
on $\BC$. Taking a number field $\BK$ and an infinite place $\sigma$ (i.e.
$\sigma:\BK\rightarrow\BC$ an embedding of the number field into~$\BC$) we can define
a map $D_{\sigma}:B(\BK)\rightarrow\BR$ defined on the symbols in $B(\BK)$ by linearly
extending $D_{\sigma}([z])=D(\sigma(z))$. An analogous function valued in $\BC_p$ was
defined by Coleman~\cite{Coleman} for finite places. The main point is to find an
analytic extension of $\Li_2(z)$ in equation~\eqref{eq:li2}, which only converges for
$z\in\BC_{p}$ with $|z|<1$. Coleman used an ``analytic continuation along Frobenius''
to define such a function after choosing a logarithm
$\log:\BC_{p}^{\times}\rightarrow\BC_p$. For example, one could take
the Iwasawa logarithm, which has $\log(p)=0$. (The values of the dilogarithm only
depend on this
choice in the residue discs around $1$ and $\infty$ and importantly the primes
we consider will never require us to evaluate here.)
Then Coleman obtains a well-defined function
$\Li_2:\BC_p\smallsetminus\{1\}\rightarrow\BC_{p}$ and defines
\be
\label{Dpdef}
D_p:\BC_{p}\smallsetminus\{0,1\}\rightarrow\BC_{p}, \qquad z \mapsto
\Li_2(z)+\frac{1}{2}\log(z)\log(1-z)\,.
\ee
This function also satisfies equation~\eqref{eq:5term}. Therefore, this gives
homomorphisms from the Bloch group in a completely analogous way to the Bloch-Wigner
dilogarithm but for finite places (i.e. $\sigma:\BK\rightarrow\BC_{p}$ an embedding of
the field). Conveniently, for unramified primes $p$ of the field $\BK$, we can combine
these functions $D_{\sigma}$ into one function $D_{p}:B(\BK)\rightarrow\BK_{p}$
such that the projection onto the various field factors of $\BK_{p}$
corresponding to $\sigma$ agree with the function $D_{\sigma}$.
Let $R=\calO_\BK[\tfrac{1}{\Delta}]$ as in~\eqref{Fetale} and $R^\wedge_p$ denote its
$p$-adic completion. The $p$-adic polylogarithm functions satisfy the following
integrality properties:

\begin{lemma}
\label{lem.pint}
If $z\in R^\wedge_p$ and $|z|_{p}=1$ and $|z-1|_{p}\geq 1$ then
$D_{p}(z)\in p^2 R^\wedge_p$.
\end{lemma}

\begin{proof}
From~\cite[Cor. 4.9]{Lip} we know that
\be
\Li_2(\z)\inn p^2 \BZ_{p}[\z]
\quad\text{and}\quad
\Li_{1}(\z)\inn p \BZ_{p}[\z]\,.
\ee
This implies that $D_{p}(\z)\in p^2 R^\wedge_p$ if $\z$ is a root of
unity. Then expanding these as Taylor series near $\z$ and noting that we assumed
$p>3$ is unramified proves the result on the unit discs near $\z$.
\end{proof}

In what follows, we are mostly interested in evaluating $\Li_2(\z)$ when $\z$ is
a root of unity of order prime to $p$. This case is important because these are
rank 1 elements, i.e., they satisfy the
equation
\be
\fr\z\=\z^p\,.
\ee
Although the $p$-adic polylogarithm $\Li_n$ is a transcendental function for
$n>0$, its modulo $p$ reduction can be described by a rational function, the
finite polylogarithm function introduced by Kontsevich (unpublished note, reproduced
as an appendix to~\cite{EVG})
\be
\mathrm{li}_{n,p}(z)
\=
\sum_{k=1}^{p-1}
\frac{z^{k}}{k^n}\,.
\ee
%Recall the Teichm\"muller character $\omega: \overline \BF_p^\times \to \BC_p^\times$.
Below, we denote by $\BQ_{p^s}$ the unique unramified extension of $\BQ_p$ of degree
$s$ (with residue field $\BF_{p^s}$), and its ring of integers by $\BZ_{p^s}$. 

\begin{proposition}~\cite[Cor. 2.2]{finli}
\label{prop:zspan}
For a prime $p$ and $\z\in\mu(\BQ_{p^s})$ with $\z \neq 1$, we have
\be
p^{-2}D_{p}(\z^p) \= \frac{1}{(\z-1)^p} \mathrm{li}_{2,p}(\z) \pmod{p}\,.
\ee
\end{proposition}

We can use this proposition to prove the following.

\begin{proposition}
\label{prop.xiz}
For every positive integer $s$, we have
\be
\label{eq.xiz}
\mathrm{Span}_{\BZ_{p}}\{p^{-2}D_{p}(\z)\,|\,\z\inn \mu(\BQ_{p^s}) \}
\= \BZ_{p^s} \,.
\ee
\end{proposition}

\begin{proof}
It suffices to prove~\eqref{eq.xiz} modulo $p$. 
Consider the rational function
\be
f_{p}(x)
\=
\frac{\mathrm{li}_{2,p}(x)}
{(x-1)^p}\,.
\ee
Since $\mathrm{li}_{2,p}(1)\;\equiv\;0\pmod{p}$ and
$\mathrm{li}_{2,p}'(1)\;\equiv\;0\pmod{p}$, there exists a polynomial $g_{p}(x)$
of degree $p-3$ such that $f_{p}(x)\=\frac{g_{p}(x)}{(x-1)^{p-2}}$.
We can view $f_p$ as a map
\be
f_{p}:\BF_{p^{s}}\smallsetminus\{1\}\;\rightarrow\;\BF_{p^{s}}\,.
\ee
Then we see that $f_{p}^{-1}(c) \subseteq
\{x\in\BF_{p^{s}}\smallsetminus\{1\}\,|\,c(x-1)^{p-2}-g_{p}(x)\}$
and therefore $\#f_{p}^{-1}(c)\leq p-2$. 
This gives a lower bound for the size of the image of $f_{p}$ 
\be
\#\mathrm{image}(f_{p})\;\geq\;\frac{p^{s}-1}{p-2}\;>\;p^{s-1}\,,
\ee
which implies that
$ p^{s-1}\,<\,p^{r} =\#\mathrm{Span}_{\BF_{p}}(\mathrm{image}(f_{p})) \;\leq\;p^{s}$.
Therefore, we see that
$\mathrm{Span}_{\BF_{p}}(\mathrm{image}(f_{p})) \= \BF_{p^s}$.
Since $\mu(\BQ_{p^s}) \simeq \BF_{p^s}^{\times}$, Proposition~\ref{prop:zspan} implies 
that $p^{-2}D_{p}(\z)$ span $\BF_{p^s}$ as an $\BF_p$-vector space. The result
follows. 
\end{proof}

Note that if $\BK/\BQ$ is a number field and $p$ is unramified, so that
$p=\prod_i\mathfrak{p}_i$, $N(\mathfrak{p}_i)=p^{s_i}$, then
$\BK_p \simeq \prod_i \BQ_{p^{s_i}}$ and $K_n(\BK_p) \simeq \prod_i K_n(\BQ_{p^{s_i}})$
for every $n$.

\begin{theorem} 
\label{thm.xiz}
Assume $p>3$. We have a $\BZ_p$-linear isomorphism
\be
D_p: K_3(\BK_p;\BZ_p) \to p^2\calO_{\BK_{p}}\,,
\ee
and $K_3(\BK_p;\BZ_p)$ is generated as a $\BZ_p$-module by
$\{[\z]\,|\,\z\in \mu(\BK_p) \}$.
\end{theorem}

For related work on $p$-adic regulators/dilogarithms, see Huber--Kings~\cite{Huber}
and also Besser--de Jeu~\cite[Thm.1.6(2)]{BJ:syntomic}.

\begin{proof}
It is known that $K_3(\BK_p;\BZ_p)\cong H^1(\BK_p,\BZ_p(2))$ is a free $\BZ_p$-module
of rank $r$, cf.~\cite[Thm. 7.4]{Weibel:Kbook} (the torsion-freeness follows from
the vanishing of $H^0(\BK_p,\BF_p(2))$, which follows from $p>3$). Under the
isomorphism between $K_3$ and the Bloch group (and noting that there is no
contribution from $K_2=K_2^M$ to $K_3(-;\BZ_p)$) and using Lemma~\ref{lem.pint}, one
sees that the image is contained in $p^2\calO_{\BK_p}$. But the elements
$\mathrm{ord}(\z)D_p(\z)\in\BK_{p}$ generate the image. Therefore, this gives a
surjective map between two free $\BZ_p$-modules of the same rank and therefore the
map has a trivial kernel.
\end{proof}

\noindent An example that illustrates this theorem is given in Example~\ref{ex.xiz1} in
Section~\ref{sub.padic.compute}.

\subsection{\texorpdfstring{Explicit sections of the $p$-completed
    Habiro ring}{Explicit sections of the p-completed Habiro ring}}
\label{sub.habp}

In this section we consider the modules over the $p$-completed Habiro ring
of Definition~\ref{def.HRxi} of Section~\ref{sub.habmod}. We will prove
Theorem~\ref{thm.locals} of Section~\ref{sub.habmod} by explicitly constructing
sections using Theorem~\ref{thm.xiz} of Section~\ref{sub.Dp}.
% Throughout the
% section we fix the ring $R$ as in~\eqref{Fetale}, its Habiro ring
% $\calH_R$, and a prime $p$ prime to $\Delta$.
% Our next task is to study modules over the $p$-completed Habiro ring.  
% This requires the local field version of the maps $\delta$ and $\ve_m$
% of equations~\eqref{emdef} and~\eqref{taudef}.
% The modules $\calH_{R_p^{\wedge},\widehat\xi}$ can be defined completely
% analogously to Definition~\ref{def.HRxi} with the obvious replacement of $\BK$ by
% $\BK_p$ in equation~\eqref{fcyclic}, as well as in the maps~\eqref{emdef}
% and~\eqref{taudef}. With this definition, the analogue of Theorem~\ref{thm.HR}
% as well as the properties listed just after it hold for the modules over the
% completed Habiro ring. 
% In these complete rings we will explain a different construction
% of sections.
Throughout this section one can keep in mind $S=R^\wedge_{p}$, where $R=\calO_{\BK}$
and $\Delta$ are as in equation~\eqref{Fetale} and $S[\tfrac{1}{p}]=\BK_{p}$.
The next lemma is a variation of a well--known lemma of Dwork~
\cite[IV.Lem.3]{Koblitz} and will be the basic tool we use to prove integrality
properties.

\begin{lemma}[Dwork's lemma]
\label{lem:dwork}
Fix a prime $p$ and a positive integer $m$ with $(m,p)=1$ and  
$f(x)\in1+xS[\tfrac{1}{p},\z_{m}]\llbracket x\rrbracket$. Then
$f(x)\in1+x S[\z_{m}]\llbracket x\rrbracket$ if and only if
\be
\frac{\fr(f)((\z_{m}+x)^{p}-\z_{m}^p)}{f(x)^p}
\inn 1+px S[\z_{m}]\llbracket x\rrbracket\,,
\ee
and $f(x)\in1+x S[\z_{m}]+\mathrm{O}(x^2)$, where $\fr$ is the Frobenius of $S$
with $\fr(\z_{m})=\z_{m}^p$ and $\fr (x)=x$. 
\end{lemma}

\begin{proof}
Let
\be
f(x)\=1+\sum_{k=1}^{\infty}a_{k}x^{k}
\quad\text{and}\quad
\frac{\fr(f)((\z_{m}+x)^{p}-\z_{m}^p)}{f(x)^p}\=1+\sum_{k=1}^{\infty}b_{k}x^{k}\,.
\ee
Notice that if $a_{k}\in S[\z_{m}]$ then
\be
\frac{\fr(f)((\z_{m}+x)^{p}-\z_{m}^p)}{f(x)^p}
\= 1\pmod{p}\,.
\ee
This completes half of the proof. Therefore, assume now that $b_k\in pS[\z_{m}]$. Let
$\delta_{p,m}(x)=\frac{(\z_{m}+x)^{p}-x^p-\z_{m}^p}{px}\in\z_{m}^{p-1}+x\BZ[\z_{m},x]$.
Then trivially
\be
1+\sum_{k=1}^{\infty}\fr(a_{k})(x^{p}+xp\delta_{p,m}(x))^{k}
\=
\Big(1+\sum_{k=1}^{\infty}a_{k}x^{k}\Big)^{p}
\Big(1+\sum_{k=1}^{\infty}b_{k}x^{k}\Big)\,.
\ee
Assume for induction that $a_{i}\in S[\z_{m}]$ for $i<n$. Then computing the $n$-th
coefficient on both sides of the equation, we find that
\be
\fr(a_{n/p})+\z_{m}^{p-1}p^{n}\fr(a_{n})+p S[\z_{m}]
\=
\fr(a_{n/p})+pa_{n}+p S[\z_{m}]\,,
\ee
and so
\be
a_{n}-\z_{m}^{p-1}p^{n-1}\fr(a_{n})\inn S[\z_{m}]\,.
\ee
Therefore, if $s$ is the order of the Frobenius endomorphism $\fr$ on $S[\z_{m}]$,
then we get
\be
\begin{aligned}
(1-\z_{m}^{s(p-1)}p^{s(n-1)})a_{n}
&\=
\sum_{k=0}^{s-1}\z_{m}^{k(p-1)}p^{k(n-1)}(\fr^{k}(a_{n})
-\z_{m}^{p-1}p^{n-1}\fr^{k+1}(a_{n}))
\inn S[\z_{m}]\,.
\end{aligned}
\ee
So assuming that $a_{1}\in S[\z_{m}]$ we find that
$a_{n}\in S[\z_{m}]$ by induction.
\end{proof}
% These technical results can now be put together and used to prove our main
% theorems. This allows us to show that both the Habiro modules are well-defined and
% that the $\Phitof$-series are elements of these modules. The result proves two things.
% Firstly, integrality and secondly, gluing. This can be summarised in the following
% diagram where $f=e^g$.
% \be
% \begin{small}
% \label{ourcd}
% \begin{tikzpicture}[baseline=(current  bounding box.center),scale=0.8]
% \draw (-2,1) node (p1) {$\fr g(q^p)-p g(q)$};
% \draw (2,1) node (p2) {$g(q)$};
% \draw (2,-1) node (p3) {$\frac{\fr f(q^p)}{f(q)^p}$};
% \draw (-2,-1) node (p4) {$f(x)$};
% \draw[->] (p1) -- node[above] {\text{Thm.}\ref{thm.logdwork}} (p2);
% \draw[->] (p1) -- node[left] {$\exp$} (p4);
% \draw[->] (p2) -- node[right] {$\exp$} (p3);
% \draw[<-] (p3) -- node[below] {Dwork~\ref{lem:dwork}} (p4);
% \draw[blue,->] (-2,1.3) arc (180:90:0.5) -- (2.8,1.8) arc (90:0:0.5)
% -- (3.3,-0.5) arc (0:-90:0.5);
% \draw[blue] (4.3,0.5) node {gluing};
% \draw[blue,->] (-3.7,1) arc (90:180:0.5) -- (-4.2,-1.5) arc (180:270:0.5)
% -- (1.5,-2) arc (270:360:0.5);
% \draw[blue] (-5.5,-0.5) node {integrality};
% \end{tikzpicture}
% \end{small}
% \ee
% The next theorem is an analytic result on locally defined meromorphic functions.

When $D_p(\xi)=0$ we can use the assumption of equation~\eqref{fxidef} in
Definition~\ref{def.HRxi} of Section~\ref{sub.habmod} to prove integrality by
applying Dwork's lemma. This will be important in the proof of the
Theorem~\ref{thm.locals} of Section~\ref{sub.habmod} that $\calH_{R_p^\wedge,\xi}$
are free rank one bundles over $\calH_{R_p^\wedge}$.

\begin{corollary}
\label{cor.diagram}
Fix a prime $p$ and a $p$-complete torsion-free ring $S$. Suppose that
\be
g(q)\=(g_m(x))_{m\geq 1,(m,p)=1}\,,
\quad
g_{m}(x)\inn S[\tfrac{1}{p},\z_m][\![x]\!]
\ee
is a collection of power series such that
\be
\label{eq:gcond}
\begin{aligned}
% \mathrm{(a)}\quad &
% \text{$g_{m'}(x)$ converges at $x=0$ for all $m'$ prime to $p$,}\\
% \mathrm{(b)}\quad &
% \text{$g_{m'}(0)\in pS[\z_{m'}]$ for all $m'$ prime to $p$, and}\\
% \mathrm{(c)}\quad &
% \text{satisfies}\\
% &\qquad
\fr g(q^p) -p g(q)\inn \prod_{(m,p)=1}pS[\z_{m}][\![q-\z_m]\!]\,.
\end{aligned}
\ee
Then $f(q):=\exp(g(q))$ is well-defined and assuming
$f(\z_m+x)\in (1+pS[\z_m])+xS[\z_m]+\mathrm{O}(x^2)$, it satisfies
\be
\label{fintg}
f(q) \inn \prod_{m\geq 1\,,(m,p)=1}S[\z_m][\![q-\z_m]\!]\,.
\ee
\end{corollary}

\begin{proof}
Notice that equation~\eqref{eq:gcond} implies that $g_{m}(0)\in pS[\z_{m}]$
and so the series $f_{m}(x)$ is well-defined.
Moreover, equation~\eqref{eq:gcond} implies that
\be
\exp(\fr g_{m}((\z_{m}+x)^p-\z_{m}^p) -p g_{m}(x))
\inn (1+pS[\z_{m}])+pxS[\z_{m}][\![x]\!]\,.
\ee
Therefore, the fact that $f_{m}(0)\in 1+pS[\z_{m}]$ together with Dwork's lemma
(Lemma~\ref{lem:dwork}) shows us that $f_{m}(0)^{-1}f_{m}(x) \in S[\z_{m}][\![x]\!]$
and hence $f_{m}(x) \in S[\z_{m}][\![x]\!]$.
\end{proof}

In view of Definition~\ref{def.HRxi} of Section~\ref{sub.habmod}, we can focus on
pairs consisting of a prime $p$ and a positive integer $m$ prime to $p$ to define
invertible sections. More generally, we can lift this to all roots of unity using
the previous corollary and the following lemma.

\begin{lemma}
\label{lem:ftild}
Fix a prime $p$, $\xi \in K_3(\BK)$, and a collection
\be
f(q) \=(f_{m}(x))_{m\geq1}, \qquad
f_{m}(x)\inn \varepsilon_{m}(\xi)^{1/m}\,
(R^\wedge_p[\z_m]^{\times} +xR_p^\wedge[\z_m]+x^2\BK_p[\z_m][\![x]\!])
\ee
that satisfies
\be
\frac{\fr \widehat{f}(q^p)}{\widehat{f}(q)^p}
\in\!\!\prod_{(m,p)=1}\exp\Big(\frac{p}{x}R^\wedge_p[\z_m][\![x]\!]\Big)\,,
\ee
with $\widehat{f}$ as in equation~\eqref{hatfm}.
Then for all positive integers $m$ with $(m,p)=1$ there is a unique
$\beta_m\in\BZ_{p}$ such that the collection
$\widetilde{f}_{mp^d}(x)=\z_{p^d}^{\beta_m}f_{mp^d}(x)$ satisfies
\be
\frac{\widetilde{f}(q^{\sigma})}{\widetilde{f}(q)^{\sigma^{-1}}}
\in\!\!\prod_{(m,p)=1}R^\wedge_p[\z_m][\![x]\!],\qquad
\text{for all }\sigma\in\BZ_p^{\times}\,.
\ee
\end{lemma}

\begin{proof}
Firstly, for all $\sigma\in\BZ_{p}^{\times}$, we have
\be
\frac{\fr f(q^{p\sigma})}{\fr f(q^p)^{\sigma^{-1}}}
\frac{f(q)^{p\sigma^{-1}}}{f(q^{\sigma})^p}
\=
\frac{\fr \widehat{f}(q^{p\sigma})}{\fr \widehat{f}(q^p)^{\sigma^{-1}}}
\frac{\widehat{f}(q)^{p\sigma^{-1}}}{\widehat{f}(q^{\sigma})^p}
\in
\!\!\prod_{(m,p)=1}\exp\big(pR^\wedge_p[\z_m][\![x]\!]\big)\,.
\ee
This implies by Lemma~\ref{lem:dwork}, by induction on $d\in\BZ_{\geq0}$, that
for $(m,p)=1$ there is a series $h_m(x)\in R^\wedge_p[\z_m][\![x]\!]$ and an
$\alpha_\sigma\in\BZ_p$ such that
\be
f(q^{\sigma})f(q)^{-\sigma^{-1}}
\=
\z_{p^d}^{\alpha_\sigma}
h_{m}(\z_{m}(\z_{p^d}-1)+x)\,,
\qquad
q\=\z_{m}\z_{p^d}+x\,.
\ee
Notice that
\be
\frac{f(q^{\sigma\sigma'})}{f(q)^{(\sigma\sigma')^{-1}}}
\=
\frac{f(q^{\sigma\sigma'})}{f(q^{\sigma'})^{\sigma^{-1}}}
\Big(\frac{f(q^{\sigma'})}{f(q)^{\sigma'^{-1}}}\Big)^{\sigma^{-1}}
\=
\frac{f(q^{\sigma\sigma'})}{f(q^{\sigma})^{\sigma'^{-1}}}
\Big(\frac{f(q^{\sigma})}{f(q)^{\sigma^{-1}}}\Big)^{\sigma'^{-1}}
\ee
and so
\be
\alpha_{\sigma\sigma'}
\=
\sigma'\alpha_\sigma+\sigma^{-1}\alpha_{\sigma'}
\=
\sigma\alpha_{\sigma'}+\sigma'^{-1}\alpha_{\sigma}\,.
\ee
There is a unique such $\alpha_\sigma$ satisfying
$\alpha_{\lim_{s\rightarrow\infty}\sigma^{p^{s}}}
=\lim_{s\rightarrow\infty}\alpha_{\sigma^{p^{s}}}$, given by
\be
\alpha_{\sigma} \=
(\sigma-\sigma^{-1})\frac{\alpha_\omega}{\omega-\omega^{-1}}\,,
\ee
where $\omega$ is a generator of the set of roots of unity in $\BZ_p^{\times}$ for
$p\neq 2,3$. Indeed, notice that for $p\neq 2,3$ we have
$\omega-\omega^{-1}\in\BZ_p^{\times}$. Letting
$\beta_m=-\frac{\alpha_{\omega}}{\omega-\omega^{-1}}$ therefore completes the proof.
\end{proof}

\begin{corollary}
\label{cor:unique.loc}
If $f$ is an invertible $L_{p}(\xi)$-section of Definition~\ref{def.HRxi} of
Section~\ref{sub.habmod}, then there exists a unique lift to collections
\be
f(q) \=(f_{m}(x))_{m\geq1}, \qquad
f_{m}(x)\inn
(R^\wedge_p[\z_m]^{\times/m} + x\BK_p[\z_m][\![x]\!])
\ee
of power series satisfying the conditions of integrality and gluing, i.e.,
\be
\frac{\fr \widehat{f}(q^p)}{\widehat{f}(q)^p}
\in\!\!\prod_{(m,p)=1}\exp\Big(\frac{p}{x}R^\wedge_p[\z_m][\![x]\!]\Big)
\qquad\text{and}\qquad
\frac{f(q^{\sigma})}{f(q)^{\sigma^{-1}}}
\in\!\!\prod_{(m,p)=1}R^\wedge_p[\z_m][\![x]\!]\,,
\ee
where $\sigma\in\BZ_{p}^{\times}$ and $\widehat{f}$ is as in equation~\eqref{hatfm}.
\end{corollary}

\begin{remark}
\label{rem:mero}
Taking the logarithm of this extension of an invertible section $\widehat f(q)$
to all roots
of unity (with some mild assumptions) can be shown to give a meromorphic function
$\log(\widehat f(q))$ with simple poles at roots of unity with residues given by
$m^{-2}D_{p}(\xi)$ and its images under powers of $p^{-2}\fr$.
\end{remark}

% With this in mind we will focus on pairs $(m,p)=1$ with $p$ prime.
Now that we have understood the basic analytic properties of $L_p(\xi)$-sections,
we can give our explicit construction.
Locally our modules are indexed by an element
\be
\label{hatxia}
\widehat\xi \inn \mathrm{Im}(K_3(\BK) \to K_3(\BK_p;\BZ_p) \otimes \BZ_p) 
\ee
and our explicit invertible $L_p(\xi)$-sections
will depend on a presentation of $\widehat\xi$ of the form
\be
\label{hatxi}
\widehat\xi\;=\sum_{\z \in \mu(\BK_p)} a_\z [\z] \inn K_3(\BK_p) \otimes \BZ_p
\ee
using Theorem~\ref{thm.xiz} of Section~\ref{sub.Dp}. Using this data alone, we
define a collection of power series over local fields that, unlike its global
field counterpart, does not require formal Gaussian integration or admissible
series, but just the infinite Pochhammer symbol. What's more, the constant term
is simply the unit $\ve$, and is obviously non-vanishing.

\begin{definition}
\label{def.psi}
Fix a prime $p$ and $\z \in \mu(\BK_p)$. Then we define the collection
$\Psi_{[\z],p}(q) = (\Psi_{[\z],p,m}(x))_{m\geq 1\,(m,p)=1}$ 
by
\be
\label{Psiz}
\begin{aligned}
\Psi_{[\z],p,m}(x) 
&\= \exp\Big(-\frac{\Li_2(\z)}{m^2\log(q)}\Big)\;\ve_m([\z])^{1/m}
(q^{m/2}\z;q^m)_\infty^{1/m}\,,\qquad q\=\z_m+x\\
&\qquad\in \ve_m([\z])^{1/m}(1 + x\BK_p[\z_m][\![x]\!])\,. \\
\end{aligned}
\ee
Moreover, for every $\widehat\xi$ given in equation~\eqref{hatxi} we define
\be
\label{Psixi}
\Psi_{\widehat\xi,p} \= \prod_{\z} (\Psi_{[\z],p})^{a_\z}\,.
\ee
\end{definition}

A slight variation of Proposition~\ref{prop.poclip} of Section~\ref{sub.poch} and
equation~\eqref{eq.cor.logpoc}, to include the factor of $q^{1/2}$, implies
the following theorem.

\begin{theorem}
\label{thm.Psi}  
For all $\widehat\xi$ as in equation~\eqref{hatxia} and unramified $p>3$,
$\Psi_{\widehat\xi}$ is an $L_{p}(\widehat\xi)$-section.
\end{theorem}

We can use this to prove Theorem~\ref{thm.locals} of Section~\ref{sub.habmod}.

\begin{proof}[Proof of Theorem~\ref{thm.locals} of Section~\ref{sub.habmod}]
Combining Theorem~\ref{thm.Psi} with Corollary~\ref{cor.diagram} immediately
implies Theorem~\ref{thm.locals} of Section~\ref{sub.habmod}.

Alternatively, we can use replace the use of Theorem~\ref{thm.Psi} by a general
construction from telescoping sums. For any
\be
\begin{aligned}
  h(q)&\;\in\prod_{m \geq1\,,(m,p)=1}\frac{1}{x}R^\wedge_p[\z_m][\![x]\!]\,,\qquad
  \text{such that}\\
  h(\z_m+x)
  &\=
  \frac{\fr D_{p}(\xi)-p^2D_{p}(\xi)}{m^2p^2\log(q)}
  +\frac{1}{mp}\log\Big(\frac{(\sigma_p\ve(\xi))^{p}}{\ve(\xi)}\Big)+\mathrm{O}(x^2)\,,
\end{aligned}
\ee
we can define an $L_p(\xi)$-section by
\be
\begin{aligned}
\log(f(q)) \= \frac{1}{m}\log(\ve(\xi)) +
\sum_{k=0}^{\infty}\frac{\fr^{k}}{p^k}\Big(
\frac{\fr D_{p}(\xi)-p^2D_{p}(\xi)}{m^2p^{2+k}\log(q)}
+\frac{1}{mp}\log\Big(\frac{\fr\ve(\xi)}{\ve(\xi)}\Big)-h(q^{p^k})\Big)\,,
\end{aligned}
\ee
which is a convergent sum.
\end{proof}

This then has the following corollary, which can be interpreted as saying
that the $\Psi$ of definition~\ref{def.psi} is a local lift of the unit computed
in~\cite{CGZ}.

\begin{corollary}
The collection $\Psi_{\widehat\xi,p}\in\calH_{R^\wedge_p,\widehat\xi}^{\times}/
\calH_{R^\wedge_p}^{\times}$ only depends on $\widehat\xi$ from equation
~\eqref{hatxia} and not on the presentation of equation~\eqref{hatxi}.
\end{corollary}

\begin{proof}
If we have two sets of $a_\z$ that represent $\widehat\xi$ as in equation~\eqref{hatxi},
then the quotient is in $\calH_{R^\wedge_p}^{\times}$.
\end{proof}

\subsection{Proof of the main theorems}
\label{sub.proofmain}

We are now ready to give proofs of our main theorems.

\begin{proof}[Proof of Theorem~\ref{thm.HR} of Section~\ref{sub.habmod}]
Firstly, we see that $\calH_{R}\subseteq\calH_{R,0}$ as $f(q)=1$ is a global invertible
$L(0)$-section. For all $p$ prime to $\Delta$, the constant collection
$1$ is an invertible $L_p(0)$-section. Therefore, from Corollary~\ref{cor:unique.loc}
of Section~\ref{sub.habp} and the gluing condition~\eqref{eq:gluing} of
Definition~\ref{def.HRmod}, if $f\in\calH_{R,0}$, then $f\in\calH_{R_p^\wedge}$.
% A priori,
% this is only a condition about $f_m$ for $(m,p)=1$, but the gluing condition of
% Definition~\ref{def.HRmod} of Section~\ref{sub.habmod} ensures that $f_{mp^d}$
% is determined by $f_m$ via the
% gluing equation~\eqref{gluef}, using that for $\gamma$ prime to $mp$ the operation
% $f_{mp^d}\mapsto f_{mp^d}(q^\gamma)^\gamma f_{mp^d}(q^{-1})$ is injective.
Varying
over all primes shows that $f\in\calH_{R}$.
This completes the proof that $\calH_{R,0}=\calH_{R}$.
The second statement follows from the fact that $\ve_m$ is
multiplicative while $D_{p}$ is additive and that the conditions
equation~\eqref{eq:gluing} and equation~\eqref{fxidef} are multiplicative as a
consequence. The final statement follows from the combination of the previous two.
\end{proof}

\begin{proof}[Proof of Theorem~\ref{thm.2} of Section~\ref{sub.results}]
Firstly, we show that the constant term of $\Phitof_{A,z,m}(x)$
is $\ve_m(\xi)^{1/m}$ times an element of $\BK[\z_m]$.
For $m$ prime to $\Delta$ (where $\Delta$ includes the primes $2$ and $3$ and
finitely many other primes that depend only on the number field $\BK$),
this follows from Theorem 1.6 and equation (14) of~\cite[Thm.1.6]{CGZ} combined
with Hutchinson~\cite{Hutchinson}.

Secondly, we will prove that for a system of $q$-difference equations
associated to the combination in equation~\eqref{eq:gluing}
there is a unique solution in formal power series in $t$. 
We will explicitly describe the case when $N=1$.
(We again omit the case when $N>1$, since it completely analogous but notationally
heavier.) To do this, we will use standard methods in the study of $q$-holonomic
modules. Recall the $q$-difference equation~\eqref{PhiAshift} and define
\be
\psi_{A,\mu}(t,q) \= F_{A}(q^\mu t,q)\,.
\ee
Fix $\gamma\in\BZ_{>0}$ and consider the function
\be
(\mu,\nu,t) \mapsto
\psi^{(\gamma)}_{A,\mu,\nu}(t,q) \=
\psi_{A,\mu_1}(t,q^\gamma)\cdots \psi_{A,\mu_\gamma}(t,q^{\gamma})
\psi_{A,\nu}(t,q^{-1})\,,
\ee
where $\mu \= (\mu_1,\dots,\mu_\gamma) \in \BZ^\gamma$ and $\nu \in \BZ$.
It satisfies $\gamma+2$ equations, corresponding to the number of variables:
\be
\label{eq:qdiff.proofs}
\begin{aligned}
\psi^{(\gamma)}_{A,\mu,\nu}(t,q)
-
\psi^{(\gamma)}_{A,\mu+\delta_i,\nu}(t,q)
&\=
(-1)^Aq^{A\gamma}t\psi^{(\gamma)}_{A,\mu+A\delta_i,\nu}(t,q)\qquad
(i=1,\cdots,\gamma)\,,\\
\psi^{(\gamma)}_{A,\mu,\nu}(t,q)
-
\psi^{(\gamma)}_{A,\mu,\nu+1}(t,q)
&\=
(-1)^Aq^{-A}t\psi^{(\gamma)}_{A,\mu,\nu+A}(t,q)\,,\\
\psi^{(\gamma)}_{A,\mu}(qt,q)
&\=
\psi^{(\gamma)}_{A,\mu+1,\nu-1}(t,q)\,.
\end{aligned}
\ee
These equations have a unique solution for power series in $t$
of the form $1+\mathrm{O}(t)$. This follows from the fact
that the first two equations imply that the coefficient of $t^k$
in $\psi^{(\gamma)}$ is constant in $\mu,
\nu$ modulo coefficients
of smaller powers of $t$. Assuming this, the last equation
implies that this coefficient multiplied by $q^k-1$ is given entirely
by combinations of lower order terms. Therefore,
everything is determined by the coefficient of $t^0$.

From the explicit formulas of equation~\eqref{PhiFGIdef} and
equation~\eqref{Ikdef}, Theorem~\ref{thm.fgi.is.admis} of Section~\ref{sub.synthesis},
Theorem~\ref{thm.FGI2} of Section~\ref{sub.results}, and Galois invariance under
$z^{1/m}\mapsto \z_mz^{1/m}$; we see that
$\psi^{(\gamma)}_{A,\mu,\nu}(t^{1/m},\z_m+x)\in S_p^{(m)}[\![x]\!]$,
for all primes $p$ prime to $\Delta$ and $m$ prime to $\gamma$. Therefore,
$\psi^{(\gamma)}_{A,\mu,\nu}(t^{1/m},\z_m+x)\in S^{(m)}[\![x]\!]$.
Moreover, since there is a unique solution to the equations~\eqref{eq:qdiff.proofs},
if we $p$-complete and re-expand by $x\mapsto x+\z_{pm}-\z_m$
then we must find an equality after applying a Frobenius, which maps $t\mapsto t^p$.
Therefore, these equalities persist when we specialise $t=1$.
This proves that $\Phitof_{A,z}$ satisfies equation~\eqref{eq:gluing}.

For the rest of the proof, we fix a positive integer $m$, a prime $p$ with
$(m,p)=1$ and a congruence class $k \in (\BZ/m\BZ)^N$. Theorems~\ref{thm.FGI2} of
Section~\ref{sub.results} and~\ref{thm.fgi.is.admis} of Section~\ref{sub.synthesis}
imply that each term $\CS_{A,m,k}(t,x)$ of
equation~\eqref{FGIcong}, specialised to $t=1$, satisfies equation~\eqref{dphiA},
which is an essentially equation~\eqref{fxidef} restricted to a disc. 
% Note incidentally that after specialisation, the power series in question have
% positive radius of convergence as follows from~\cite[Thm.9.1]{GZ:kashaev}.
% Since $\CS_{A,m,k}(1,x)$ satisfies equation~\eqref{fxidef} restricted to
% the $p$-adic disc around $q=\z_m$,
% they give $L_{p}(\xi)$-sections restricted to that disc up to universal
% $\z_{p^d}$ roots of unity for the series around $q=\z_{mp^d}$
% for the larger ring where we include $z^{1/m}$.
Summing up using~\eqref{PhiFGIdef} and~\eqref{FGIcong}, we obtain that 
\be
\Phitof_{A,m}^\FGI(1,x) \=\!\!\!\!
\sum_{k \in (\BZ/m\BZ)^N}
\frac{(-1)^{\mathrm{diag}(A)\cdot k}q^{\frac{1}{2}(k^t A k + \mathrm{diag}(A)\cdot k)}}
{(q;q)_{k_1}\cdots(q;q)_{k_N}}\CS_{A,m,k}(1,q)\,,\quad q\=\z_{m}+x \,.
\ee
It follows that $\Phitof_{A}^\FGI(1,q)$ gives an element of the
module $\calH_{R[\delta^{-1/2}]^\wedge_p[\z_m,z^{1/m}],\xi}$ for $q$-restricted
to the disc around $\z_m$.
Finally, from the invariance under $z^{1/m}\mapsto\z_m z^{1/m}$ from
Lemma~\ref{lem.PhiA} of Section~\ref{sub.FGI}, we see that $\Phitof_{A,z}(q)$
descends to an element of $\calH_{R[\delta^{-1/2}]^\wedge_p,\xi}$ restricted to
the disc around $\z_m$.

Varying over all $m$ prime to $\Delta$ and combining this with the previous
result that equation~\eqref{eq:gluing} is satisfied, we conclude that
$\Phitof_{A,z}(q)\in\calH_{R[\delta^{-1/2}],\xi}|_\Delta$.
\end{proof}

\begin{proof}[Proof of Corollary~\ref{cor.3} of Section~\ref{sub.results}]
This would immediately follow from Theorem~\ref{thm.2} of Section~\ref{sub.results},
Theorem~\ref{thm.HR} of Section~\ref{sub.habmod}, and Proposition~\ref{prop.HRxi}
of Section~\ref{sub.habmod}; if we did not have the technical
assumption on the order of roots of unity being prime to $\Delta$.

However, equation~\eqref{symhab} does immediately follow from the proof of
Theorem~\ref{thm.2} of Section~\ref{sub.results} that showed $\Phitof_{A,z}$ satisfies
equation~\eqref{eq:gluing}, which in particular holds when $\gamma=1$.

For part (b), a weaker statement asserting that
$\Phitof_{A,z}(q)^{r} \in \calH_{R[1/\sqrt{\delta}]}|_\Delta$ follows from
Theorem~\ref{thm.2} of Section~\ref{sub.results} and Theorem~\ref{thm.HR} of
Section~\ref{sub.habmod}. To replace $\calH_{R[1/\sqrt{\delta}]}|_\Delta$
by $\calH_{R[1/\sqrt{\delta}]}$, note the surjective map $K_3(\BK) \to B(\BK)$
of Suslin (see e.g.,~\cite[Eqn.(1.1)]{Zickert:bloch}). Hence, if $r\xi=0$, then
the image of $\xi$ in the Bloch group can be written as a sum of $5$-term relations
in the Bloch group. Then the identity of~\cite{KMS} that was used in~\cite{CGZ}
shows that the constant term at $\z_m$ is in $S^{(m)}$.
\end{proof}

\begin{remark}
\label{rem.desc}
In fact, the specialisation $t=q^\nu$ for $\nu \in \BZ^N$ (i.e., $t_j=q^{\nu_j}$
for integers $\nu_j$) defines an element $\Phitof_{A,z,\nu}(q) \in
\calH_{R[1/\sqrt{\delta}],\xi}|_\Delta$.
The proof is identical to the proof of Theorem~\ref{thm.2} of
Section~\ref{sub.results} and is omitted.
\end{remark}

\begin{proof}[Proof of Proposition~\ref{prop.HRxi} of Section~\ref{sub.habmod}]
The proofs of the proposition are all elementary and outlined as follows:
\begin{itemize}
\item
  Property (a) and (d) follow from the $\chi^{-1}$ equivariance of $\ve$.
\item
  Property (b) follows from Theorem~\ref{thm.HR} of Section~\ref{sub.habmod}.
\item
  Property (c) is trivial and follows from the definition of $\gamma^{*}$ just prior
  to the proposition.
\item
  Property (e) this property follows from the fact that, varying over $\gamma,\gamma'$,
  the collections $f(q^{\gamma_1})\cdots f(q^{\gamma_n})f(q^{-\gamma_1'})
  \cdots f(q^{-\gamma_{n'}'})$,
  where $\sum_{k}\gamma_k^{-1}-\sum_k\gamma_{k}'^{-1}=0$ determine
  $f(q)$. Alternatively, one could use the remark
  % ~\ref{rem:mero}
  of Section~\ref{sub.habp} to connect
  various roots of unity.
  \item Property (f) follows from the fact that the constants
  $f_m(0) \in R^\wedge_p[\z_m,\ve_m^{1/m}]$ for all $m$ and
  $p$ with $(mp,\Delta)=1$.  
\end{itemize}
\end{proof}

An immediate corollary of part (f) is the following result, which even in very simple
instances like the $4_1$-knot
seems very difficult to prove directly.

\begin{corollary}
For $m$ prime to $\Delta$, the constants $U_{m}^\FGI(1)$ from equation
~\eqref{Umtdef} are $\Delta$-integral.
\end{corollary}

%%%%%%%%%%%%%%%%%%%%%%%%%%%%%%%%%%%%%%%%%%%%%%%%%%%%%%%%%%%%%%%%%%%%%%%%%%%% 
%%%%%%%%%%%%%%%%%%%%%%%%%%%%%%%%%%%%%%%%%%%%%%%%%%%%%%%%%%%%%%%%%%%%%%%%%%%%

\section{Examples and computations}
\label{sec.examples}

\subsection{Symmetrisation and a residue formula}
\label{sub.sym}

Elements of the usual Habiro ring $\calH_{\BZ}$ are extremely easy to write down.
For example, the ring $\BZ[q]$ is a subring of $\calH_{\BZ}$. This is not the case
for the Habiro ring of a number field, where $\calO[q]$ is no longer a subring
of $\calH_{\calO[1/\Delta]}$. One can construct
elements of these rings for certain presentations of a number field. In this section
we explain how to get formulas for such elements using combinatorial data. 

Fix a symmetric integer matrix $A$, and consider the following expression
\be
J_A(t,w,q)
\=
\sum_{n\in\BZ^N_{\geq0}}
\frac{(-q^{\frac{1}{2}})^{n^{\Tran}An}q^{\frac{1}{2}\diag(A)\cdot n}w^{An}
t^n}{(qw;q)_n}\,,
\ee
where $t^n=t_1^{n_1}\cdots t_N^{n_N}$ and
$(qw;q)_n=(qw_1;q)_{n_{1}}\cdots(qw_N;q)_{n_{N}}$.

We can expand this sum when $q=\z_m+x$
is near a root of unity $\z_m$ and obtain an element
of $\BZ[\z_m][w^{\pm1},(1-w)^{-1}]\llbracket t\rrbracket\llbracket x \rrbracket$.
In fact, it is not hard to see that 
\be
\label{Jwdef}
J_A(t,w,\z_m+x) \inn \BZ[\z_m][t,w_i^{\pm1},
(1-t_i^{m}P_i(w^{m}))^{-1} \, | \, i\=1,\dots,N]
\llbracket x\rrbracket\,,
\ee
where
\be
\label{ABge}
P_{i}(z)\= (-1)^{A_{ii}} (1-z_i)^{-1} \prod_{j=1}^N z_j^{A_{ij}},\qquad i=1,\dots,N\,.
\ee

We define the collection $\Psi_A(t,q) \= (\Psi_{A,m}(t,x))_{m \geq 1}$ by
\be
\label{psitmdef}
\Psi_{A,m}(t,x)
\=
\Res{w^{m}=z}
t^{-1/m}J_A(t^{1/m},w,\z_m+x)\,\frac{dw}{w}\,,
\ee
where the residue is taken over all $w^{m}=z$ and $z$ satisfies the equations
\be
\label{tpz}
t_i P_i(z)\=1, \qquad i\=1,\dots, N \,. 
\ee

The value at $\z_m$ is given by the manifestly integral formula
\be
\Psi_{A,m}(t,0)
\=
\frac{1}{\delta_A(t) m^N}
\sum_{w^m=z}
\sum_{n\in\BZ^{N}/m\BZ^{N}}
\frac{(-\z^{\frac{1}{2}})^{n^{\Tran}An}\z^{\frac{1}{2}\diag(A)\cdot n}w^{An}
t^n}{(qw;q)_n} 
\ee
with $\delta_A(t)$ as in~\eqref{dt}. 
The next theorem says in particular that this number is the same as the symmetrised
value $\Phitof_A^{\mathrm{sym}}(t,\z_m)=\Phitof_A(t,\z_m)\Phitof_A(t,\z_m^{-1})$, which
therefore is also integral.

\begin{theorem}
\label{thm.sym}
For every positive integer $m$ and $\Phitof_A$ of Theorem~\ref{thm.1} in
Section~\ref{sub.results}, we have
\be
\label{phipsi}
\Psi_A(t,q) \= \Phitof_A(t,q) \Phitof_A(t,q^{-1}) \inn
\BZ[\z_m][\![x]\!][\![t^{1/m}]\!], \qquad q\=\z_m+x \,.
\ee
\end{theorem}

\begin{proof}
The series $\Psi_A(t,q)$ and $\Phitof_A(t,q) \Phitof_A(t,q^{-1})$ are the specialisations
to $\mu=\nu=0$ of two families indexed by two integer vectors $\mu,\nu\in\BZ^{N}$ 
\be
\label{2Phi}
\begin{aligned}
\Phitof_{A,\mu,\nu}(t,q) & \= \Phitof_A(q^{m\mu}t,q) \Phitof_A(q^{-m\nu}t,q^{-1})\,,
\\
\Psi_{A,\mu,\nu}(t,q) & \= \Res{w^{m}=z}
q^{\nu}w^{\mu+\nu}t_1^{-1/m}\cdots t_N^{-1/m}J_A(q^{\mu}t^{1/m},w,\z_m+x)\,
\frac{dw}{w}\,.
\end{aligned}
\ee
Both families satisfy the same system of $q$-difference equations. In the
one-dimensional case (i.e. $N=1$), for $q=\z_m+x$ we have
\be
\label{eq:two.qdiffs}
\begin{aligned}
\Phitof_{A,\mu,\nu}(t,q)-\Phitof_{A,\mu+1,\nu}(t,q)
&\=
(-1)^{A}q^{A}t^{1/m}\Phitof_{A,\mu+A,\nu}(t,q)\,,\\
\Phitof_{A,\mu,\nu}(t,q)-\Phitof_{A,\mu,\nu+1}(t,q)
&\=
(-1)^{A}q^{-A}t^{1/m}\Phitof_{A,\mu,\nu+A}(t,q)\,,\\
\Phitof_{A,\mu,\nu}(q^mt,q)
&\=
\Phitof_{A,\mu+1,\nu-1}(t,q)\,.
\end{aligned}
\ee
To see that $\Psi_{A,\mu,\nu}(t,q)$ also satisfies these equations note
that
\be
\begin{aligned}
J_A(t,w,q)-wJ_A(qt,w,q)
% \=
% \sum_{n\in\BZ^N_{\geq0}}
% (-1)^{An}
% \frac{q^{An(n+1)/2}w^{An}t^n}
% {(qw;q)_{n-1}}
% \=
% \sum_{n\in\BZ_{\geq-1}}
% (-1)^{A(n+1)}
% \frac{q^{An(n+1)/2+An+A}w^{An+A}t^{n+1}}
% {(qw;q)_{n}}
&\=
-(1-w) +(-1)^Aq^{A}tw^{A} J_A(q^{A}t,w,q)\,,\\
(1-qw)J_A(t,w,q)
% \=
% \sum_{n\in\BZ^N_{\geq0}}
% (-1)^{An}
% \frac{q^{An(n+1)/2-An}(qw)^{An}t^n}
% {(q^2w;q)_{n-1}}
% \=
% \sum_{n\in\BZ_{\geq-1}}
% (-1)^{An+A}
% \frac{q^{An(n+1)/2}(qw)^{An+A}t^{n+1}}
% {(q^2w;q)_{n}}
&\=
-(1-qw) +(-1)^{A}q^{-A}tq^{A}(qw)^{A}J_A(t,qw,q)\,.
\end{aligned}
\ee
Therefore after taking the residues with a change of variables in the RHS of the
second equation $qw\mapsto w$ we find that $\Psi_{A,\mu,\nu}(t,q)$ satisfies the
equations~\eqref{eq:two.qdiffs}. In both cases the solutions can be taken of the form
\be
\sum_{k=0}^{\infty} a_{\mu,\nu,k}(q)\,t^{k/m} \inn\BZ[\z_m][\![x]\!][\![t^{1/m}]\!]\,.
\ee
The $q$-difference equations imply that for
\be
\begin{aligned}
a_{\mu,\nu,k}-a_{\mu+1,\nu,k}
&\=
(-1)^{A}q^{A}a_{\mu+A,\nu,k-1}\,,\\
a_{\mu,\nu,k}-a_{\mu,\nu+1,k}
&\=
(-1)^{A}q^{-A}a_{\mu,\nu+A,k-1}\,,\\
q^ka_{\mu,\nu,k}
&\=
a_{\mu+1,\nu-1,k}\,,
\end{aligned}
\ee
where we set $a_{\mu,\nu,k}=0$ for $k<0$. This completely determines all $a_{\mu,\nu,k}$
from the value of~$a_{0,0,0}$. Notice that $\Phitof_{A,m}(0,q)\Phitof_{A,m}(0,q^{-1})=1$.
To see that $\Psi_{A,\mu,\nu}(0,q)$ also equals $1$, notice that only the coefficient
of $x^{0}$ contributes to the coefficient of $t^{0}$, because the residues of the
factors $(1-t_i^{m}P_i(w^{m}))^{-\ell}$ contribute at least one factor of $t$ when
$\ell>0$. Therefore an explicit computation leads to $\Psi_{A,m}(0,x)=1$.
\end{proof}

To get elements of the Habiro ring of a number field, we specialise to 
$t=1$ and assume that the equations~\eqref{tpz} define a reduced
zero-dimensional scheme over $\BQ$. Fix a solution $z$ of these equations
and denote the corresponding collection of power series by $\Psi_{A,\mu,\nu,z}(q)$.
A solution $z$ generates a number field $\BK$.
Combining Theorem~\ref{thm.1} of Section~\ref{sub.results}, Theorem~\ref{thm.sym}
and Theorem~\ref{thm.2} of Section~\ref{sub.results}, we obtain that
$\Psi_{A,\mu,\nu,z}(q) \in \calH_{R}$.

\begin{theorem}
\label{thm.desc}
For all $\mu, \nu \in \BZ^N$, we have $\Psi_{A,\mu,\nu,z}(q) \in \calH_R$.
\end{theorem}

Some cases of the above theorem were first proven in the thesis of Ferdinand Wagner.
Here is a concrete example for the cubic field of discriminant $-23$.

%% see pari file: digits.for.52.padic.stavros.with.tvar.pari

\begin{example}
\label{ex.23}
Consider the sum
\be
J(t,w,q)
\=
\sum_{k=0}^{\infty}
(-1)^{k}\frac{q^{3k(k+1)/2}w^{3k}t^k}{(qw;q)_{k}} \,.
\ee
Expanding $J(t,w,1+x)$ as a power series in $x$ and observing that each coefficient
is a sum of derivatives of geometric series in $t$, we find that
\be
\begin{small}
\begin{aligned}
J(t,w,1+x)
&\=
\frac{-1 + w}{-tw^3+w-1}
+\frac{3tw^3 - 5tw^4 + 2tw^5}{(-tw^3+w-1)^3}x \\
&+\frac{1}{(-tw^3+w-1)^5} \Big( 3tw^3 - 9tw^4 + 10tw^5 + (-21t^2 - 5t)w^6
+ (50t^2 + t)w^7 \\ &\qquad\qquad - 39t^2w^8 + (3t^3 + 10t^2)w^9
- 4t^3w^{10} + t^3w^{11} \Big) x^2
+O(x^3)
\end{aligned}
\end{small}
\ee
and
\be
\begin{aligned}
\Psi_{1}(t,x)
&\=
\frac{2tz^2 + 3tz - 9t}{27t-4}
+\frac{x^2}{(27t-4)^4}\big((-4374t^4 - 2106t^3 + 3t^2)z^2\\
&\qquad+ (-2187t^4 - 2997t^3 - 129t^2 + 2t)z
+ (2916t^3 + 1404t^2 - 2t)\big)
+\mathrm{O}(x^3) \,,
\end{aligned}
\ee
where $z$ satisfies the equation
\be
1-z\=-t z^3 \,.
\ee
Specialising to $t=1$ we find the expansion 
\be
\Psi_{1}(x)
\=
\frac{2z^2 + 3z - 9}{23}
+
\frac{-6477z^2 - 5311z + 4318}{23^4}x^2
+O(x^3)
\ee
of the element of the Habiro ring $R=\BZ[z,\tfrac{1}{23}]$, where $z^3-z+1=0$
generates the cubic field of discriminant $-23$.
Using
\be
  \delta\=-2t-\frac{t}{1-z}\=-tz^2 - tz + (-3t + 1)
\ee
one can compare these equations to equation~\eqref{fsymvals3} and
equation~\eqref{eq:ex.fsym.admis1}.
% $\xi^3-\xi^2+1=0$
% $z=1-\xi^2$
\end{example}

%% older pari file: digits.for.52.padic.stavros.pari

The next remark is for knot-theorists.

\begin{remark}
\label{rem.knots}
Use the matrices
\be
\label{3knots}
A_{4_1} \= \begin{pmatrix} 1 & 1 \\ 1 & 1 \end{pmatrix},
\qquad
A_{5_2} \= \begin{pmatrix} 2 & 1 & 1 \\ 1 & 1 & 0 \\ 1 & 0 & 1 \end{pmatrix},
\qquad
A_{(-2,3,7)} \=
\begin{pmatrix} 1 & 0 & 1 \\ 0 & 1 & 2 \\ 1 & 2 & 4 \end{pmatrix}
\ee
in Theorem~\ref{thm.2} of Section~\ref{sub.results} to compute the asymptotic
series $\Phitof^{(K)}(h)$ of the three simplest hyperbolic knots (for the $(-2,3,7)$
pretzel knot, the matrix was given in~\cite[Rem.A.5]{GZ:qseries}). To get the
asymptotic series of~\cite{DG} for any knot, one can use a triangulation of it
from \texttt{SnapPy}, choose quad types with Neumann--Zagier matrices
$(\mathbf A|\mathbf B)$ satisfying that $\mathbf B^{-1} \mathbf A$ is integral
(if this is possible) and apply Theorem~\ref{thm.2} of Section~\ref{sub.results}
with $A=\mathrm{I}- \mathbf B^{-1} \mathbf A$. 
\end{remark}

\subsection{A rank one admissible series}
\label{sub.F3}

$q$-hypergeometric series give admissible series that are easy to analyse. We
illustrate this with the example of the $1 \times 1$ matrix $A=(3)$ in~\eqref{FAdef}

\be
\label{Fexdef}
F(t,q) \= \sum_{k=0}^\infty \frac{(-1)^k q^{3 k(k+1)/2}}{(q;q)_k} t^k
\inn \BQ(q)[\![t]\!] \,.
\ee

%% see Mathematica file: kontsevich/DTInvariants.Admissible.Series.Rank1.nb

It satisfies the linear $q$-difference equation 

\be
\label{Frec}
F(t,q) - F(q t,q) + q^3 t F(q^3 t,q) \= 0 \,.
\ee

The DT invariant of~\eqref{logF} $c_{n,i}$ is non-zero only for
$3n+1 \leq i \leq n^2+n+1$ (and exceptionally, for $c_{1,3}=1$) and satisfy the
positivity $c_{n,i} \in \BN$. The first few values of $c_{n,i}$ are given by

\be
\label{DTvalues}
\begin{array}{|r|c|} \hline
  n &  c_{n,i}, \,\, i=3n+1,\dots,n^2+n+1 \\ \hline
%  1 & 1 \\ %\hline
  2 & 1 \\ %\hline
  3 & 1, 1, 0, 1 \\ %\hline
  4 & 1, 1, 2, 1, 2, 1, 1, 0, 1 \\ %\hline
  5 & 1, 2, 3, 4, 4, 5, 4, 4, 3, 3, 2, 2, 1, 1, 0, 1 \\ %\hline
  6 & 1, 2, 5, 7, 11, 11, 15, 13, 15, 13, 14, 10, 12, 8, 8, 6, 6, 3, 4, 2, 
2, 1, 1, 0, 1 \\ \hline
\end{array}
\ee
Despite appearances in low degrees, the DT invariants grow fast exponentially.
For example,
\be
c_{20,142} \= 44549701024 \,. 
\ee
The WKB expansion (see e.g.,~\cite{Bender}) of the solution $F(t,q)$
of~\eqref{Frec} allows one to compute the leading asymptotics of $\Phitof_1(t,x)$
in terms of the power series $z(t)$ and $V(t)$ defined by
\be
\label{z3}
1 - z  \= - t z^3, \qquad z(0)\=1\,,
\ee
where
\be
\label{zfew3}
z(t) \= 1 + t + 3 t^2 + 12 t^3 + 55 t^4 + 273 t^5 + 1428 t^6 + 7752 t^7 + 
 43263 t^8 + 246675 t^9 + O(t^{10}) % 1430715 t^{10} 
\ee
and
\be
\label{V3}
\begin{aligned}
V(t) \= & -\Li_2(1-z(t)) -\frac{3}{2} (\log(z(t)))^2
\\
\= & t + \frac{5}{4} t^2 + \frac{28}{9} t^3 + \frac{165}{16} t^4 + \frac{1001}{25} t^5
+ \frac{1547}{9} t^6 + \frac{38760}{49} t^7 + \frac{245157}{64} t^8
+ O(t^9) \,. %\frac{1562275}{81} t^9 
\end{aligned}
\ee

%% see: pari file A.3.Nahm.sum.admissible.series.q.1.pari.stavros

Using the auxiliary function 
\be
\label{delta3}
\begin{aligned}
\delta(t) & \=  -2t-\frac{t}{1-z(t)} \\
& \=1 - 5 t - 3 t^2 - 10 t^3 - 42 t^4 - 198 t^5 - 1001 t^6
- 5304 t^7 - 29070 t^8 + \dots %- 163438 t^9 - 937365 x^10 - 5462730 x^11
\inn \BZ[\![t]\!]
\end{aligned}
\ee
and setting $q=1+x$ and abbreviating $V(t)$, $\delta(t)$ and $z(t)$ by $V$, $\delta$
and $z$, respectively, we obtain the first few coefficients of $\Phitof_1(t,x)$ as
follows:
\be
\label{Phi1tx3}
\begin{aligned}
&\widehat{\Phitof}_1(t,x) \= e^{\tfrac{V}{x}} \frac{1}{\sqrt{\delta}} \Big(1 +
\frac{1}{24 \delta^3}
((308 t^3 - 74 t^2) z^2 + (234 t^3 - 74 t^2) z + (216 t^3 - 382 t^2 + 74 t)) x
\\
&
+\frac{1}{1152\delta^6}(
(748116 t^6 - 893688 t^5 + 281084 t^4 - 19924 t^3 - 1104 t^2) z^2
\\ &  + (397872 t^6 - 685624 t^5 + 257848 t^4 - 21028 t^3 - 1104 t^2) z
\\ &  + (186624 t^6 - 1252224 t^5 + 1127196 t^4 - 303216 t^3 + 18820 t^2 + 1104 t))
  x^2 + O(x^3) \Big) \,.
\end{aligned}
\ee
In general, the coefficient of $x^k$ in $\Phitof_1(t,x) e^{-\tfrac{V(t)}{x}}
\delta^{\frac{1}{2} + 3k}$ is in $\BQ[t,z]$, where $z$ satisfies~\eqref{z3}.

Recall the symmetrisation $G(t,q)$ and $F^\sym(t,q)$ from equation~\eqref{GSdef}.
equation~\eqref{Frec} implies that $G$ satisfies a non-linear (Ricatti type)
$q$-difference equation 

\be
\label{Grec}
1-G(t,q) +q^3 t \, G(t,q) \, G(q t,q) \, G(q^2 t,q)  \= 0 \,.
\ee
It follows from this that the limit $\lim_{q \to 1} G(tq,q) = z(t)$ exists and
satisfies the algebraic equation~\eqref{z3}. Equation~\eqref{GS2} gives
the factorisation
\be
\label{gexp}
z(t) \= \prod_{n \geq 1} %\prod_{i \in \BZ}
(1-t^n)^{-n \sum_{i \in \BZ} \, c_{n,i}}
\ee
which proves that the exponent of $1-t^n$ in the above product expansion
is divisible by $n$. 
This is an application of~\cite{KS:cohomological}, where one may think of the
polynomials $L_n(q)$ as a categorification of the integers $n \, c_{n,i}$. 

The symmetrised series $F^\sym(t,q)$ was not considered previously in
the literature. It is easy to show that it satisfies a sixth order linear
$q$-difference equation, which we omit.
% Its Newton polygon has one horizontal edge
% of length $1$ and slope $1$ corresponding to the series $F^\sym(t,q)$.
Below, we will treat the series $G(t,q)$ and $F^\sym(t,q)$ on the same footing.

We now illustrate a remarkable aspect of the series $G(t,q)$ and $F^\sym(t,q)$,
namely their $(q-1)$-expansion. It is clear that they both lie in the
completed ring $\BZ[\![t]\!][\![q-1]\!]$ but more is true. The
expansion~\eqref{Phi1tx3} contains a volume prefactor that cancels, as well
as universal denominators for each power of $x$, which also cancel, so that what
remains are series in $\BZ[t^{\pm 1},z,1/\delta]$. Explicitly, we can write $G(t,q)$ 
\be
\label{Gexp}
G(t,1+x) \= \sum_{k \geq 0} g_k(t) x^k
\ee
and then it follows from~\eqref{Grec} and induction that $g_0=z$ and 
$\delta^{3k} g_k \in \BZ[t^{\pm 1}, z]$, e.g., 

%% see pari file: A.3.Nahm.sum.admissible.series.q.1.pari.stavros

\be
\label{gvals3}
\begin{aligned}
g_0 \= & z\,,
\\
\delta^3 g_1 \= & (15 t^3 - 3 t^2) z^2 + (18 t^3 - 3 t^2) z + (-18 t^2 + 3 t)\,,
\\
\delta^6 g_2 \= & (711 t^6 - 708 t^5 + 107 t^4 + 17 t^3 - 3 t^2) z^2
+ (405 t^6 - 586 t^5 + 115 t^4 + 14 t^3 - 3 t^2) z
\\ & + (-1176 t^5 + 828 t^4 - 96 t^3 - 20 t^2 + 3 t)\,,
\\
\delta^{9} g_3 \= &
(26325 t^9 - 69399 t^8 + 32035 t^7 + 6234 t^6 - 6470 t^5 + 1259 t^4 - 69 t^3 - t^2) z^2
\\ &
+ (9720 t^9 - 47322 t^8 + 29899 t^7 + 2658 t^6 - 5430 t^5 + 1187 t^4 - 70 t^3 - t^2) z
\\ &
+ (-56187 t^8 + 95787 t^7 - 32141 t^6 - 10699 t^5 + 7584 t^4 - 1330 t^3 + 68 t^2 + t)
\,.
% \\
%g_4(t) &=
\end{aligned}
\ee

\noindent
Likewise, we have
\be
\label{Fsymexp}
F^\sym(t,1+x) \= \sum_{k \geq 0} f^\sym_k(t) x^k\,,
\ee
where $\delta^{3k+1} f^\sym_k \in \BZ[t^{\pm 1}, z]$, with the first few
values given by 

\be
\label{fsymvals3}
\begin{aligned}
\delta f^\sym_0 \= & 1\,,
\\
\delta^4 f^\sym_1 \= & 0\,,
\\
\delta^7 f^\sym_2 \= &
(39 t^6 - 109 t^5 - 18 t^4 + 34 t^3 - 5 t^2) z^2
+ (9 t^6 - 85 t^5 + t^4 + 29 t^3 - 5 t^2) z
\\ & + (-96 t^5 + 124 t^4 + 42 t^3 - 39 t^2 + 5 t)\,,
\\
\delta^{10} f^\sym_3 \= &
(1296 t^9 - 9183 t^8 + 7230 t^7 + 1604 t^6 - 2730 t^5 + 858 t^4 - 109 t^3 + 5 t^2) z^2
\\ &
+ (243 t^9 - 5328 t^8 + 6310 t^7 + 415 t^6 - 2139 t^5 + 764 t^4 - 104 t^3 + 5 t^2) z
\\ &
+ (-3969 t^8 + 14304 t^7 - 7635 t^6 - 3231 t^5 + 3405 t^4 - 957 t^3 + 114 t^2 - 5 t)
\,.
% \\
%\delta^{13} f^\sym_4 =&
%\\
%f^\sym_5(t) &=
%\\
%f^\sym_6(t) &=
\end{aligned}
\ee

But a further surprise is awaiting us when we expand $G(t,q)$ and $F^\sym(t,q)$
at $q=\z_m+x$. As expected, we now get series in
$\BZ[\z_m][t^{\pm 1},z,1/\delta][\![x]\!]$. But these series glue, after applying
$p$-Frobenius. Concretely, if we specialise $t=1$, then we obtain
\be\label{eq:ex.fsym.admis1}
F^\sym(1,1+x) \= \frac{1}{\delta} + (- 59 z^2 - 51 z + 36) \frac{x^2}{\delta^7} +
(- 1029 z^2 + 166 z + 2026) \frac{x^3}{\delta^{10}} + O(x^4) \,,
\ee
where $z$ satisfies the cubic equation $1-z+z^3=0$ and
$\delta=-2-1/(1-z)=-z^2-z-2$, an algebraic integer of norm $-23$. $z$ generates
a cubic field $F=\BQ(z)$ of discriminant $-23$. Let $R=\BZ[z,1/23]$.

The specialisation $F^\sym(1,q)\in \BZ[z,\frac{1}{23}][\![q-1]\!]$ and
$G(1,q)$ lie in $\calH_{\BZ[z,1/23]}$ and its field of fractions, respectively.
This illustrates part (a) of Corollary~\ref{cor.3} of Section~\ref{sub.results}.

These are elements of the Habiro ring of the \'etale map
\be
\label{Hetale}
\BZ[t^{\pm 1},1/\delta] \to \BZ[t^{\pm 1},1/\delta][z]/(1-z + t z^3) \,.
\ee

When $m=1$ and $A$ is a $1 \times 1$ matrix, the power series $z(t)=z_A(t)$,
$V(t)=V_A(t)$ and $\delta(t)=\delta_A(t)$ have coefficients polynomials in $A$.
Indeed, the unique power series $z(t)=z_A(t)$ 
\be
\label{gAdef}
1-z\=(-1)^A t z^A, \qquad z(0)\=1 \,.
\ee
has coefficients integer-valued polynomials of $A$ (for integer $A$), with the
first few given by
\be
\label{fewgt}
z(t) \= 1- (-1)^{A} t + A t^2 - \frac{1}{2}(-1)^{A} A(3A-1) t^3 + O(t^4) \,.
%  +\frac{1}{3}A(4A-1)(2A-1)t^4 + O(t^5) \,.
\ee
This, together with 
\begin{align*}
z(t) & \= \lim_{q \to 1} \frac{F(tq,q)}{F(t,q)} \= \lim_{q \to 1}
\exp\Big(-\sum_{n, \ell \geq 1} \sum_{i \in \BZ} \frac{L_n(q^\ell)}{\ell(1-q^\ell)}
(q^{\ell n}-1) t^{\ell n} \Big)
\\
& \= \exp\big( \sum_{n \geq 1} n L_n(1) \Li_1(t^n) \big)
\= \exp\big( t \partial_t \sum_{n \geq 1} L_n(1) \Li_2(t^n) \big)
\= \exp\big( t \partial_t V(t) \big)\,,
\end{align*}
implies that $V(t)$ satisfies 
\be 
\label{g2V}
\begin{aligned}
V(t) 
& \= -\Li_2(1-z(t)) - \frac{A}{2} (\log(z(t)))^2 \inn \BQ[\![t]\!],
\qquad V(0)\=0 \\
& \= -(-1)^A t -\frac{2A-1}{4} t^2 + (-1)^A \frac{(3A-1)(3A-2)}{18} t^3
+ O(t^4) \,.
\end{aligned}
\ee
Finally,
\be
\label{deltadef}
\begin{aligned}
\delta(t) & \= (-1)^A \Big( (A-1)t + \frac{t}{1-z(t)} \Big) \inn \BZ[\![t]\!],
\qquad \delta(0)\=1 \\
& \=
1 - (-1)^A (2A-1) t + \frac{A(A-1)}{2} t^2 + (-1)^A \frac{(2A-1)A(A-1)}{3} t^3
+ O(t^4) \,.
\end{aligned}
\ee

% The semiclassical approximation of $F(t,q)$ is expressed in terms of $V(t)$ and
% of $\delta(t)$ defined by 
% see pari file: A.3.Nahm.sum.admissible.series.q.1.pari.stavros
% see Mathematica file: DTInvariants.Admissible.Series.Rank1.nb

The rest of the coefficients of the unique zero-slope solution $F_A(t,q)$
of~\eqref{FArec} can be computed inductively as was illustrated in the beginning
of this section with the example of~$A=3$. 

\subsection{Torsion in the Bloch group from admissible series}
\label{sub.8554}

In this section we illustrate part (b) of Corollary~\ref{cor.3} of
Section~\ref{sub.results} with the matrix
\be
\label{A2ex}
A\=\begin{pmatrix} 8 & 5 \\ 5 & 4 \end{pmatrix}
\ee
(symmetric and positive definite) taken from the survey article~\cite{Zagier:dilog}.

The Nahm equations for $z=(z_1,z_2)$
\be
\label{z1z28554}
1-z_1 \= z_1^8 z_2^5, \qquad 1-z_2 \= z_1^5 z_2^4\,,
\ee
have eight solutions in two Galois orbits defined over two quartic fields, one given
by
\be
\label{z1eqn}
  z_1^4 + z_1^3 + 3z_1^2 - 3z_1 - 1\=0\,,\qquad z_2\=\frac{1}{5}(-9z_1^3 - 6z_1^2 - 25z_1 + 37)\,\phantom{,}
\ee
and the other by
\be
  z_1^4 - z_1^3 + 3z_1^2 - 3z_1 + 1\=0\,,\qquad z_2\=z_1^3 + 3z_1\,.
\ee
% satisfying the equation
% %% see Mathematica file: nahm/TorsionElementNahmEquation.Rank2.nb
% \be
% \label{z1eqn}
% (z_1^4 + z_1^3 + 3z_1^2 - 3z_1 - 1)
% (z_1^4 - z_1^3 + 3z_1^2 - 3z_1 + 1)
% \=0\,,
% \ee
% and $z_2$ given by
% \be
% z_2
% \=
% \frac{1}{10}(-16z_1^7 - 3z_1^6 - 84z_1^5 + 79z_1^4 - 147z_1^3 + 245z_1^2 - 124z_1 + 33)
% \,.
% \ee
Since $A$ is positive definite equation~\eqref{z1z28554} has a unique solution in
$(0,1)^2$, given to a few decimals by $(0.88483\cdots,0.78939\cdots)$, belonging
to the real embedding of the quartic number field $F$ of type $(2,1)$ and discriminant
$-5^2 \cdot 19$ defined by~\eqref{z1eqn}.
This solution of the Nahm equation is non-degenerate and defines an element
$\xi=[z_1]+[z_2]$ of the Bloch group~$B(F)$. It turns out to be 60-torsion,
and the corresponding series $(\Phitof_{A,z,\nu})^{60}$ belongs to
$\calH_{\calO_F[1/(5 \cdot 19)]}$ for all~$\nu \in \BZ^2$ (conjecturally, but provably
if we invert $6$). The $q$-holonomic module of the Nahm sum associated to $A$
has rank 8, spanned by $\BQ(q)$-linear combinations of the Nahm sums $F_\nu(q)$
defined by
\be
\label{nahm8554}
F_\nu(q) \= \sum_{n=(n_1,n_2) \in \BN^2} \frac{q^{\frac{1}{2} n^t A n + n^t \nu}}{
(q;q)_{n_1} (q;q)_{n_2}}, \qquad \nu \inn \BZ^2 \,.
\ee
The radial asymptotics of these $q$-hypergeometric functions as $q$ approaches
a root of unity allow one to compute the asymptotic series $\Phitof_{A,z,\nu}$
(abbreviated by $\Phitof$ for $\nu=(0,0)$ below). This
is done using a numerical computation of the function $F_\nu$, followed by an
acceleration that improves the precision of the found numbers, and their eventual
recognition as exact algebraic numbers (explained in detail
in~\cite{GZ:kashaev,GZ:qseries}). Applying this
method, we find that the asymptotics of $F_0(1+x)$ for
$x\in\BR$ as $x\rightarrow 0$ have the form
\be
F_0(1+x) \;\sim\; \widehat{\Phitof}_{1}((1+x)^{-1}+1)\,,
\ee
where
\be
\widehat{\Phitof}_{1}(x) \= 
%\exp\Big(\frac{\pi^2}{15}\frac{\tau}{2\pi i}\Big)
e^{\frac{\pi^2}{15} \frac{1}{\log(1+x)}}
\frac{1}{\sqrt{\delta}}
\big(1+a_{1}x+a_{2}x^2+a_{3}x^3 + O(x^4)\big)
\ee
with $\delta$ and $a_{i}$ in $F$ given by
\be
\begin{aligned}
\delta&\=\frac{753-505z_1-124z_1^2-186z_1^3}{5}\,,\\
a_{1} &\= \frac{-1284z_1^3 + 384z_1^2 - 5520z_1 + 2047}{2^2\cdot3\cdot5^2\cdot19^2}\,,
\\
a_{2} &\= \frac{-3084024z_1^3 - 11262336z_1^2 - 1073760z_1 + 17201653}{
2^5\cdot3^2\cdot5^2\cdot19^4}\,, \\
a_{3} &\= \frac{-1185017476284z_1^3 + 1129707725184z_1^2 - 5869777630320z_1
+ 1818824190547}{2^7\cdot3^4\cdot5^4\cdot19^6} \,.
\end{aligned}
\ee
%% see pari files: 2x2nahm, nahm854rawdata
We computed the series $\Phitof_{1}(x)$ up to $\mathrm{O}(x)^{28}$. The denominator of
the series to order $27$ a priori should include all primes less
than or equal to $29$ as found in~\cite[Sec.9.2, Thm.9.1]{GZ:kashaev}. However, the 
denominator is actually $2^{77}\cdot 3^{40}\cdot 5^{34}\cdot 19^{52}$, so the series is
$\Delta$-integral. Moreover, if we take the $60$-th power of $\Phitof_{0,1}(x)$, we
find that
the denominator improves to $5^{33}\cdot 19^{54}$, where the remaining primes $5$ and
$19$ are the prime factors of the discriminant $-5^4 \cdot 19$ of the number field.

The same experiment can be performed for other values of $\nu \in \BZ^2$. Actually
one needs only~8 values since the holonomic rank of the 2-parameter $q$-holonomic
function $\nu \mapsto F_\nu(q)$ is~8. Doing so, we find the same $\Delta$-integrality
results.

Incidentally, the element corresponding to the second quartic field $E$ is not
a torsion element of $B(E)$, and its corresponding series,  computed from the
asymptotics of $F_0(\e(-1/\tau))$ when $\tau$ approaches infinity near the real line,
does not exhibit any $\Delta$-integrality properties.

\subsection{\texorpdfstring{Some $p$-adic computations}{Some p-adic computations}}
\label{sub.padic.compute}

In this section we discuss $p$-adic computations, starting with the absolute basics.
Hensel's lemma states that if a polynomial factors into irreducible polynomials over
$\BZ/p\BZ$, then this factorisation lifts to a unique factorisation over $\BZ/p^n\BZ$.
Therefore, if $\xi$ is a generator of the field $\BK$ with minimal polynomial $P(x)$
and $p$ is an unramified prime, applying Hensel's lemma to $P(x)$ lifts the Frobenius
automorphism $R/pR$ to $R_{p}^\wedge\cong R\otimes_{\BZ}\BZ_{p}$, where
$R=\calO_\BK[1/\Delta]$. Hensel's lemma is constructive and we can easily use it to
compute the Frobenius automorphism.
Firstly, notice that
\be
P(\xi^{p}) \cm 0\pmod{p}\,.
\ee
Suppose by induction that for some $d$ there is a unique $\alpha\in R/p^{d-1} R$
such that
\be
P(\xi^{p}+\alpha p) \cm 0\pmod{p^{d}}\,.
\ee
Choose a lift of $\alpha$ to $R/p^{d} R$. Let
$\beta\in R/p R$. Then we notice that
\be
\frac{P(\xi^{p}+\alpha\,p+\beta\,p^{d})-P(\xi^{p}+\alpha\,p)}{p^{d}}
\cm P'(\xi^{p})\,\beta\pmod{p}\,.
\ee
As $p$ is an unramified prime $P'(\xi^{p})\in( R/p R)^{\times}$ and hence we
can set
\be
\beta\=-\frac{P(\xi^{p}+\alpha\,p)}{p^{d}P'(\xi^{p})} \inn R/p R\,.
\ee
With this choice we find that
\be
P(\xi^{p}+\alpha p+\beta p^{d})\cm 0\pmod{p^{d+1}}\,.
\ee
Therefore, by induction, we can lift the element $\xi^{p}\in R/p R$ to an
element $\fr(\xi)\in R_{p}^\wedge$ such that $P(\fr(\xi))=0$ and
$\fr(\xi)=\xi^{p}\in R/p R$.

\medskip

Next we will explain how to evaluate the $p$-adic
polylogarithm. We can evaluate this function on an element $t$ of a $p$-adic field
using the algorithm of Besser--de Jeu described in~\cite{Lip}. Let us recall how
this is done.

\begin{itemize}
\item
  When $|t|<1$, we can use the power series definition of $\Li_n(t)$ at $t=0$.
\item
  When $t=\z \neq 1$, a root of unity of unity, we use the series expansion of
  $\Li^{(p)}_n(t/(1-t))$, which converges for $|t|< p^{\tfrac{1}{p-1}}$ for
  $t=t_0=\z/(1-\z)$ (satisfying $|t_0|=1$) to compute
  $\Li^{(p)}_n(\z)$, together with equation~\cite[Prop. 4.2]{Lip} 
\be
\Li_n(\z)
\=
(p^{ns}-1)^{-1}
\sum_{r=0}^{s-1}
p^{n(s-r)}\Li_n^{(p)}(\z_{p^{s}-1}^{p^{r}})
\ee
  to compute $\Li_n(\z)$.
\item
  When $|t-\z|<1$, for $\z$ as above, use the Taylor series expansion of $\Li_n(t)$
  at $t=\z$, together with the fact that
  $(t \tfrac{\partial}{\partial t})^s|_{t=\z} \Li_n(t)=\Li_{n-s}(\z)$ to compute
  $\Li_n(t)$. 
\item
  When $|t|>1$, use the inversion formula 
\be
\Li_n(z)+(-1)^{n}\thinspace\Li_n(z^{-1})
\=
-\frac{1}{n!}\log(z)^{n}
\ee
  to compute $\Li_n(t)$. 
\item
  Finally, one can also compute $\Li_n(t)$ when $0<|t-1|<1$ by considering the
  function $\Li_n(t)-\log(t)\Li_{n-1}(t)/(n-1)$, which is analytic in that domain.
\end{itemize}

$\;$ The next example illustrates Theorem~\ref{thm.xiz} of Section~\ref{sub.Dp}.

\begin{example}
\label{ex.xiz1}
Suppose that $\mathbb{K}=\BQ(\alpha)/(\alpha^3-\alpha^2+1)$. Let
\be
z_1\=1-\alpha^2\,,\qquad
z_2\=z_1^2-z_1+2\,,\qquad
z_3\=z_1\,.
\ee
Then $\xi=[z_1]+[z_2]+[z_3]\in B(\mathbb{K})$ represents the element of the Bloch
group of $5_2$. Consider the prime $p=5$, where
$\mathbb{K}_{5}\cong\BQ_{5^2}\times\BQ_{5}$ is a product of two unramified extension
of $\BQ_5$, one of degree two and one of degree one. We have
\be
\begin{aligned}
&D_{5}(\xi)
\=
D_5(z_1)+D_5(z_2)+D_5(z_3) \\
& \=
(3\cdot5^2 + 5^3 + 2\cdot5^4 + \cdots)\a^2 + (5^2 + 3\cdot5^3 + \cdots)\a
+ (2\cdot5^2 + 3\cdot5^3+\cdots)\,.
\end{aligned}
\ee
The roots of unity $\mu(\mathbb{K}_{5})$ is a product of finite cyclic groups of order
$5^2-1=24$ and $5-1=4$. There is an order $24$ subgroup generated by $\z_{24}$, where
\be
\z_{24} \= \lim_{s\to\infty} \alpha^{5^{2s}} \=
(4\cdot5^2 + \cdots)\a^2 + (1 + \cdots)\a + (3\cdot5 + \cdots)\,.
\ee
Then we have
\be
\begin{small}
\begin{aligned}
D_5(\z_{24})
&\=
(4\!\cdot\! 5^2 + 4\!\cdot\! 5^3 + \cdots)\a^2 + (4\!\cdot\! 5^2 + 2\!\cdot\! 5^3
+ \cdots)\a
+ (2\!\cdot\! 5^2 + 5^3 + 2\!\cdot\! 5^4 + \cdots),\\
D_5(\z_{24}^5)
&\=
(2\!\cdot\! 5^2 + 2\!\cdot\! 5^3 + 5^4 + \cdots)\a^2
+ (3\!\cdot\! 5^2 + 3\!\cdot\! 5^3 + 5^4 + \cdots)\a + (5^2 + 3\!\cdot\! 5^4
+ \cdots),\\
D_5(\z_{24}^6)
&\=
(3\!\cdot\! 5^2 + 3\!\cdot\! 5^3 + 2\!\cdot\! 5^4 + \cdots)\a^2\!\!
+ (3\!\cdot\! 5^2 + 2\!\cdot\! 5^3
+ 3\!\cdot\! 5^4 + \cdots)\a\! + (2\!\cdot\! 5^2 + 4\!\cdot\! 5^4 + \cdots).
\end{aligned}
\end{small}
\ee
This implies that
\be
D_5(\xi)
\=
(1 + 4\cdot 5 + 3\cdot 5^2 + \cdots)D_5(\z_{24}) + (3 + 5 + \cdots)D_5(\z_{24}^2)
+ (1 + 5 + 4\cdot 5^2 + \cdots)D_5(\z_{24}^6)
\ee
which illustrates Theorem~\ref{thm.xiz} of Section~\ref{sub.Dp}. This implies that
\be
\xi \= (1 + 4\cdot 5 + 3\cdot 5^2 + \cdots)[\z_{24}] + (3 + 5 + \cdots)[\z_{24}^2]
+ (1 + 5 + 4\cdot 5^2 + \cdots)[\z_{24}^6] \inn B(\BK_5)\otimes\BQ_5 \,.
\ee
\end{example}

%% Investigate what happens to these theorems when $p=23$ is a bad prime.
%% Then, the extensions of $\BQ_p$ are ramified, and the pari code needs to
%% be modified.

\subsection{\texorpdfstring{The $4_1$ knot}{The 4\_1 knot}}

%% see file: README.pari.computations which summarises the numerical computations

Every result presented in the paper has been numerically verified. We will describe
a select few of the computations that were carried out for two examples associated
to the knots for $4_1$ and $5_2$.

\medskip

Firstly, there are a variety of methods available, both numerical and exact, to compute
these collections of power series at roots of unity. Using the exact formulas of
formal Gaussian integration, we can easily compute the series around $q=1$.
For example in~\cite[Sec. 7]{GSW} the computation was detailed for $4_1$.
With $\psi$ as in equation~\eqref{Psikdef}, this gives the series
\be
\begin{small}
\begin{aligned}
\Phitof^{4_1}_1(x) \= (1+x)^{\frac{1}{8}}
\langle
\psi_{0,\z_6,1}(w_1,x)\psi_{0,\z_6,1}(w_2,x)\rangle_{
\Lambda_{4_1}}
=\;
(1+x)^{\frac{1}{6}}
\Big\langle
\exp\Big(\frac{w}{2}\Big)\psi_{0,\z_6,1}(w,x)^2\Big\rangle_{
2\z_6-1}\,,
\end{aligned}
\end{small}
\ee
where
\be
\Lambda_{4_1}
\=
\begin{pmatrix}
\z_6-1 & -1\\
-1 & \z_6-1
\end{pmatrix}\,.
\ee
This simple expression allows the computation of this series to high precision. We
have computed of order $600$ in this example. The first few terms are
given in the Introduction~\ref{sub.intexp} in equation~\eqref{as41} for $q=e^h=1+x$,
with $\Phi^{4_1}_1(h)=f_1^{4_1}(x)$. The integrality of
the symmetrisation was then given in equation~\eqref{Phi41sym}.
The Ohtsuki property of the series was checked with this data.
This field is $\BQ(\sqrt{-3})$ and has non-trivial roots of unity. In fact, the symbols
of these roots of unity generate the Bloch group.
In this case we can construct a generator of the
Habiro module globally using the infinite Pochhammer symbol with a root of unity as
argument. Indeed, we find
\be
\begin{aligned}
&\sqrt[4]{-3}\,\widehat{\Phitof}^{4_1}_1(x)\,(1+x)^{-\frac{1}{24}}
\,(\z_6^2;1+x)_{\infty}^{3}
(1-\z_6)^{-\frac{3}{2}}\\
&\=
1 + \frac{1}{3^3}(5\z_6 - 7)x + \frac{1}{3^5}(-30\z_6 + 31)x^2
+ \frac{1}{3^9}(1565\z_6 - 1444)x^3+\cdots\,.
\end{aligned}
\ee
This whole series is integral away from $3$. This observation was a key clue to
finding a definition of the Habiro modules.

\subsection{\texorpdfstring{The pair of the $5_2$ and the
    $(-2,3,7)$-pretzel knots}{The pair of the 5\_2 and the (-2,3,7)-pretzel knots}}
\label{sub.52-237}

In this section we give a pair of elements of the same module over the Habiro
ring of the cubic field $\BK$ of discriminant $-23$ that come from the asymptotic
series associated to a pair of hyperbolic knots, namely the $5_2$ and the
$(-2,3,7)$-pretzel knot. This pair of knots was studied in detail in~\cite{GZ:kashaev}
and was instrumental in formulating the results and conjectures of that paper. Recall
that this pair of knots has common trace field $\BK$ generated by a solution to
the cubic equation $\xi^3-\xi^2+1=0$. In fact, these knots are scissors congruent;
their complements can be decomposed in three ideal tetrahedra with shapes in $\BK$,
but assembled differently for each knot. It follows that the corresponding elements of
the Bloch group are equal, modulo 6-torsion. Although several terms of their series
$\Phitof^{(5_2)}_m(x)$ and $\Phitof^{(-2,3,7)}_m(x)$ were computed at various roots of
unity of small order $m$, no relation between the two series was found.

There are several methods to compute the series $\Phitof^{(5_2)}_m(x)$ and
$\Phitof^{(-2,3,7)}_m(x)$: 
\begin{itemize}
\item
  numerically compute the Kashaev invariant, its asymptotic expansion
  at roots of unity of order $m$ as in~\cite[Eqn.(1.5)]{GZ:kashaev} using
  high precision and extrapolation, and then lifting to elements of $\BK$,
\item
  numerically compute the asymptotics of $q$-series associated to these knots
  using high precision and extrapolation, and then lifting to elements of $\BK$,
  as in~\cite{GZ:qseries}
\item
  compute using exact arithmetic the formal Gaussian integral associated to these
  knots.
\end{itemize}
The state-integrals of these knots a priori are 3 and 4-dimensional, but they
reduce to explicit 1-dimensional state integrals for both (see~\cite[Eqn.(39)]{AK}
for the $5_2$ knot and~\cite[Eqn.(58)]{GK:evaluation} for the $(-2,3,7)$-pretzel knot).
When $m=1$, the third method applied to these 1-dimensional integrals gives an
efficient way to compute the series $\Phitof^{(5_2)}_1(x) + \mathrm{O}(x)^{401}$ and
$\Phitof^{(-2,3,7)}_1(x) + \mathrm{O}(x)^{401}$; see~\cite[Sec.4]{AGL}. Although the
coefficients of both series have a universal denominator defined
in~\cite[Thm.9.1]{GZ:kashaev} the product (keeping in mind that
$q=1+x$, $q^{-1}=1 -x/(x+1)$)
\be
\begin{aligned}
\Phitof^{(5_2)}_1(x)\,\Phitof^{(-2,3,7)}_1\Big(-\frac{x}{1+x}\Big) \= &
c \Big( 1 + \frac{-7 \a^2 + 20 \a + 33}{2^4 \cdot 23} x \\
& + \frac{226541 \a^2 - 275879 \a - 218336}{2^9 \cdot 23^3} x^2 \\
& + \frac{-95096039 \a^2 + 85905420 \a + 49207882}{2^{13} \cdot 23^4} x^3 
+ \dots \Big)
\end{aligned}
\ee
has denominators given by powers of 2 and 23, e.g.
the denominator of $x^{400}$ is $2^{1997} \cdot  23^{581}$. Here $\a$ satisfies
$\a^3-\a^2+1=0$. The constant term is the product of the square roots of
the $\delta$-invariant of the two knots
\be
c \= \frac{1}{\sqrt{-6 \a^2 + 10 \a -4}} \cdot
\frac{1}{\sqrt{-24 \a^2 + 32 \a -26}} \= \frac{1}{\sqrt{2} \cdot (2 \a^2 -2\a+3)}
\ee

%% see pari file: Exact-Check-52-237-sage

Despite the above similarity, the series $\Phitof^{(5_2)}_1(x)$ and
$\Phitof^{(-2,3,7)}_1(x)$
are quite different from each other: for instance the coefficients of the
series of $5_2$ ``see'' only the cubic trace field of discriminant $-23$, whereas
those of $(-2,3,7)$ see in addition the abelian field $\BQ(2\cos(2\pi/7))$ of
discriminant $49$. Thus, even the rank of the étale algebras is different.

\medskip

We can explore $5_2$ in some more detail. A Gaussian integral in this case can be
given
\be
\begin{aligned}
\Phitof^{5_2}_1(x)
\=
\frac{1}{\sqrt{\delta_{5_2}}}
\Big\langle
e^w\psi_{0,z,1}(w,x)^3\Big\rangle_{3\alpha - 2}\,,
\end{aligned}
\ee
where $\alpha^3-\alpha^2+1=0$ and $\delta=3\alpha-2$. The first few terms of this
series are well documented\footnote{Here we again use $\Phi(h)$ instead of
  $\Phitof(x)$ where $q=e^h=1+x$. Also, there are factors of $q=e^h$ and $\z_8$
  that differ here from the formulas in~\cite[Equ. 4]{GZ:kashaev}.} but we will
give them again here:
\be
\begin{small}
\begin{aligned}
\Phi^{5_2}_1(h)
&\=
\frac{1}{\sqrt{3\alpha-2}}
\Big(1+
\frac{765\alpha^2 - 1086\alpha + 1043}{24\,(3\alpha-2)^3}h
+\frac{1757583\alpha^2 - 2956029\alpha + 2241964}{1152\,(3\alpha-2)^6}h^2\\
&\qquad\qquad\qquad\qquad
+\frac{21285784611\alpha^2 - 37166037066\alpha + 27969826252}{414720\,(3\alpha-2)^9}h^3
+\cdots\Big)
\end{aligned}
\end{small}
\ee
Since $q=1+x$ and $q^{-1}=1-x/(1+x)$, it follows that the symmetrisation is given by
\be\label{eq:52sym}
\begin{aligned}
\delta\Phitof^{5_2}_1(x)\,\Phitof^{5_2}_1\Big(-\frac{x}{1+x}\Big)
&\=
1 + \frac{1}{23^3}
(465\a^2 - 465\a + 54)x^2
+\frac{1}{23^3}
(-465\a^2 + 465\a - 54)x^3\\
&\qquad\qquad+\frac{1}{23^6}
(4934541\a^2 - 4934541\a + 462834)x^4
+\cdots\,.
\end{aligned}
\ee
This data given is enough to compute the first digit of the constant of the
symmetrised series $5$-adically at $\z_5$. Indeed, we find that
\be
\Phitof^{5_2}_1(\z_5-1)\,\Phitof^{5_2}_1(\z_{5}^{-1}-1)
\cm
(\a^2 + 3\a + 2)\z_5^3 + (\a^2 + 3\a + 2)\z_5^2
\pmod{5}\,.
\ee
while
\be
\Phitof^{5_2}_5(0)^2
\cm
(\a^2 + 4\a)x^3 + (\a^2 + 4\a)x^2
\pmod{5}\,.
\ee
This agrees with the gluing in the Habiro ring as
\be
(\a^2 + 4\a)^5
\cm
\a^2 + 3\a + 2
\pmod{5}\,.
\ee
These computations can be carried out to any desired order and for any roots of
unity by replacing the $p$-power map with the Frobenius automorphism described in
Section~\ref{sub.padic.compute}.

\medskip

The final experiment to describe here is the integrality of the series
$\Phitof^{5_2}_1(x)$ after division by the generator $\Psi_{\xi,p}(x)$.
We can use Example~\ref{ex.xiz1} of Section~\ref{sub.padic.compute} to give the
description of our series. For some constant $C$ we find
\be
\begin{small}
\begin{aligned}
&C\,\Phitof^{5_2}_5(x)
\Psi_{[\z_{24}],5}(x)^{-(1 + 4\cdot 5 + 3\cdot 5^2 + \cdots)}
\Psi_{[\z_{24}^2],5}(x)^{-(3 + 5 + \cdots)}
\Psi_{[\z_{24}^6],5}(x)^{-(1 + 5 + 4\cdot 5^2 + \cdots)}\\
&\=1 + ((2 + 4\!\cdot\!5 + 4\!\cdot\!5^3 + \cdots)\a^2 + (2 + 3\!\cdot\!5
+ 3\!\cdot\!5^2
+ 3\!\cdot\!5^3 + \cdots)\a + (3 + \cdots))x\\
&+ ((3 + 2\!\cdot\!5 + 2\!\cdot\!5^2 + 3\!\cdot\!5^3 + \cdots)\a^2 + (3 + 4\!\cdot\!5^2
+ 3\!\cdot\!5^3 + \cdots)\a + (3 + 3\!\cdot\!5 + 2\!\cdot\!5^2 + 3\!\cdot\!5^3
+ \cdots))x^2\\
&+ ((2 + 4\!\cdot\!5 + 5^2 + 4\!\cdot\!5^3 + \cdots)\a^2 + (3 + 4\!\cdot\!5^2
+ 5^3 + \cdots)\a + (2 + 3\!\cdot\!5^2 + 3\!\cdot\!5^3 + \cdots))x^3
+\cdots\,.
\end{aligned}
\end{small}
\ee
The coefficient of $x^{200}$ is given by
\be
(1 + 5^2 + 5^3 + \cdots)\xi^2 + (3 + 2\cdot5 + 2\cdot5^2 + 4\cdot5^3 + \cdots)\xi
+ (4 + 3\cdot5 + 2\cdot5^2 + 5^3 + \cdots),
\ee
which is clearly 5-integral.

\medskip

The perturbative series associated to the three boundary-parabolic representations
of the $(-2,3,7)$-pretzel knot with values in the abelian number field
$\BK=\BQ(2\cos(2\pi/7))$ give an element of $\calH_{\calO_\BK[1/7]}$. The corresponding
element of the Bloch group is 3-torsion. We have computed 400 terms of the series
at $m=1$ and have checked their $7$-integrality. For instance, the coefficient of
$x^{400}$
of this series is $2^{1997} \cdot 3^{596} \cdot 7^{466}$, which improves to
$2^{1997} \cdot 7^{466}$ upon taking the third power of the series, illustrating
part (b) of Corollary~\ref{cor.3} of Section~\ref{sub.results}.

%% see pari file: Exact-Check-52-237-sage

\subsection{A note on modularity}
\label{sub.mod}

Modular forms at roots of unity give rise to special elements associated to torsion
classes in $K_3$. We will consider two simple examples: one coming from a quadratic
Gauss sum (and hence related to the asymptotics of Jacobi $\th$-series) and the other
related to the Roger-Ramanujan functions, which are well known to be modular.

\begin{example}[A quadratic Gauss sum]
Consider
\be
\begin{aligned}
F_{m}(x)&\=\frac{1}{m}\bigg(\sum_{k\in\BZ/m\BZ}\z_m^{k^2}\bigg)
\bigg(\sum_{k\in\BZ/m\BZ}\z_m^{-k^2}\bigg)
\= \frac{1}{2}(1+\z_4^{m})(1+\z_4^{-m})\\
&\=
\left\{
\begin{array}{cl}
1 & \text{if }m\equiv 1\pmod{4}\,,\\
0 & \text{if }m\equiv 2\pmod{4}\,,\\
1 & \text{if }m\equiv 3\pmod{4}\,,\\
2 & \text{if }m\equiv 4\pmod{4}\,.
\end{array}
\right.
\end{aligned}
\ee
We see that this gives an almost trivial element of the Habiro ring
$\calH_{\BZ[1/2]}$. We can also consider
\be
G_{m}(x)\=\frac{\z_4}{m}\bigg(\sum_{k\in\BZ/m\BZ}\z_m^{k^2}\bigg)^2
\= \frac{1}{2}(1+\z_4^{m})^2
\=
\left\{
\begin{array}{cl}
i & \text{if }m\equiv 1\pmod{4}\,,\\
0 & \text{if }m\equiv 2\pmod{4}\,,\\
-i & \text{if }m\equiv 3\pmod{4}\,,\\
2 & \text{if }m\equiv 4\pmod{4}\,,
\end{array}
\right.
\ee
giving a slightly less trivial element of $\calH_{\BZ[i,1/2]}$.
\end{example}

\begin{example}[Rogers-Ramanujan symmetrised]
The Rogers-Ramanujan function has an associated field $\BQ(\sqrt{5})$, which has an
abelian Galois group over $\BQ$. Consider
\be
J(z,q) \= \sum_{k=0}^{\infty}\frac{q^{k(k+1)}z^{2k}}{(qz;q)_{k}}
\inn \calH_{\BQ(z)}\,.
\ee
This $J$ has the special property that
\be
J(z,\z_m(1-u)) \inn \BZ[z,(1-z^m-z^{2m})^{-1}][\z_m]\llbracket u\rrbracket\,.
\ee
We define for $\xi^{2}+\xi-1=0$
\be
F_{m}(u) \= \sum_{\th^{m}=\xi} \Res{z=\th} J(z,\z_m(1-u))\frac{dz}{z}\,.
\ee
Each $F_m(u)$ in $\BQ(\sqrt{5},\z_m)$ is constant, the first five values being:
\be
\begin{small}
\begin{aligned}
F_{1}(u)
&\=-\frac{1}{2}-\frac{\sqrt{5}}{10}\,,\qquad\qquad\qquad\qquad
F_{2}(u)
&\=-\frac{1}{2}+\frac{\sqrt{5}}{10}\,,\\
F_{3}(u)
&\=-\frac{1}{2}+\frac{\sqrt{5}}{10}\,,\qquad\qquad\qquad\qquad
F_{4}(u)
&\=-\frac{1}{2}-\frac{\sqrt{5}}{10}\,,\\
F_{5}(u)
&\=\frac{\sqrt{5}}{5}\z_5^3 + \frac{\sqrt{5}}{5}\z_5^2 -\frac{1}{2}
+\frac{\sqrt{5}}{10}\,.
\end{aligned}
\end{small}
\ee
Note that, depending on the embeddings, $F_5(u)$ is equal to either $0$ or $-1$ in
$\BC$. In fact, using the modularity of the Rogers-Ramanujan functions, one can show
that
\be
-5F_{m}^{2}(u)-5F_{m}(u)
\=
\bigg\{
\begin{array}{cl}
1 & \text{if }(m,5)\=1\,,\\
0 & \text{otherwise}\,,
\end{array}
\ee
the expressions on either side being the value at $\z_m(1-u)$ of an element of
$\calH_{\BZ[1/5]}$. More generally, $\calH_{\BZ[1/5]}$ contains an element $\chi_N$
sending $\z_m$ to $\delta_{v_{5}(m),N}$.  These elements can be constructed using the
operation $f(q)\mapsto f(q^5)$ --- which does act on
the usual Habiro ring of~$\BZ$. Note
that the operation $f(q)\mapsto f(q^k)$ is in general only defined
on the Habiro
ring of a ring $R$ if it contains $1/k$. For example for $k=2$ we see that
$f((1-u)^2)=f((-(1-u))^2)$ and hence the re-expansion will not require a Frobenius
$\varphi_2$, which it should.)

This example illustrates how the different roots of unity can be disconnected
by denominators. This viewpoint will be pushed further in the next section; where
it will be used to give an alternative approach to the Habiro ring of a number field.
\end{example}

%%%%%%%%%%%%%%%%%%%%%%%%%%%%%%%%%%%%%%%%%%%%%%%%%%%%%%%%%%%%%%%%%%%%%%%%%%%% 
%%%%%%%%%%%%%%%%%%%%%%%%%%%%%%%%%%%%%%%%%%%%%%%%%%%%%%%%%%%%%%%%%%%%%%%%%%%%

\section{Alternative approach: Habiro rings via congruences}
\label{sec.hab}

In this final section we study the basic properties of the original Habiro ring
$\calH$ and its generalisation to the Habiro $\calH_{R}$ of a ring $R$.
The most important point is how to identify $\calH$ explicitly within the
ring of ``Galois invariant functions near roots of unity'' by means of a
sequence of congruences. The Habiro ring $\calH_{R}$ when $R$ is the ring of integers
(or $S$-integers) of a
number field is then given by the \emph{same} collection of congruences,
but twisted by the application of a certain Frobenius automorphism at each root of
unity. We will consider filtrations with finite quotients of both the ``naive''
Habiro ring
and the space of functions near roots of unity. On these finite quotients we will
explicitly
describe the embedding of the Habiro ring. This will be a simple map between lattices
and we will give an explicit formula for the index, which will illustrate the
strength of the constraints for a function near roots of unity to be in a Habiro ring.
We will illustrate the procedure on several examples. Finally, we will explore some
exotic functions near roots of unity that exhibit some Habiro-like properties.

Throughout this section, the ground rings $R$ we will consider will be associative,
commutative, unital, and torsion-free rings $R$. The original Habiro ring will be
when $R=\BZ$ but more generally $R$ can be thought of as the ring of integers of a
number field with a finite set of primes inverted.

\subsection{Functions near roots of unity and the naive Habiro ring}
\label{sec:naive}

In Habiro's original work~\cite{Habiro:completion} the Habiro rings of $\BZ$ and
$\BQ$ were defined as $\HH:=\HH_\BZ$ and $\HH_{\BQ}$, where $\HH_R$ is defined for
any ground ring $R$ by
\be
\label{eq:Hnaive}
\HH_{R} \= \varprojlim_n R[q]/(q;q)_{n}R[q]\,.
\ee
(Here we use the notation $\HH_R$, defined as an inverse limit, as opposed to
$\calH_R$, which is defined as a subset of $\calP_R$ in Definition~\ref{def.hab} of
Section~\ref{sub.hab}.)
(The two rings $\HH=\HH_\BZ$ and $\calH=\calH_\BZ$ can be identified canonically,
but in general $\HH_R$, which we will call the \emph{naive Habiro ring}, and
$\calH_R$ are different.)
Note that $\HH_\BQ\neq\HH_{\BZ}\otimes\BQ$. Indeed, the ring
$\HH_{\BZ}$ is an integral domain while $\HH_{\BQ}$ has a large number of zero
divisors~\cite[Sec. 7.5]{Habiro:completion}.

The definition of $\HH_R$ leads immediately to
expressions of the form~\eqref{eq:elts.hab} with $P_n(q)\in R[q]$.
These expressions are not unique but can be made unique by assuming
$\deg P_{n}<n$ so that $P_{n}(q)=\sum_{k=0}^{n-1}a_{n,k}q^k$ for some
$a_{n,k}\in R$.
(The example from equation~\eqref{eq:KZ} has $a_{n,k}$ given simply by $\delta_{k,0}$.)
The coefficients $a_{n,k}$ give an explicit isomorphism of $R$-modules
\be
\label{eq:hdecomp}
\begin{array}{ccccc}
  {}_{\phantom{2_{2_{2_{2_{2_{2}}}}}}}\HH_R\!
  &=&
  \!\prod_{n=1}^{\infty}
  (R\oplus qR\oplus\cdots\oplus q^{n-1}R)(q;q)_{n-1}
  % \;\cong\;
  % \prod_{n=1}^{\infty}R[q]/(1-q^n)R[q]
  &\cong&
  \prod_{n>k\geq0}R\,,\\
  \sum_{n=1}^{\infty}P_{n}(q)(q;q)_{n-1}\!
  &\mapsto&
  \!\sum_{n=1}^{\infty}\sum_{k=0}^{n-1}a_{n,k}q^k(q;q)_{n-1}
  &\mapsto&
  \{a_{n,k}\}_{n>k\geq0}\,.
\end{array}
\ee

As in Section~\ref{sub.habZ}, we have a map $\iota:\HH_R\to\calP_R$, where
$\calP_R$ was defined in equation~\eqref{PR}:
\be
\begin{array}{ccccc}
\calP_{R}
&\;:=\;&
\Big(\prod_{\z\in\mu_\infty} R[\z]
\llbracket u\rrbracket\Big)^{\mathrm{Gal}(\BQ^{\mathrm{ab}}/\BQ)}
&\;\cong\;&
\prod_{m\geq 1} R[\z_m]\llbracket u\rrbracket\,,\\
% \cap & & \cap & & \cap\\
% \calP_{\BQ}
% &\=&
% \Big(\prod_{\z\in\mu_\infty} \BQ[\z]
% \llbracket u\rrbracket\Big)^{\mathrm{Gal}(\BQ^{\mathrm{ab}}/\BQ)}
% &\;\cong\;&
% \prod_{m\geq 1} \BQ[\z_m]\llbracket u\rrbracket\,,\\
& & & &\\
f(q) & \mapsto & (f_\z(u)=F(\z(1-u)))_\z & \;\mapsto\; & (f_m(u)=f_{\z_m}(u))_m\,.
\end{array}
\ee
Here, for each $m$, we use the local variable $u$ defined by $q=\z_m(1-u)$.
(In the Introduction, we used
$q=\z_m+x$ instead, so $u$ and $x$ are related by $x=-\z_mu$
and $f_m$ differs here from that of equation~\ref{iota} by the change of variable.)
The map $\iota:\HH_R\to\calP_R$
splits as a product of maps $\iota=\prod_{m\geq1}\iota_m$, where
for $F\in\HH_R$, we define the map~$\iota_m$ by the expansion at $q=\z_m(1-u)$, i.e.
$\iota_m(F)=F(\z_m(1-u))\in R[\z_m][\![u]\!]$.

Similar to the coordinates~$a_{n,k}$ coming from the isomorphism~\eqref{eq:hdecomp},
the $R$-module $\calP_R$ has a natural basis over $R$ given by
$\{\gamma_{m,\ell,j}\}_{m,\,\ell\geq1,\,\phi(m)>j\geq0}$, where
\be
\label{eq:oht.c}
f_{m}(u) \= \sum_{\ell=1}^{\infty} C_{\ell}(\z_m)\,u^{\ell-1}\,,
\qquad C_{\ell}(\z_m)\=
  \sum_{j=0}^{\phi(m)-1}\gamma_{m,\ell,j}\,\z_m^{j}\inn \BZ[\z_m]\,.
\ee
(Again, the shift by $1$ --- making $C_{\ell}(\z_m)$ the coefficient of $u^{\ell-1}$
rather than $u^{\ell}$ --- makes later formulas more natural, because its properties
depend on the multiplicative nature of $\ell$ as opposed to $\ell-1$.)
We then have explicit isomorphisms of abelian groups,
\be
\label{eq:gamma}
\begin{array}{ccccc}
\calP_{R}
&\;\cong\;&
\prod_{n>0}\prod_{m|n}
R[\z_m]
&\;\cong\;&
\prod_{n>0}
\prod_{\underset{\phi(m)>j\geq0}{m\ell=n}}
R\,,\\
\{f_{m}(u)\}_{m>0}
&\;\mapsto\;&
\{C_{\ell}(\z_m)\}_{m,\ell>0}
&\;\mapsto\;&
\{\gamma_{m,\ell,j}\}_{\underset{\phi(m)>j\geq0}{m,\ell>0}}\,.
\end{array}
\ee

All of these $R$-modules come with natural
filtrations
and these are compatible.
From equation~\eqref{eq:hdecomp} we can naturally consider
\be
\begin{aligned}
\HH_{R,N}
&\;:=\; (q;q)_{N-1}\HH_R
\=
\bigg\{\sum_{n=N}^\infty\sum_{k=0}^{n-1}a_{n,k}q^k (q;q)_{n-1}\;
\Big|\;a_{n,k}\inn R\bigg\}\\
&\=\prod_{n=N}^{\infty}
(R\oplus qR\oplus\cdots\oplus q^{n-1}R)(q;q)_{n-1}
% \;\cong\;
% \prod_{n=1}^{\infty}\BZ[q]/(1-q^n)\BZ[q]
\;\cong\;\prod_{n\geq N,\,n>k\geq0}R\,,
\end{aligned}
\ee
i.e. $\HH_{R,N}$ consists of sums $\sum_{n=N}^{\infty}P_{n}(q)(q;q)_{n-1}$, which in
coordinates is described by a collection $\{a_{n,k}\}_{n>k\geq0}$ with $a_{n,k}=0$
for $n<N$. Denote the quotient $\HH_R/\HH_{R,N}$ by $\HH_R^{N}$.
We thus have an increasing sequence $\HH_{R,N}$ of submodules of $\HH_R$ and a
decreasing sequence $\HH_R^N$ of quotients,~i.e.
\be
\label{eq:seq.app}
0\;\subset\;\cdots\;\subset\;\HH_{R,2}\;\subset\;\HH_{R,1}\=\HH_R\;
\twoheadrightarrow\;\cdots
\;\twoheadrightarrow\;\HH_R^{2}
\;\twoheadrightarrow\;\HH_R^{1}
\=0\,.
\ee
Note that as $R$-modules,
\be\label{eq:filt.quot.HH}
\HH_R^{N}
\,\;\cong\;\,
R[q]/(q;q)_{N-1}R[q]
\,\;\cong\;\,
\bigoplus_{n=1}^{N-1}
(R\oplus qR\oplus\cdots\oplus q^{n-1}R)(q;q)_{n-1}\,,
\ee
where the first isomorphism is induced by the canonical inclusion
$R[q]\hookrightarrow\HH_R$.
It follows that $\HH_R = \varprojlim_{N} \HH_R^{N}$ (or equivalently
$\bigcap_N\HH_{R,N}=\{0\}$). Note also that the exact sequence of $R$-modules
$0\to\HH_{R,N}\to\HH_R\to\HH_R^{N}\to0$ splits so that
$\HH_R\cong\HH_{R,N}\oplus \HH_R^{N}$ as $R$-modules for every $N$.

We now look at how expansions near roots of unity interact with
the filtration $\{\HH_{R,N}\}_{N>0}$ on $\HH_R$.
If we note that $(q;q)_{N-1}$ vanishes at a primitive $m$-th root of unity $\z_m$
to order $\lceil N/m\rceil-1$, then we see that $\iota$ maps this filtration
to the filtration $\{\calP_{R,N}\}_{N>0}$ on $\calP_R$ given by
\be
\calP_{R,N} \;:=\; \big\{f\inn\calP_R\;|\;
f_m(u)\=O(u^{\lceil N/m\rceil-1})\big\}
\,\;\cong\;\,
\prod_{n=N}^{\infty}\prod_{m|n}
R[\z_m]
\,\;\cong\;\,
\prod_{n=N}^{\infty}
\prod_{\underset{\phi(m)>j\geq0}{m\ell=n}}
R\,,
\ee
where the isomorphisms are of $R$-modules.
In coordinates, $\calP_{R,N}$ consists of $\{C_{\ell}(\z_m)\}_{m,\ell>0}$ in
$\calP_R$ satisfying $C_{\ell}(\z_m)=0$ for $m\ell<N$. As before,
we denote the quotient $\calP/\calP_{R,N}$ by~$\calP_R^{N}$. We
thus have an increasing sequence~$\calP_{R,N}$ of submodules of~$\calP_R$ and a
decreasing sequence~$\calP_R^{N}$ of quotients, just as in equation~\eqref{eq:seq.app}.
Moreover, we have a splitting of $R$-modules so
that~$\calP_R\cong\calP_{R,N}\oplus \calP_R^{N}$.
Note also that we have the canonical isomorphism
\be
\calP_R^{N}
\,\;\cong\;\,
\bigoplus_{m=1}^{N}
\big(R[\z_m][u]\+\mathrm{O}(u^{\lceil N/m\rceil-1})\big)\,,
\ee
which implies that $\calP_R = \varprojlim_{N} \calP_R^{N}$ (or equivalently
$\bigcap_N\calP_{R,N}=\{0\}$), just as for $\HH$.
We see that $\iota(\HH_{R,N})\subseteq\calP_{R,N}$ and therefore
we have well-defined maps
\be
\iota^{N}:
\HH_R^{N}\to\calP_R^{N}\,,
\ee
each factoring as~$\prod_{m\geq1}\iota_m^N$.
This is an injective map between free $R$-modules of the same rank (equal
$N(N-1)/2$), invertible after tensoring with $\BQ$. This implies in turn that
$\iota$ is an injection over $R$ and an isomorphism over $R\otimes \BQ$, e.g.,
\be
\begin{array}{ccc}
\HH_{\BZ}& \hookrightarrow
& \calP_{\BZ}\\
\cap & & \cap\\
\HH_{\BQ} & \cong & \calP_{\BQ}
\end{array}
\ee
for the original Habiro rings.

To prove the last statements, we consider the associated graded groups,
\be
\begin{aligned}
\calP_{R,N}/\calP_{R,N+1}
&\;\cong\;
\bigoplus_{m\ell=N}
\big(u^{\ell-1}R[\z_m]\llbracket u\rrbracket\+\mathrm{O}(u^{\ell})\big)
\;\cong\;
\bigoplus_{m\ell=N}R[\z_m]\,,\\
\HH_{R,N}/\HH_{R,N+1}
&\;\cong\;
R[q]/(1-q^N)R[q]\,,
\end{aligned}
\ee
where the last isomorphism is induced from the first isomorphism of
equation~\eqref{eq:filt.quot.HH}. Notice that
\be
\mathrm{rk}(\calP_{R,N}/\calP_{R,N+1})
\=
\sum_{m|N} \mathrm{rk}(R[z_m])
\=
\sum_{m|N} \phi(m)
\= N \= \mathrm{rk}(\HH_{R,N}/\HH_{R,N+1})\,,
\ee
which implies the values given in the above formulas for the ranks of $\HH_{R}^N$.
The map $\iota$ descends to these graded quotients. For $P(q)\in R[q]/(1-q^{N})$,
we multiply by $(q;q)_{N-1}$ and expand at $\z_m$ for $m|N$.
For
$q=\z_m(1-u)$
\be
\label{eq:poch.dml}
(q;q)_{m\ell-1}
\=
D_{m,\ell}\,u^{\ell-1}\+ O(u^\ell),
\qquad D_{m,\ell} \;:=\;m^{2\ell-1}(\ell-1)!\,,
\ee
by a simple calculation, which is left to the reader.
Therefore, $\iota$ is represented on these quotients by the map
\be
P(q)
\;\mapsto\,
\bigoplus_{m\ell=N}\;
\big(
D_{m,\ell}\,u^{\ell-1}\,
P(\z_m)\+\mathrm{O}(u^{\ell})\big)\,.
\ee
% where we use equation~\eqref{eq:poch.dml}.
Since $D_{m,\ell}\in\BZ_{>0}$ is a non-zero integer, it follows that
$\iota$ is an injection
on these quotients and an isomorphism after tensoring with $\BQ$.
(Indeed, here we have used that $P(q)\in(1-q^N)\BZ[q]$ if and only if $P(\z_m)=0$
for all $m|N$.)
Given that $\iota$ respects the grading and is injective on the associated graded
pieces, it is injective.
This fact implies that understanding $\HH_R$ amounts to understanding the image of
$\iota$, which was the philosophy used in Definition~\ref{def.hab} of
Section~\ref{sub.hab}.

\medskip

We now restate the above considerations in terms of
the $R$-bases $\{q^k (q;q)_{n-1}\}_{k\leq 0,\,n>0}$ and $\{\z_m^{j}u^{\ell-1}\}$ and
the corresponding coordinates $\{a_{n,k}\}$ and $\{\gamma_{m,\ell,j}\}$ for~$\HH_R$
and~$\calP_R$ as defined in equations \eqref{eq:hdecomp}~and~\eqref{eq:gamma},
respectively. These bases give rise to a square matrix~
$\bM_{N}$ representing the map~$\iota^{N}$.
This matrix has integers entries and is independent of the ring~$R$.
(This is equivalent to noticing that the choice of basis and coordinates imply
isomorphisms $\HH_R\cong\HH_{\BZ}\otimes_{\BZ} R$ and
$\calP_R\cong\calP_{\BZ}\otimes_{\BZ} R$, with $\iota_R=\iota_\BZ\otimes_\BZ 1$,
with $\bM_N$ represents $\iota_\BZ^N$.)
Also, since~$\iota^{N}$ respects the filtrations, the matrix~$\bM_{N}$ is in block
lower triangular form, as illustrated for $N=5$ by the following equation.
\be
\begin{small}
\begin{aligned}
\left(\begin{array}{c}
\gamma_{1,1,0}\\
\hline
\gamma_{2,1,0}\\
\gamma_{1,2,0}\\
\hline
\gamma_{3,1,0}\\
\gamma_{3,1,1}\\
\gamma_{1,3,0}\\
\hline
\gamma_{4,1,0}\\
\gamma_{4,1,1}\\
\gamma_{2,2,0}\\
\gamma_{1,4,0}
\end{array}
\right)
\=
\left(\begin{array}{c|cc|ccc|cccc}
1 &   &   &   &   &   &   &   &   &  \\
\hline
1 & 2 & -2 &   &   &   &   &   &   &  \\
0 & 1 & 1 &   &   &   &   &   &   &  \\
\hline
1 & 1 & 1 & 3 & 0 & -3 &   &   &   &  \\
0 & -1 & 2 & 0 & 3 & -3 &   &   &   &  \\
0 & 0 & -1 & 2 & 2 & 2 &   &   &   &  \\
\hline
1 & 1 & 1 & 2 & 2 & -2 & 4 & 0 & -4 & 0\\
0 & -1 & 1 & -2 & 2 & 2 & 0 & 4 & 0 & -4\\
0 & -1 & 3 & 4 & -4 & 4 & 8 & -8 & 8 & -8\\
0 & 0 & 0 & -1 & -3 & -5 & 6 & 6 & 6 & 6
\end{array}
\right)
\left(\begin{array}{c}
a_{1,0}\\
\hline
a_{2,0}\\
a_{2,1}\\
\hline
a_{3,0}\\
a_{3,1}\\
a_{3,2}\\
\hline
a_{4,0}\\
a_{4,1}\\
a_{4,2}\\
a_{4,3}
\end{array}
\right)\,.
\end{aligned}
\end{small}
\ee
Finally, we can give an explict formula for the absolute value of determinant
of $\bM_{N}$.

\begin{proposition}
\label{prop:det.iotaN}
The number $D(N):=|\det(\bM_{N})|$ is given by
\be
D(N)
\= \prod_{n=1}^{N} D_{1}(n)D_{2}(n)\,,
\ee
where
\be
D_{1}(n) \= \prod_{m\ell=n}D_{m,\ell}^{\phi(m)}
\=
\prod_{m\ell=n}
m^{(2\ell-1)\phi(m)}(\ell-1)!^{\phi(m)}
\ee
with $D_{m,\ell}$ defined as in equation~\eqref{eq:poch.dml} and
\be
D_{2}(n) \= n^{\frac{n}{2}}\prod_{m|n}|\mathrm{disc}(\Phi_{m})|^{-\frac{1}{2}}
\= \prod_{m|n}\Big(\frac{n}{m}\prod_{p|m\,,\,p\text{ prime}}
p^{\frac{1}{p-1}}\Big)^{\frac{\phi(m)}{2}}\,.
\ee
\end{proposition}
\noindent Note that $D_2(n)$ is an integer since if $p|m$ then $\phi(m)$ is divisible
by $2(p-1)$ except for $m=p^k,2p^k$, where $\phi(p^k)=\phi(2p^k)$, which
can be seen to be integral for both even and odd~$n$.

We have tabulated the first few of these numbers below.
One can see their extremely rapid growth.
This illustrates how small
the Habiro ring is inside the set of functions near roots of unity with integral
Taylor expansions.
\begin{small}
\begin{center}
\def\arraystretch{1.5}
\begin{tabular}{|c||c|c|c|c|c|c|c|c|c|}
\hline
$N$ & $1$ & $2$ & $3$ & $4$ & $5$ & $6$ & $7$ & $8$ & $9$ \\ \hline\hline
$D_1(N)$ & $1$ & $2$ & $18$ & $768$ & $15000$ & $2.0\times10^{8}$
 & $8.5\times10^{7}$ & $6.5\times10^{13}$ & $5.1\times10^{15}$
\\ \hline
$D_2(N)$ & $1$ & $2$ & $3$ & $8$ & $5$ & $72$ & $7$ & $128$ & $81$ \\ \hline
$D(N)$ & $1$ & $4$ & $216$ & $1327104$ & $99532800000$
& $1.4\times10^{21}$ & $8.6\times10^{29}$ & $7.1\times10^{45}$
& $3.0\times10^{63}$\\
\hline
\end{tabular}
\end{center}
\end{small}
These numbers are not only very large, but also highly factored.
Here are two bigger examples illustrating both properties for $D(N)$:
\be
\begin{aligned}
D(25)
&\=
2^{1050}\cdot
3^{469}\cdot
5^{255}\cdot
7^{118}\cdot
11^{71}\cdot
13^{25}\cdot
17^{25}\cdot
19^{25}\cdot
23^{25}\;\approx\;2.35\times10^{1016}\,,\\
D(100)
&\=
2^{26102}\cdot
3^{12800}\cdot
5^{6404}\cdot
7^{3965}\cdot
11^{2158}\cdot
13^{1514}\cdot
17^{961}\cdot
19^{983}\cdot
23^{763}\cdot
29^{497}\\
&\qquad\times
31^{503}\cdot
37^{285}\cdot
41^{289}\cdot
43^{291}\cdot
47^{295}\cdot
53^{100}\cdot
59^{100}\cdot
61^{100}\cdot
67^{100}\cdot
71^{100}\\
&\qquad\times
73^{100}\cdot
79^{100}\cdot
83^{100}\cdot
89^{100}\cdot
97^{100}\;\approx\;9.33\times 10^{34419}\,.
\end{aligned}
\ee

\medskip

Since the nuymber D(N) is the cardinality of the cokernel of $\iota_\BZ^{N}$, we
see that begin in the image of $\iota^{N}$ is equivalent to a set of congruences
on the coordinates $\gamma_{m,\ell,j}$ to a total modulus $D(N)$.
In particular, we have a well-defined map
\be
(\iota^{N})^{-1}:\calP_{R}/\calP_R^{N}\;\to\;D(N)^{-1}R[q]/(q;q)_{N-1}R[q]\,,
\ee
and the condition for a given element $H\in\calP_{R}$ to be the image of $\iota$
is that
\be
(\iota^{N})^{-1}(H)\inn R[q]/(q;q)_{N-1}R[q]
\ee
for all $N\in\BZ_{>0}$. We can express this in terms of matrices as follows:

\begin{proposition}
\label{prop:imiota}
If $H\in\calP_{R}$ with coordinates $\gamma_{m,\ell,j}\in R$
($m,\ell\geq 1,0\leq j<\phi(m)$), then $H\in\iota(\HH_{R})$
if and only if
\be
\bM_{N}^{*} \,
\Big((\gamma_{m,\ell,j})_{\underset{0\leq j<\phi(m)}{m\ell<N}}\Big)
\;\equiv\; 0\pmod{D(N)}
\ee
for all $N\in\BZ_{>0}$, where $\bM_{N}^{*}:=D(N)\,\bM_{N}^{-1}\in M_{N(N-1)/2}(\BZ)$.
\end{proposition}

Note that~$D(N)$ and~$\bM_{N}^{*}$ have many common factors, which
simplifies~$\bM_{N}^{-1}$.
This means that in practice the congruences that one constructs can be simplified.
For example, using the previous coordinates in
equation~\eqref{eq:hdecomp} and equation~\eqref{eq:gamma}, we find that for~$N=5$
these conditions become
\be
\begin{small}
\begin{aligned}
\left(\begin{array}{c|cc|ccc|cccc}
1 &   &   &   &   &   &   &   &   &  \\
\hline
-1 & 1 & 2 &   &   &   &   &   &   &  \\
1 & -1 & 2 &   &   &   &   &   &   &  \\
\hline
-7 & -9 & -6 & 16 & -8 & 12 &   &   &   &  \\
-1 & 9 & 6 & -8 & 16 & 12 &   &   &   &  \\
17 & -9 & 18 & -8 & -8 & 12 &   &   &   &  \\
\hline
-13 & 9 & -13 & -32 & 0 & 0 & 36 & 0 & 9 & 12\\
-14 & -18 & -1 & 32 & -32 & 12 & 0 & 36 & -9 & 12\\
9 & 27 & 23 & 0 & 32 & 24 & -36 & 0 & 9 & 12\\
68 & -36 & 59 & -32 & 0 & 36 & 0 & -36 & -9 & 12
\end{array}
\right)
\left(\begin{array}{c}
\gamma_{1,1,0}\\
\hline
\gamma_{2,1,0}\\
\gamma_{1,2,0}\\
\hline
\gamma_{3,1,0}\\
\gamma_{3,1,1}\\
\gamma_{1,3,0}\\
\hline
\gamma_{4,1,0}\\
\gamma_{4,1,1}\\
\gamma_{2,2,0}\\
\gamma_{1,4,0}
\end{array}
\right)
\in
\left(\begin{array}{c}
R\\
\hline
4\,R\\
4\,R\\
\hline
72\,R\\
72\,R\\
72\,R\\
\hline
288\,R\\
288\,R\\
288\,R\\
288\,R
\end{array}
\right).
\end{aligned}
\end{small}
\ee
Summarising, we have found an explicit inductive set of congruences that
the coefficients of an element of $\calP_{R}$ must satisfy to be an element of
the image of $\iota$ of the naive Habiro ring of~$R$,
the first few of these congruences (corresponding to $N=4$) being
\be
\label{eq:cong1}
\begin{small}
\begin{aligned}
\-\gamma_{1,1,0} \+ \gamma_{2,1,0} \+ 2\gamma_{1,2,0}
&\cm 0\pmod{4}\,,\\
\gamma_{1,1,0} \- \gamma_{2,1,0} \+ 2\gamma_{1,2,0}
&\cm 0\pmod{4}\,,\\
\-7\gamma_{1,1,0} \- 9\gamma_{2,1,0} \- 6\gamma_{1,2,0} \+16 \gamma_{3,1,0}
\- 8\gamma_{3,1,1} \+ 12\gamma_{1,3,0}
&\cm 0\pmod{72}\,,\\
\-\gamma_{1,1,0} \+ 9\gamma_{2,1,0} \+ 6\gamma_{1,2,0} \- 8\gamma_{3,1,0}
\+ 16\gamma_{3,1,1} \+ 12\gamma_{1,3,0}
&\cm 0\pmod{72}\,,\\
17\gamma_{1,1,0} \- 9\gamma_{2,1,0} \+ 18\gamma_{1,2,0} \- 8\gamma_{3,1,0}
\- 8\gamma_{3,1,1} \+ 12\gamma_{1,3,0}
&\cm 0\pmod{72}\,.
\end{aligned}
\end{small}
\ee
These equations are not independent and we can reduce them to
\be
\label{eq:congrd1}
\begin{small}
\begin{aligned}
\gamma_{1,1,0} \- \gamma_{2,1,0} \+2\gamma_{1,2,0} \+4\gamma_{1,3,0}
&\cm 0\pmod{8}\,,\\
\gamma_{1,1,0} \- \gamma_{3,1,0} \- \gamma_{3,1,1} \-3\gamma_{1,3,0}
&\cm 0\pmod{9}\,,\\
\gamma_{1,2,0} \+\gamma_{3,1,1}
&\cm 0\pmod{3}\,.
\end{aligned}
\end{small}
\ee
Those familiar with the Habiro ring will easily
recognise these congruences as those discovered by Ohtsuki~\cite{Ohtsuki:poly}.
Indeed, we find that they are equivalent to the beginning of the expansions
\be
\gamma_{2,1,0}
\=
\sum_{k=0}^{\infty}\gamma_{1,k,0}\,2^{k}\in R_{2}^\wedge\,,
\quad\text{and}\quad
\gamma_{3,1,0}\+\z_3\,\gamma_{3,1,1}
\=
\sum_{k=0}^{\infty}\gamma_{1,k,0}\,(1\-\z_3)^{k}\in R_{3}^\wedge[\z_3]\,.
\ee

\subsection{The Habiro ring of a number field}

In the previous subsection, although everything was done over an arbitrary ground
ring $R$, the considerations were really based on the case when $R=\BZ$ (the
original Habiro ring), because the congruences needed to recognise whether an
element of $\calP_R$ comes from the naive Habiro ring of $R$ were independent of $R$
and came from this special case.
In this subsection, we will consider more general ground rings and describe the
relation between the naive Habiro ring and the Habiro ring studied in
Sections~\hbox{\ref{sec.intro}--\ref{sec.examples}}.
We will explain that the natural definition
involves twisting the compatibility conditions by suitable
Frobenius automorphisms similar to those that appeared in Definition~\ref{def.hab}
of Section~\ref{sub.hab}.
We will focus on ground
rings given by rings of integers (or $S$-integers, meaning that we adjoin the
reciprocal of a non-zero integer) of a number field, or various completions of this.

\medskip

Let $R$ be the ring $\calO_{\BK}[1/\Delta]$, where $\calO_{\BK}$ is the ring of
integers of a number field $\BK$ and
$\Delta$ any non-zero integer divisible by the discriminant of $\BK$.
Let $p$ be a prime that does not divide $\Delta$ and is therefore unramified, so
that
\be
p \= \mathfrak{p}_{1}\cdots\mathfrak{p}_{n}\,,
\ee
with the ideals $\mathfrak{p}_{i}$ distinct. This gives rise to an isomorphism
\be
R/pR \;\cong\; R/\mathfrak{p}_{1}R\times\cdots\times R/\mathfrak{p}_{n}R\,,
\ee
where each $R/\mathfrak{p}_{i}R\cong\BF_{q_i}$ with $q_i=|R/\mathfrak{p}_{i}R|$ is a
finite field with cyclic Galois group over $\BF_{p}$ generated by the
Frobenius automorphism
\be
x\mapsto x^p\,.
\ee
The Frobenius automorphisms on each factor $R/\mathfrak{p}_{i}R$ give rise to a
Frobenius
automorphism $\fr:R/pR\cong R/pR$ given by the same formula. Hensel's lemma states
that if a polynomial factors into irreducible polynomials over $\BZ/p\BZ$, then
this factorisation lifts to a unique factorisation over $\BZ/p^n\BZ$. Therefore,
if $\xi$ is a generator of the field $\BK$ with minimal polynomial $P(x)$, then
applying
Hensel's lemma to $P(x)$ lifts the Frobenius automorphism of $R/pR$ to an
automorphism of
$\fr:R/p^nR\cong R/p^nR$ for all $n\in\BZ_{>0}$. Hensel's lemma is completely
constructive, as was explained in Section~\ref{sub.padic.compute}. These Frobenius
lifts can be used to define automorphisms for any $R/MR$ for $M\in\BZ$ prime to
the discriminant of $\BK$. To do this, we write $M=\prod_{p\text{ prime}}p^{d_p}$.
The Chinese remainder theorem gives a canonical isomorphism
\be
R/MR \;\cong\; \prod_{p\text{ prime}}R/p^{d_p}R\,.
\ee
We can define $\fr:R/MR\to R/MR$ to be the Frobenius automorphism on the $p$-th
factor and the identity on the others.
Then for $m\in\BZ_{>0}$ (or even $m\in\BQ^{\times}$) we define
\be
\varphi_{m}\=\prod_{p}\fr^{v_{p}(m)}:R/MR\rightarrow R/MR\,,
\ee
where $\fr$ on $R/MR$ is simply the identity if $p$ does not divide $M$.
This can then be used to give Frobenius automorphisms on the completion of
the integers over all unramified primes at once. Consider the completion
\be
\widehat{R} \;:=\; \varprojlim_{M} R/MR \;\cong\; \prod_{p\text{ prime}}\!\!\!
R_p\,,
\ee
where $R_p:=\varprojlim_{n}R/p^nR\cong R\otimes\BZ_p$.
Notice that if $p|\Delta$ then $R_p=\{0\}$ is the zero ring. This completion then has
Frobenius automorphisms for each $m\in\BQ^{\times}$
\be
\varphi_{m}:\widehat{R}
\rightarrow
\widehat{R}\,.
\ee
We can lift $\varphi_{m}$ uniquely to an automorphism
$\varphi_{m}:\widehat{R}[\z_n]\llbracket
x\rrbracket\to\widehat{R}[\z_n]\llbracket x\rrbracket$ acting trivially on $\z_n$
and $x$. These then combine to an automorphism
$\varphi:\calP_{\widehat{R}}\to\calP_{\widehat{R}}$ defined by
\be
(\varphi\,f)_{m}
\=
\varphi_{m}(f_{m}(x))
\inn
\widehat{R}[\z_m]\llbracket x\rrbracket\,.
\ee
We will also use $\varphi$ to denote the composite map
$\calP_R\hookrightarrow\calP_{\widehat{R}}\overset{\varphi}{\rightarrow}
\calP_{\widehat{R}}$.
We can now restate the definition of the Habiro ring (Definition~\ref{def.hab}
of Section~\ref{sub.hab}) as follows.

\begin{proposition}
The Habiro $\calH_R$ of $R=\calO_{\BK}[1/\Delta]$ is given by
\be
  \calH_R
  \=\big\{H\in\calP_R\;\big|\;\varphi(H)\in\iota(\HH_{\widehat{R}})\big\}\,,
\ee
where $\HH_{\widehat{R}}$ is defined as in~\eqref{eq:Hnaive}. Equivalently,
$\calH_R=\calP_R\cap\calH_{\widehat{R}}$ with
$\calH_{\widehat{R}}=\varphi^{-1}(\iota(\HH_{\widehat{R}}))$.
\end{proposition}

\noindent Using Proposition~\ref{prop:imiota} of Section~\ref{sec:naive}, we can
reformulate this concretely as follows:

\begin{proposition}
\label{prop:hr.cong}
Let $H\in\calP_{R}$ with coordinates $\gamma_{m,\ell,j}\in R$
($m,\ell\geq 1,0\leq j<\phi(m)$). Then $H\in\calH_R$
if and only if
\be
\bM_{N}^{*} \,
\Big((\varphi_m\gamma_{m,\ell,j})_{\underset{0\leq j<\phi(m)}{m\ell<N}}\Big)
\;\equiv\; 0\pmod{D(N)}
\ee
for all $N\in\BZ_{>0}$, where $\bM_{N}^{*}$ is the matrix defined in
Proposition~\ref{prop:imiota} of Section~\ref{sec:naive}.
\end{proposition}

\noindent This gives an easy numerical check of whether a function near roots of
unity is an element of the Habiro ring of $R$.

\begin{remark}
In this section, we have only treated the Habiro ring $\calH_R$ and not its modules
$\calH_{R,\xi}$ indexed by $\xi\in K_3(\BK)$. Given that Theorem~\ref{thm.locals}
implies the completion $\calH_{\widehat{R},\xi}$ is a free $\calH_{\widehat{R}}$-module,
we can use similar descriptions to that of Proposition~\ref{prop:hr.cong} to describe
the modules by simply dividing elements by a chosen generator of
$\calH_{\widehat{R},\xi}$.
\end{remark}

\subsection{Examples}
\label{sub.examples2}

We now explore how the approach described in last two subsections works for some
different ground rings. We will illustrate how one can test whether a Galois
invariant function near roots of unity is actually an element of a Habiro ring.

\begin{example}
Recall the Kashaev invariant of the trefoil from equation~\eqref{eq:KZ}
(often called the Kontsevich-Zagier series). We can compute the first few terms
in the expansions at $\z_1,\z_2,\z_3,\z_4$ (with $q=\z_m(1-u)$)
\be
\begin{aligned}
F_{3_1,1}(u)
&\=
1 + u + 2u^2 + 5u^3+\mathrm{O}(u^4)\,,\\
F_{3_1,2}(u)
&\=
3 + 11u+\mathrm{O}(u^2)\,,\\
F_{3_1,3}(u)
&\=
5-\z_3+\mathrm{O}(u)\,,\\
F_{3_1,4}(u)
&\=
8-3\z_4+\mathrm{O}(u)\,.
\end{aligned}
\ee
These give rise to the vector
\be
\begin{small}
\begin{aligned}
&(&\gamma_{1,1,0}&\;|&
\gamma_{2,1,0}&\;&
\gamma_{1,2,0}&\;|&
\gamma_{3,1,0}&\;&
\gamma_{3,1,1}&\;&
\gamma_{1,3,0}&\;|&
\gamma_{4,1,0}&\;&
\gamma_{4,1,1}&\;&
\gamma_{2,2,0}&\;&
\gamma_{1,4,0}\;&)\,\phantom{,,_{2_{2_{2_{2_{2}}}}}}\\
\=\;&(&1 &\;|& 3 &\;& 1 &\;|& 5 &\;& -1 &\;& 2 &\;|
& 8 &\;& -3 &\;& 11 &\;& 5\;&)\,,\phantom{,_{2_{2_{2_{2_{2}}}}}}
\end{aligned}
\end{small}
\ee
which when multiplied by $\bM_5^{-1}$ gives the expected
\be
\begin{small}
\begin{aligned}
&(&a_{1,0}&\;|&
a_{2,0}&\;&
a_{2,1}&\;|&
a_{3,0}&\;&
a_{3,1}&\;&
a_{3,2}&\;|&
a_{4,0}&\;&
a_{4,1}&\;&
a_{4,2}&\;&
a_{4,3}&)\,\phantom{.,_{2_{2_{2_{2_{2}}}}}}\\
\=&(&1 &\;|& 1 &\;& 0 &\;|& 1 &\;& 0 &\;& 0 &\;|
& 1 &\;& 0 &\;& 0 &\;& 0&)\,.\phantom{,_{2_{2_{2_{2_{2}}}}}}
\end{aligned}
\end{small}
\ee
To see just how delicate the property of being in the Habiro ring is, notice
that simply changing the number $\gamma_{1,1,0}$ from $1$ to $2$ by would give
the sequence
\be
\begin{small}
\begin{aligned}
&\Big(\!\!\!\!&2 \;&\;\Big|& \frac{3}{4} &\;
& \frac{1}{4}\; &\;\Big|& \frac{65}{72} &\;& \frac{-1}{72}
&\;& \frac{17}{72}\; &\;\Big|& \frac{275}{288} &\;& \frac{-7}{144} &\;& \frac{1}{32}
&\;& \frac{17}{72}\;&\Big)\,,
\end{aligned}
\end{small}
\ee
where the integrality has been completely ruined.
\end{example}

\begin{example}
In this example we will see what the introduction of denominators does to the
Habiro ring. The Habiro function $F(q)\in\HH_\BQ$ whose image under $\iota$ is the
collection of (constant) power series
\be
F(\z_m(1-u)) \=
\bigg\{
\begin{array}{cl}
1 & \text{if }m\text{ is odd}\,,\\
0 & \text{if }m\text{ is even}\,
\end{array}
\ee
does not belong to $\HH_{\BZ}$, but is an element of $\HH_{\BZ[1/2]}$. Indeed, it
can be written
\be
F(q)\=1
+\frac{1}{4}(-1+q)\,(q;q)_1
+\frac{1}{8}(1-q+q^2)\,(q;q)_2
+\frac{1}{32}(-5+2q+q^2+4q^3)\,(q;q)_3+\cdots\,.
\ee
Similarly, the element $G(q)=F(q^2)-F(q)$, whose image under $\iota$ is the
collection of (constant) power series
\be
G(\z_m(1-u)) \=
\bigg\{
\begin{array}{cl}
1 & \text{if }m\;\equiv\;2\pmod{4}\\
0 & \text{otherwise}\,,
\end{array}
\ee
is also an element of $\HH_{\BZ[1/2]}$, and can be written
\be
G(q) \=
\frac{1}{4}(1-q)\,(q;q)_1
+\frac{1}{8}(-1+q-q^2)\,(q;q)_2
+\frac{1}{32}(1-2q+3q^2-4q^3)\,(q;q)_3+\cdots\,.
\ee
\end{example}

\begin{example}
Our final example is an element of the Habiro ring of the ring of integers,
with $1/23$ adjoined, of the field
$\mathbb{K}=\BQ(\alpha)/(\alpha^3-\alpha^2+1)$, the cubic field of discriminant
$-23$, associated to the Kashaev invariant of the knot $5_2$.
We define a collection of series $(\Psi_{5_2,m})_{m\geq1}$ using the residue formulas
given in Section~\ref{sub.sym}. Specifically, we define
\be
\Psi_{5_2,m}(u)
\;:=\;
\Res{w^m=z}J_{5_2}(1,w,\z_m(1-u))\frac{dw}{w}\,,
\ee
where
\be
J_{5_2}(t,w,q)
\;:=\;
\sum_{k=0}^{\infty}
\frac{q^{k(k+1)}\thinspace w^{2k}}{(qw;q)_{k}^3}t^k
\ee
and set $z=1-\alpha^2$ so that $z^2(1-z)^{-3}=1$.
The collection $\Psi_{5_2,m}$ is related to the symmetrised series
$\Phitof^{(5_2)}(q)\thinspace\Phitof^{(5_2)}(q^{-1})$ discussed in
Section~\ref{sub.52-237}. For example, equation~\eqref{eq:52sym} is equal to
$(3\alpha-2)\Psi_{5_2,1}(-x)$.
One can compute the first few coordinates via the expansions
\be
\begin{aligned}
\Psi_{5_2,1}(u)
&\=
\frac{1}{23}(-9\alpha^2 + 3\alpha + 2)
+ \frac{1}{23^4}(444\alpha^2 + 3417\alpha - 1287)u^2\\
&\qquad\qquad + \frac{1}{23^4}(444\alpha^2 + 3417\alpha - 1287)u^3
+ \mathrm{O}(u^4)\,,\\
\Psi_{5_2,2}(u)
&\=\frac{1}{23}(-43\alpha^2 + 22\alpha + 7) + \mathrm{O}(u^2)\,,\\
\Psi_{5_2,3}(u)
&\=\frac{1}{23}(-111\alpha^2 + 60\alpha + 17) + \mathrm{O}(u)\,,\\
\Psi_{5_2,4}(u)
&\=\frac{1}{23}(-242\alpha^2 + 119\alpha + 41) + \mathrm{O}(u)\,,
\end{aligned}
\ee
giving rise to the vector
\be
\left(
\begin{array}{c}
\gamma_{1,1,0}\\
\hline
\gamma_{2,1,0}\\
\gamma_{1,2,0}\\
\hline
\gamma_{3,1,0}\\
\gamma_{3,1,1}\\
\gamma_{1,3,0}\\
\hline
\gamma_{4,1,0}\\
\gamma_{4,1,1}\\
\gamma_{2,2,0}\\
\gamma_{1,4,0}\\
\end{array}
\right)
\=
\left(
\begin{array}{c}
  {}_{\phantom{2_{2_{2_{2_2}}}}}\tfrac{1}{23}(-9\alpha^2 + 3\alpha
  + 2)_{\phantom{2_{2_{2_{2_2}}}}} \\
 \hline
 {}^{\phantom{2^{2^2}}}\tfrac{1}{23}(-43\alpha^2 + 22\alpha
 + 7)^{\phantom{2^{2^2}}}\\
 0 \\
 \hline
 {}^{\phantom{2^{2^2}}}\tfrac{1}{23}(-111\alpha^2 + 60\alpha
 + 17)^{\phantom{2^{2^2}}}\\
 0\\
 {}_{\phantom{2_{2_{2_{2_2}}}}}\tfrac{1}{23^4}(444\alpha^2 + 3417\alpha
 - 1287)_{\phantom{2_{2_{2_{2_2}}}}} \\
 \hline
 {}^{\phantom{2^{2^2}}}\tfrac{1}{23}(-242\alpha^2
 + 119\alpha + 41)^{\phantom{2^{2^2}}}\\
 0\\
 0\\
 \tfrac{1}{23^4}(444\alpha^2 + 3417\alpha - 1287)
\end{array}
\right)\,.
\ee
(That the coordinates here are in $\BK$ as opposed to $\BK[\z_m]$ is because
$\phi(m)\leq 2$ for $m<5$ and $\Psi_{5_2}$ is a symmetrisation; in general, the
coordinates are in $\BK[\z_m+\z_m^{-1}]$.)
This vector when multiplied by $\bM_{5}^{-1}$ gives the vector
\be
\renewcommand\arraystretch{1.3}
\left(
\begin{array}{c}
  {}_{\phantom{2_{2_{2_{2_2}}}}}\tfrac{1}{23}(-9\alpha^2
  + 3\alpha + 2){}_{\phantom{2_{2_{2_{2_2}}}}}\\
 \hline
 \tfrac{1}{92}(-34\alpha^2 + 19\alpha + 5)\\
 {}_{\phantom{2_{2_{2_{2_2}}}}}\tfrac{1}{92}(34\alpha^2
 - 19\alpha - 5){}_{\phantom{2_{2_{2_{2_2}}}}}\\
 \hline
 \tfrac{1}{6716184}(-5376038\alpha^2 + 3018917\alpha + 785707)\\
 \tfrac{1}{6716184}(2070166\alpha^2 - 1142197\alpha - 309323)\\
 {}_{\phantom{2_{2_{2_{2_2}}}}}\tfrac{1}{6716184}(4552234\alpha^2
 - 2529235\alpha - 674333){}_{\phantom{2_{2_{2_{2_2}}}}} \\
 \hline
 \tfrac{1}{26864736}(-22020494\alpha^2 + 10246115\alpha + 3924793)\\
 \tfrac{1}{13432368}(-5376038\alpha^2 + 3018917\alpha + 785707)\\
 \tfrac{1}{8954912}(10100386\alpha^2 - 4938301\alpha - 1720695)\\
 \tfrac{1}{6716184}(4552234\alpha^2 - 2529235\alpha - 674333)
\end{array}
\right)\,,
\ee
which has denominator $26864736=2^5\cdot3\cdot23^4$. If we instead apply the
correct Frobenius automorphism and then multiply by $\bM_{5}^{-1}$ we find
\be
\begin{aligned}
\left(
\begin{array}{c}
\gamma_{1,1,0}\\
\hline
\varphi_2(\gamma_{2,1,0})\\
\gamma_{1,2,0}\\
\hline
\varphi_3(\gamma_{3,1,0})\\
\varphi_3(\gamma_{3,1,1})\\
\gamma_{1,3,0}\\
\hline
\varphi_4(\gamma_{4,1,0})\\
\varphi_4(\gamma_{4,1,1})\\
\varphi_2(\gamma_{2,2,0})\\
\gamma_{1,4,0}\\
\end{array}
\right)
&\=
\frac{1}{23^4}\left(
\begin{array}{c}
 225\alpha^2 + 213\alpha + 142\\
 241\alpha^2 + 113\alpha + 74\\
 0\\
 279\alpha^2 + 132\alpha + 151\\
 0\\
 156\alpha^2 + 249\alpha + 153\\
 161\alpha^2 + 115\alpha + 4\\
 0\\
 0\\
 156\alpha^2 + 249\alpha + 153
\end{array}
\right)\pmod{288}\,,\\
\left(
\begin{array}{c}
a_{1,0}\\
\hline
a_{2,0}\\
a_{2,1}\\
\hline
a_{3,0}\\
a_{3,1}\\
a_{3,2}\\
\hline
a_{4,0}\\
a_{4,1}\\
a_{4,2}\\
a_{4,3}\\
\end{array}
\right)
&\=
\frac{1}{23^4}
\left(
\begin{array}{c}
 225\alpha^2 + 213\alpha + 142\pmod{288}\\
 4\alpha^2 + 47\alpha + 55\pmod{72}\\
 68\alpha^2 + 25\alpha + 17\pmod{72}\\
 0\pmod{4}\\
 2\alpha^2 + 2\alpha \pmod{4}\\
 2\alpha^2 + 3\alpha + 1\pmod{4}\\
 0\pmod{1}\\
 0\pmod{1}\\
 0\pmod{1}\\
 0\pmod{1}
\end{array}
\right)\,,
\end{aligned}
\ee
where now everything is integral away from the prime $23$.
\end{example}

\subsection{Some wilder Habiro-like elements}

In this final subsection we describe some exotic functions that behave like elements
of Habiro rings that involve different fields at each root of unity. On the face of
it, it could seem that this would make patching together series via $p$-adic
re-expansion impossible.
The reason that it nevertheless works is that the fields in question become
canonically isomorphic after completion and the gluing of equation~\eqref{gluef}
can take place.
We will explain this in a simple example, but it could be generalised \hbox{easily} by
increasing the dimensions of the sums involved.
These elements will generalise the construction given in Section~\ref{sub.sym}.

\medskip

Take a polynomial $P(X)\in\BZ[X]$ and set
\be
f_{P}(w,t;q)
\=
\sum_{k=0}^{\infty}\bigg(\prod_{j=0}^{k-1}P(q^jw)\bigg)\,t^k\,.
\ee
We would like to take $t=1$. However, the function does not converge there, just
as in the case of admissible series studied in Section~\ref{sec.adm}.
As in Theorem~\ref{thm.FGI2} of Section~\ref{sub.results}, the $x$-expansion
of $f_{P}(w,t;\z+x)$ for a root of unity $\z$ has coefficients contained in rational
functions in $t$~and~$w$. Therefore, we can specialise to $t=1$.
For example, near $\z=1$ we find that
\be
\begin{aligned}
f_{P}(w,t;1+x)
&\=
\sum_{k=0}^{\infty}\bigg(\prod_{j=0}^{k-1}P(q^jw)\bigg)t^k\\
&\=
\sum_{k=0}^{\infty}t^k\Big(P(w)^k+\binom{k}{2}P(w)^{k-1}\,P'(w)x+\mathrm{O}(x^2)\Big)
\\
&\=
\frac{1}{1-tP(w)}\+\frac{tP'(w)}{(1-tP(w))^3}x\+\mathrm{O}(x^2)\,,
\end{aligned}
\ee
in which we can set $t=1$ to find
\be
f_{P}(w,1;1+x)
\=
\frac{1}{1-P(w)}\+\frac{P'(w)}{(1-P(w))^3}x\+\mathrm{O}(x^2)\,.
\ee
Similarly, at any primitive $m$-th root of unity $\z_m$,
one can show that there are polynomials $A_{k,m}(w)\in\BZ[\z_m,w]$ so that
\be
f_{P}(w,1;\z_m+x)
\=
\sum_{k=0}^{\infty}\frac{A_{k,m}(w)}{(1-P_m(w^m))^{2k+1}}x^k\,,
\ee
where
\be
P_m(w)
\=
\prod_{j=1}^{m}P(\z_m^jw^{1/m})
\=
\prod_{u^m=w}P(u)\,.
\ee
Notice that if (and only if) $P(w)$ has the special form
\be
P(w)\=w^a(1-w)^b\,,
\ee
for some $a,b\in\BZ$, then
\be
P_{m}(w)\=P(w)\,,
\ee
for all $m$, giving functions similar to those considered in Section~\ref{sub.sym}.
Going back to the general case, if we expand $f_P(w,1;\z_m+x)$ near
$w=\z_m^j\a_m^{1/m}$, where $P_m(\a_m)=1$, we get a Laurent series.
Therefore, we can take the residue of these Laurent expansions, just as was done
in Section~\ref{sub.sym}, to define series
\be
f_{P,m}(x)
\=
\sum_{z^m=\a_m}\Res{w=z}
f_{P}(w,1;\z+x)
\frac{dw}{w}
\inn\mathcal{O}_{P,m}[\z]\llbracket x\rrbracket\,,
\ee
where $\mathcal{O}_{P,m}=\mathcal{O}_{\BQ(\a_m)}\big[\frac{1}{\mathrm{disc}(1-P_m)}\big]$.
The most important property is that for $p$ not dividing $\Delta$,
sending $\alpha_{pm}$ to $\alpha_m$ modulo $p$ gives a canonical isomorphism
\be
\label{eq:weird.elts.iso}
\calO_{P,pm}/p\calO_{P,pm}
\;\cong\;
\calO_{P,m}/p\calO_{P,m}\,.
\ee
Hensel's lemma implies that there exists a unique lift to an isomorphism from
$\calO_{P,pm}/p^n\calO_{P,pm}$ to $\calO_{P,m}/p^n\calO_{P,m}$ for all~$n$, which is
exactly what we need for the gluing of equation~\eqref{gluef}.

\begin{example}
Take $P(X)=-X^3+8$. Then
\be
P_{m}(X)
\=
\bigg\{
\begin{array}{cl}
-X^{3m}+8^m & \text{if }m\equiv 1,2\pmod{3}\,,\\
-X^{3m}+3\cdot2^mX^{2m}-3\cdot4^mX^{m}+8^m & \text{if }m\equiv 0\pmod{3}\,.
\end{array}
\ee
Therefore, we see that $\mathcal{O}_{P,m}$ is a cubic ring generated by $\xi$ satisfying
\be
\begin{array}{cl}
\xi^{3}-(8^m-1)\=0 & \text{if }m\equiv 1,2\pmod{3}\,,\\
\xi^{3}-3\cdot2^m\xi^{2}+3\cdot4^m\xi-(8^m-1)\=0 & \text{if }m\equiv 0\pmod{3}\,.
\end{array}
\ee
Note that
\be
\xi^{3}-3\cdot2^m\xi^{2}+3\cdot4^m\xi-(8^m-1)
\=
(\xi + 1-2^m)(\xi^2 - (2^{m+1} - 1)\xi + 4^{m} + 2^{m} + 1)\,.
\ee
This makes obvious the isomorphism of equation~\eqref{eq:weird.elts.iso}, since for
example $\xi^3-(8^{pm}-1)\equiv\xi^3-(8^{m}-1)$ modulo $p$. In this case we can
compute and surprisingly (though for the moment only conjecturally) find that
\be
f_{P,m}(x)
\=
\bigg\{
\begin{array}{cl}
\frac{1}{21} & \text{if }m\equiv 1,2\pmod{3}\,,\\
0 & \text{if }m\equiv 0\pmod{3}\,,
\end{array}
\ee
at least for small values of $m$.
If this is true, then $7f_{P}$ belongs to $\HH_{\BZ[1/3]}$ and its image under
$\iota$ to $\calH_{\BZ[1/3]}$.
\end{example}

\begin{example}
In the previous case, our ``exotic'' construction gave (at least conjecturally)
an element of an ordinary Habiro ring.
We take another example to see some more interesting behaviour. Choose
$P(X)=-X^3 + 7X^2 - 14$. Then we find that
\be
\begin{small}
\begin{aligned}
f_{P,1}(x)
&\=
\frac{1}{14505}(21\xi_1^2 - 49\xi_1 - 551)\\
&\quad+\frac{1}{983690367661125}(7160238883\xi_1^2 - 22515890236\xi_1
- 64413491205)x^2+\cdots\,,\\
f_{P,2}(x)
&\=
\frac{1}{15179085}(11711\xi_2^2 - 534814\xi_2 + 555516)+\cdots\,,
\end{aligned}
\end{small}
\ee
where $\xi_1$ and $\xi_2$ are cubic irrationalities defined by 
\be
\xi_1^3 - 7\xi_1^2 + 15\=0\,,\qquad
\xi_2^3 - 49\xi_2^2 + 196\xi_2 - 195\=0
\ee
Again we expect Ohtsuki congruences modulo $p^n$ between the series around roots
of unity whose orders are related by multiplication by the prime $p$. 
In this example we can check the first Ohtsuki congruence between the expansions
around $q=1$ and $-1$. This will be a congruence modulo $2$.
Specifically, this first Ohtsuki congruence is equivalent to the equality
\be
f_{P,1}(-2)
\equiv
\xi^2 + \xi + 1
\equiv
f_{P,2}(0)^{2}
\equiv
\varphi_2(f_{P,2}(0))\pmod{2}\,,
\ee
where $\xi^3 +\xi^2 + 1\equiv 0$ modulo $2$.
\end{example}

To prove identities of this kind one can use the simple $q$-holonomic methods
used to prove Theorem~\ref{thm.2} of Section~\ref{sub.results}.

%%%%%%%%%%%%%%%%%%%%%%%%%%%%%%%%%%%%%%%%%%%%%%%%%%%%%%%%%%%%%%%%%%%%%%%%%%%% 
%%%%%%%%%%%%%%%%%%%%%%%%%%%%%%%%%%%%%%%%%%%%%%%%%%%%%%%%%%%%%%%%%%%%%%%%%%%%

\section*{Acknowledgements} 

The authors wish to thank Frank Calegari, Maxim Kontsevich, Yan Soibelman,
Matthias Storzer and Ferdinand Wagner for enlightening conversations.

P.S. was supported by a Leibniz Prize, and through the
Hausdorff Center for Mathematics (grant no. EXC-2047/1–390685813) at the MPIM,
Bonn. C.W. wants to thank the International Mathematics Center at SUSTech University,
Shenzhen for their hospitality. C.W. has been supported by the Huawei Young Talents
Program at Institut des Hautes Études Scientifiques, France.

\vspace{-0.05in}

\bibliographystyle{plain}
\bibliography{biblio}

\begin{thebibliography}{10}

\bibitem{AGL}
Ni~An, Stavros Garoufalidis, and Shana~Yunsheng Li.
\newblock Algebraic aspects of holomorphic quantum modular forms.
\newblock Preprint 2023,
  \href{https://arxiv.org/abs/2403.02880}{arXiv:2403.02880} to appear in Res.
  Math. Sciences.

\bibitem{AK-review}
J\o rgen~Ellegaard Andersen and Rinat Kashaev.
\newblock The {T}eichm\"{u}ller {TQFT}.
\newblock In {\em Proceedings of the {I}nternational {C}ongress of
  {M}athematicians---{R}io de {J}aneiro 2018. {V}ol. {III}. {I}nvited
  lectures}, pages 2541--2565. World Sci. Publ., Hackensack, NJ, 2018.

\bibitem{AK}
J{\o}rgen~Ellegaard Andersen and Rinat Kashaev.
\newblock A {TQFT} from {Q}uantum {T}eichm\"uller theory.
\newblock {\em Comm. Math. Phys.}, 330(3):887--934, 2014.

\bibitem{AarhusII}
Dror Bar-Natan, Stavros Garoufalidis, Lev Rozansky, and Dylan Thurston.
\newblock The {Å}rhus integral of rational homology 3-spheres. {II}.
  {I}nvariance and universality.
\newblock {\em Selecta Math. (N.S.)}, 8(3):341--371, 2002.

\bibitem{Bender}
Carl Bender and Steven Orszag.
\newblock {\em Advanced mathematical methods for scientists and engineers}.
\newblock International Series in Pure and Applied Mathematics. McGraw-Hill
  Book Co., New York, 1978.

\bibitem{finli}
Amnon Besser.
\newblock Finite and {$p$}-adic polylogarithms.
\newblock {\em Compositio Math.}, 130(2):215--223, 2002.

\bibitem{BJ:syntomic}
Amnon Besser and Rob de~Jeu.
\newblock The syntomic regulator for the {$K$}-theory of fields.
\newblock {\em Ann. Sci. \'{E}cole Norm. Sup. (4)}, 36(6):867--924 (2004),
  2003.

\bibitem{Lip}
Amnon Besser and Rob de~Jeu.
\newblock {${\rm Li}^{(p)}$}-service? {A}n algorithm for computing {$p$}-adic
  polylogarithms.
\newblock {\em Math. Comp.}, 77(262):1105--1134, 2008.

\bibitem{Bloch}
Spencer Bloch.
\newblock {\em Higher regulators, algebraic {$K$}-theory, and zeta functions of
  elliptic curves}, volume~11 of {\em CRM Monograph Series}.
\newblock American Mathematical Society, Providence, RI, 2000.

\bibitem{CGZ}
Frank Calegari, Stavros Garoufalidis, and Don Zagier.
\newblock Bloch groups, algebraic {$K$}-theory, units, and {N}ahm's conjecture.
\newblock {\em Ann. Sci. \'{E}c. Norm. Sup\'{e}r. (4)}, 56(2):383--426, 2023.

\bibitem{Coleman}
Robert Coleman.
\newblock Dilogarithms, regulators and {$p$}-adic {$L$}-functions.
\newblock {\em Invent. Math.}, 69(2):171--208, 1982.

\bibitem{dimofte-rev}
Tudor Dimofte.
\newblock Perturbative and nonperturbative aspects of complex {C}hern-{S}imons
  theory.
\newblock {\em J. Phys. A}, 50(44):443009, 25, 2017.

\bibitem{DG}
Tudor Dimofte and Stavros Garoufalidis.
\newblock The quantum content of the gluing equations.
\newblock {\em Geom. Topol.}, 17(3):1253--1315, 2013.

\bibitem{DG2}
Tudor Dimofte and Stavros Garoufalidis.
\newblock Quantum modularity and complex {C}hern-{S}imons theory.
\newblock {\em Commun. Number Theory Phys.}, 12(1):1--52, 2018.

\bibitem{DGLZ}
Tudor Dimofte, Sergei Gukov, Jonatan Lenells, and Don Zagier.
\newblock Exact results for perturbative {C}hern-{S}imons theory with complex
  gauge group.
\newblock {\em Commun. Number Theory Phys.}, 3(2):363--443, 2009.

\bibitem{Efimov}
Alexander Efimov.
\newblock Cohomological {H}all algebra of a symmetric quiver.
\newblock {\em Compos. Math.}, 148(4):1133--1146, 2012.

\bibitem{EVG}
Philippe Elbaz-Vincent and Herbert Gangl.
\newblock On poly(ana)logs. {I}.
\newblock {\em Compositio Math.}, 130(2):161--210, 2002.

\bibitem{Faddeev}
Ludwig Faddeev.
\newblock Discrete {H}eisenberg-{W}eyl group and modular group.
\newblock {\em Lett. Math. Phys.}, 34(3):249--254, 1995.

\bibitem{GG:qdiff}
Stavros Garoufalidis and Jeffrey Geronimo.
\newblock Asymptotics of {$q$}-difference equations.
\newblock In {\em Primes and knots}, volume 416 of {\em Contemp. Math.}, pages
  83--114. Amer. Math. Soc., Providence, RI, 2006.

\bibitem{GK:evaluation}
Stavros Garoufalidis and Rinat Kashaev.
\newblock Evaluation of state integrals at rational points.
\newblock {\em Commun. Number Theory Phys.}, 9(3):549--582, 2015.

\bibitem{GL:skein}
Stavros Garoufalidis and Thang~T.Q. L\^{e}.
\newblock From 3-dimensional skein theory to functions near {$\BQ$}.
\newblock Preprint 2023,
  \href{https://arxiv.org/abs/2307.09135}{arXiv:2307.09135}.

\bibitem{GSW}
Stavros Garoufalidis, Matthias Storzer, and Campbell Wheeler.
\newblock Perturbative invariants of cusped hyperbolic 3-manifolds.
\newblock Preprint 2023,
  \href{https://arxiv.org/abs/2305.14884}{arXiv:2305.14884}.

\bibitem{GZ:asymptotics}
Stavros Garoufalidis and Don Zagier.
\newblock Asymptotics of {N}ahm sums at roots of unity.
\newblock {\em Ramanujan J.}, 55(1):219--238, 2021.

\bibitem{GZ:nahm}
Stavros Garoufalidis and Don Zagier.
\newblock Asymptotics of {N}ahm sums at roots of unity.
\newblock {\em Ramanujan J.}, 55(1):219--238, 2021.

\bibitem{GZ:qseries}
Stavros Garoufalidis and Don Zagier.
\newblock Knots and their related {$q$}-series.
\newblock {\em SIGMA Symmetry Integrability Geom. Methods Appl.}, 19:Paper No.
  082, 2023.

\bibitem{GZ:kashaev}
Stavros Garoufalidis and Don Zagier.
\newblock Knots, perturbative series and quantum modularity.
\newblock {\em SIGMA Symmetry Integrability Geom. Methods Appl.}, 20:Paper No.
  055, 2024.

\bibitem{Habiro:completion}
Kazuo Habiro.
\newblock Cyclotomic completions of polynomial rings.
\newblock {\em Publ. Res. Inst. Math. Sci.}, 40(4):1127--1146, 2004.

\bibitem{Habiro:WRT}
Kazuo Habiro.
\newblock A unified {W}itten-{R}eshetikhin-{T}uraev invariant for integral
  homology spheres.
\newblock {\em Invent. Math.}, 171(1):1--81, 2008.

\bibitem{Hikami}
Kazuhiro Hikami.
\newblock Generalized volume conjecture and the {$A$}-polynomials: the
  {N}eumann-{Z}agier potential function as a classical limit of the partition
  function.
\newblock {\em J. Geom. Phys.}, 57(9):1895--1940, 2007.

\bibitem{Huber}
Annette Huber and Guido Kings.
\newblock A {$p$}-adic analogue of the {B}orel regulator and the {B}loch-{K}ato
  exponential map.
\newblock {\em J. Inst. Math. Jussieu}, 10(1):149--190, 2011.

\bibitem{Hutchinson}
Kevin Hutchinson.
\newblock The {C}hern class for {$K_{3}$} and the cyclic quantum dilogarithm.
\newblock {\em J. Algebra}, 649:433--443, 2024.

\bibitem{HuL}
Vu~Huynh and Thang~T.Q. L\^{e}.
\newblock The colored {J}ones polynomial and the {K}ashaev invariant.
\newblock {\em Fundam. Prikl. Mat.}, 11(5):57--78, 2005.

\bibitem{K94}
Rinat Kashaev.
\newblock Quantum dilogarithm as a {$6j$}-symbol.
\newblock {\em Modern Phys. Lett. A}, 9(40):3757--3768, 1994.

\bibitem{KMS}
Rinat Kashaev, Vladimir Mangazeev, and Yuri Stroganov.
\newblock Star-square and tetrahedron equations in the {B}axter-{B}azhanov
  model.
\newblock {\em Internat. J. Modern Phys. A}, 8(8):1399--1409, 1993.

\bibitem{Koblitz}
Neal Koblitz.
\newblock {\em {$p$}-adic numbers, {$p$}-adic analysis, and zeta-functions}.
\newblock Graduate Texts in Mathematics, Vol. 58. Springer-Verlag, New
  York-Heidelberg, 1977.

\bibitem{KS:cohomological}
Maxim Kontsevich and Yan Soibelman.
\newblock Cohomological {H}all algebra, exponential {H}odge structures and
  motivic {D}onaldson-{T}homas invariants.
\newblock {\em Commun. Number Theory Phys.}, 5(2):231--352, 2011.

\bibitem{Nahm}
Werner Nahm.
\newblock Conformal field theory and torsion elements of the {B}loch group.
\newblock In {\em Frontiers in number theory, physics, and geometry. {II}},
  pages 67--132. Springer, Berlin, 2007.

\bibitem{NZ}
Walter Neumann and Don Zagier.
\newblock Volumes of hyperbolic three-manifolds.
\newblock {\em Topology}, 24(3):307--332, 1985.

\bibitem{Ohtsuki:poly}
Tomotada Ohtsuki.
\newblock A polynomial invariant of integral homology {$3$}-spheres.
\newblock {\em Math. Proc. Cambridge Philos. Soc.}, 117(1):83--112, 1995.

\bibitem{Rademacher}
Hans Rademacher and Emil Grosswald.
\newblock {\em Dedekind sums}.
\newblock The Carus Mathematical Monographs, No. 16. Mathematical Association
  of America, Washington, DC, 1972.

\bibitem{RT:ribbon}
Nikolai Reshetikhin and Vladimir Turaev.
\newblock Ribbon graphs and their invariants derived from quantum groups.
\newblock {\em Comm. Math. Phys.}, 127(1):1--26, 1990.

\bibitem{RV:Apoly}
Fernando Rodriguez~Villegas.
\newblock A refinement of the {A}-polynomial of quivers.
\newblock Preprint 2011,
  \href{https://arxiv.org/abs/1102.5308}{arXiv:1102.5308}.

\bibitem{Scholze:canonical}
Peter Scholze.
\newblock Canonical {$q$}-deformations in arithmetic geometry.
\newblock {\em Ann. Fac. Sci. Toulouse Math. (6)}, 26(5):1163--1192, 2017.

\bibitem{Tu:book}
Vladimir Turaev.
\newblock {\em Quantum invariants of knots and 3-manifolds}, volume~18 of {\em
  de Gruyter Studies in Mathematics}.
\newblock Walter de Gruyter \& Co., Berlin, 1994.

\bibitem{Weibel}
Charles Weibel.
\newblock Algebraic {$K$}-theory of rings of integers in local and global
  fields.
\newblock In {\em Handbook of {$K$}-theory. {V}ol. 1, 2}, pages 139--190.
  Springer, Berlin, 2005.

\bibitem{Weibel:Kbook}
Charles Weibel.
\newblock {\em The {$K$}-book}, volume 145 of {\em Graduate Studies in
  Mathematics}.
\newblock American Mathematical Society, Providence, RI, 2013.
\newblock An introduction to algebraic $K$-theory.

\bibitem{Witten:jones}
Edward Witten.
\newblock Quantum field theory and the {J}ones polynomial.
\newblock {\em Comm. Math. Phys.}, 121(3):351--399, 1989.

\bibitem{Witten:complexCS}
Edward Witten.
\newblock Quantization of {C}hern-{S}imons gauge theory with complex gauge
  group.
\newblock {\em Comm. Math. Phys.}, 137(1):29--66, 1991.

\bibitem{Za:strange}
Don Zagier.
\newblock Vassiliev invariants and a strange identity related to the {D}edekind
  eta-function.
\newblock {\em Topology}, 40(5):945--960, 2001.

\bibitem{Zagier:dilog}
Don Zagier.
\newblock The dilogarithm function.
\newblock In {\em Frontiers in number theory, physics, and geometry. {II}},
  pages 3--65. Springer, Berlin, 2007.

\bibitem{Z:diffeq}
Don Zagier.
\newblock The arithmetic and topology of differential equations.
\newblock In {\em European {C}ongress of {M}athematics}, pages 717--776. Eur.
  Math. Soc., Z\"urich, 2018.

\bibitem{Zickert:bloch}
Christian Zickert.
\newblock The extended {B}loch group and algebraic {$K$}-theory.
\newblock {\em J. Reine Angew. Math.}, 704:21--54, 2015.

\end{thebibliography}
\end{document}